\documentclass[preprint,3p,times]{elsarticle}
\usepackage[toc,page,title,titletoc,header]{appendix}
\usepackage{amsmath}
\usepackage{amssymb}
\usepackage{amsthm}
\usepackage{enumerate}
\usepackage{enumitem}
\usepackage{mathrsfs}
\usepackage{graphicx}
\usepackage{epsfig}
\usepackage{tikz,pgfplots,tikz-3dplot}
\usepackage{epstopdf}
\usepackage{microtype}
\usetikzlibrary{fit}
\biboptions{sort&compress}

\newtheorem{thm}{Theorem}[section]
\newtheorem{lem}[thm]{Lemma}
\newtheorem{rem}{Remark}
\renewcommand{}

\newcommand{\bI}{\mathbb{I}}

\newcommand{\nn}{\nonumber}

\newcommand{\mS}{\mathcal{S}}

\def\epsilon{\varepsilon} 
\newcommand{\mat}[1]{\boldsymbol{#1}}

\allowdisplaybreaks
\pgfplotsset{compat=1.18}

\begin{document}
\begin{frontmatter}
\title{
Solid-state dewetting of axisymmetric thin film on axisymmetric curved-surface substrates: modeling and simulation}
\author[1]{Zhenghua Duan}
\author[1]{Meng Li*}
\address[1]{School of Mathematics and Statistics, Zhengzhou University,
Zhengzhou 450001, China.}
\ead{This author's research was supported by National Natural Science Foundation of China (No. 11801527,U23A2065), the China Postdoctoral Science Foundation (No. 2023T160589).
Corresponding author: limeng@zzu.edu.cn. }
\author[1]{Chunjie Zhou}

\begin{abstract}
 In this work, we consider the solid-state dewetting of an axisymmetric thin film on a curved-surface substrate, with the assumption that the substrate morphology is also axisymmetric. 
Under the assumptions of axisymmetry, the surface evolution problem on a curved-surface substrate can be reduced to a curve evolution problem on a static curved substrate. 
Based on the thermodynamic variation of the anisotropic surface energy, we thoroughly derive a sharp-interface model that is governed by anisotropic surface diffusion, along with appropriate boundary conditions.
The continuum system satisfies the laws of energy decay and volume conservation, which motivates the design of a structure-preserving numerical algorithm for simulating the mathematical model.
By introducing a symmetrized surface energy matrix, we derive a novel symmetrized variational formulation. Then, by carefully discretizing the boundary terms of the variational formulation, we establish an unconditionally energy-stable parametric finite element approximation of the axisymmetric system. 
By applying an ingenious correction method, we further develop another structure-preserving method that can preserve both the energy stability and volume conservation properties. 
Finally, we present extensive numerical examples to demonstrate the convergence and structure-preserving properties of our proposed numerical scheme. 
Additionally, several interesting phenomena are explored, including the migration of ‘small’ particles on a curved-surface substrate generated by curves with positive or negative curvature, pinch-off events, and edge retraction.
\end{abstract}



\end{frontmatter}
\section{Introduction}\label{sec1}
Capillarity plays a crucial role in the interactions between liquid molecules, driving both wetting and dewetting processes and guiding the system toward a state of lower surface energy.
At the micro- and nano-scales, solid materials also exhibit pronounced capillary effects, which significantly influence surface behavior and material dynamics. 
Solid thin films on rigid substrates display significant instability even at temperatures well below their melting points.
The continuous thin films undergo dewetting or agglomeration, causing complex morphological changes that result in the formation of small particles on the substrate \cite{Jiran90,Jiran92,Ye10a,Ye10b,Ye11a,Ye11b}.
During this process, thin films remain in the solid state, which is why it is referred to as solid-state dewetting (SSD) \cite{Thompson12}. Unlike liquid-state dewetting, SSD is strongly affected by surface energy anisotropy \cite{Ye11b,Kim13,Zucker13,Zucker16}, and the mass transport involved is governed primarily by surface diffusion \cite{Danielson06,Mullins57}.

The phenomenon of SSD has garnered increasing attention in recent decades due to its widespread applications in optical and magnetic devices, thin films, sensors, and catalyst formation \cite{Armelao06,Bollani19,Schmidt09}. To better understand its underlying mechanisms, various models have been developed, drawing insights from both experimental studies \cite{Giermann05,Jiran90,Jiran92,Kovalenko17,Naffouti16,Wang11} and mathematical modeling approaches \cite{dornel2006surface,Jiang12,Jiang18a,Jiang19a,Jiang19c,Srolovitz86,Zucker13}.
In the evolution of solid thin films, two crucial dynamic processes are surface diffusion and contact line migration, where the thin film/vapor interface meets the substrate. 
In particular, at the contact line, the equilibrium contact angle (i.e., Young's law \cite{Young1805}) emerges from the balance of forces, or line tensions, acting along the substrate.

Srolovitz and Safran \cite{Srolovitz86a} first proposed a sharp-interface model and applied it to study the growth of the hole under three assumptions: isotropic surface energy, small slope profile, and cylindrical symmetry. 
Based on the sharp-interface model, Wong et al. \cite{Wong00} investigate the tendency of a semi-infinite, uniform film on a substrate to undergo two-dimensional retraction from the edge, with the aim of reducing the system's surface energy.
Dornel et al. \cite{Dornel06} constructed another numerical scheme to study the pinch-off phenomenon of two-dimensional island films with high-aspect ratios during SSD. 
Jiang et al. \cite{Jiang12} addressed a similar problem using the phase field method, which not only captures topological events during the evolution process but is also applicable in any dimension.
Jiang et al. \cite{Jiang16} proposed a sharp-interface continuum model based on a thermodynamic variational approach to study the strong anisotropic effects on SSD, including contact line dynamics.  
They demonstrated that for strong surface energy anisotropy, multiple equilibrium shapes can arise, which cannot be described by the traditional Winterbottom construction, and show that these shapes are dynamically accessible through their evolution model. 

Most theoretical studies on SSD focus primarily on flat substrates, with relatively little attention given to topologically patterned substrates.
Jiang et al. \cite{Jiang18a} derived a mathematical sharp-interface model to simulate SSD in thin films on rigid substrates, demonstrating that the migration velocity of a small solid particle is influenced by both the substrate curvature and the particle area. 
Zhao et al. \cite{ZHAO2024120407} applied the Onsager principle to develop a reduced-order model for the motion of solid particles on a curved substrate. They highlighted the relationship between the system's free energy and substrate curvature, and noted that the dissipation function is governed by the normal velocity on the two-dimensional surface.
Bao et al. \cite{bao2024} introduced an arclength parameterization for the substrate curve to consider a two-dimensional sharp-interface model for SSD of thin films on curved substrates. They also simulated the edge retraction of a semi-infinite step film and the pinch-off phenomenon of a long film.

However, we observe that the existing research primarily focuses on two-dimensional scenarios. To the best of my knowledge, there is currently limited research on three-dimensional SSD on curved-surface substrate.
In particular, both three-dimensional thin films and curved-surface substrates often exhibit rotational symmetry, enabling us to simplify the complex three-dimensional SSD into a curve-based system on a curved substrate. This approach has already been applied to the SSD on flat substrates. For instance, Zhao \cite{Zhao19} used thermodynamic variations to derive a sharp-interface model for SSD in thin films on flat substrates, assuming that the film morphology is axisymmetric. Under this assumption, the problem is simplified by reducing its dimensionality, making it easier for analysis and simulation.
In \cite{li2024structure}, we designed two novel structure-preserving parametric finite element approximations, along with a mesh-improving method, for the SSD with axisymmetric geometry.
However, the evolution of thin films on three-dimensional curved-surface substrates remains an area that requires further theoretical and numerical investigations.

In the current work, assuming that the film morphology and the curved-surface substrate are both axisymmetric, we first establish a sharp-interface model for SSD by using the thermodynamic variation. 
After formulating the model, we further develop a structure-preserving parametric finite element method (PFEM) for its numerical solution. 
PFEMs are widely regarded as highly effective for solving geometric PDEs, offering significant advantages over other methods, including more relaxed time step constraints and better mesh distribution.
Applications of PFEMs include isotropic cases \cite{Bansch05, bao2021structure, Barrett07, Barrett08JCP, kovacs2021convergent, Zhao20} and anisotropic cases \cite{Bao17, baojcm2022, Barrett07Ani, Barrett08Ani, Hausser07, li2021energy, Zhao19b, Barrett20}.
A particular method, introduced by Barrett, Garcke, and Nürnberg in \cite{Barrett08JCP}, has proven to be highly efficient and significant. 
By incorporating tangential degrees of freedom, the 'BGN' method guarantees excellent mesh quality, eliminating the need for the mesh regularization/smoothing procedures typically required in other methods (see \cite{Barrett20} and the references therein for a detailed introduction).
Bao and Zhao in \cite{bao2021structure} employed a time-weighted discrete normal to design a mass-conservative PFEM for the isotropic surface diffusion flow.
By introducing novel surface energy matrices, Bao et al. developed several energy-stable PFEMs for anisotropic surface diffusion flows \cite{li2021energy, bao2023symmetrized, bao2023symmetrized1,bao2023unified}. 
The structure-preserving methods were further utilized for simulating 
the SSD \cite{li2023symmetrized,li2024structure,ZHANG2025113605}.
 Recently, Bao et al. \cite{bao2024} proposed a novel area-conservative and energy-stable parametric finite element approximation for the two-dimensional sharp-interface model for SSD with thin film on a curved substrate.
In this paper, we build upon the aforementioned work to develop novel energy-stable and volume-conserving algorithms for simulating the anisotropic SSD on curved-surface substrates with axisymmetric geometry.

In this paper, we assume that SSD is only caused by surface diffusion, the effect of elasticity can be ignored, and no relevant chemical reactions occur during the evolution process.
The objectives of this work include:  
(1) through thermodynamic variations, we derive a sharp-interface model for three-dimensional SSD of axisymmetric thin films evolving on axisymmetric curved-surface substrates, encompassing both weakly and strongly anisotropic cases;
(2) by carefully discretizing the boundary terms, a novel energy-stable parametric finite element approximation is designed to numerically solve the sharp-interface model proposed above; 
(3) the error in the enclosed volume, determined by the substrate profile, is estimated, and then, by using a correction method, a structure-preserving method is developed that can preserve both volume conservation and energy stability;
(4) several interesting experiments are presented, including the evolution of thin films on three-dimensional axisymmetric curved-surface substrates generated by curves with positive and negative curvatures (see Figure \ref{fig:1}), pinch-off events of a long film and edge retraction of a semi-infinite step film.
\begin{figure}
    \centering
    \includegraphics[width=0.24\linewidth]{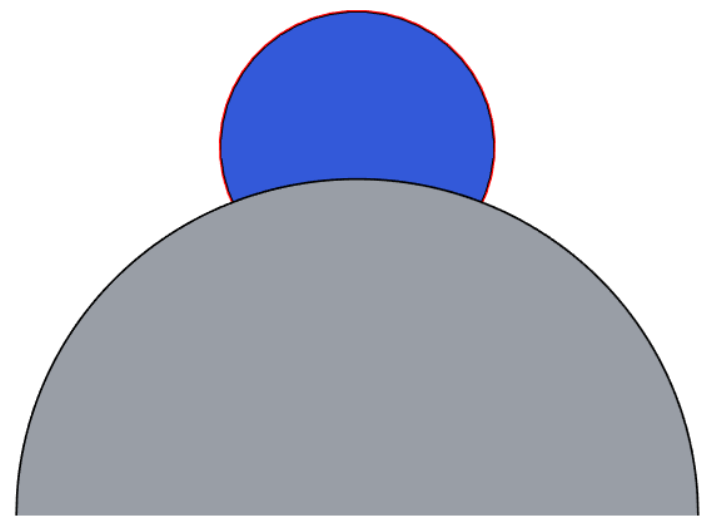}
    \hspace{10mm}
    \includegraphics[width=0.24\linewidth]{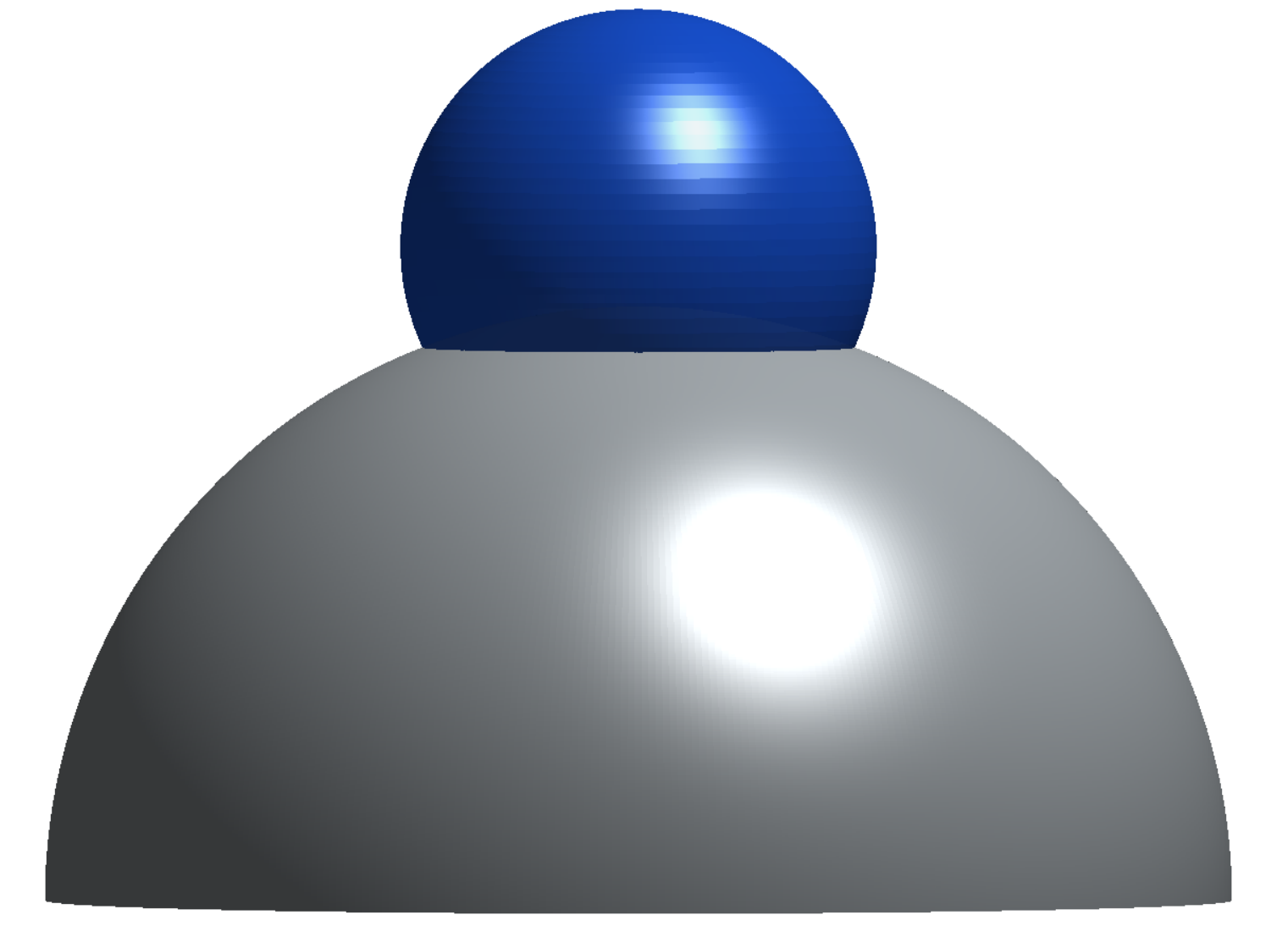}
    ~\\~\\~\\
    \includegraphics[width=0.24\linewidth]{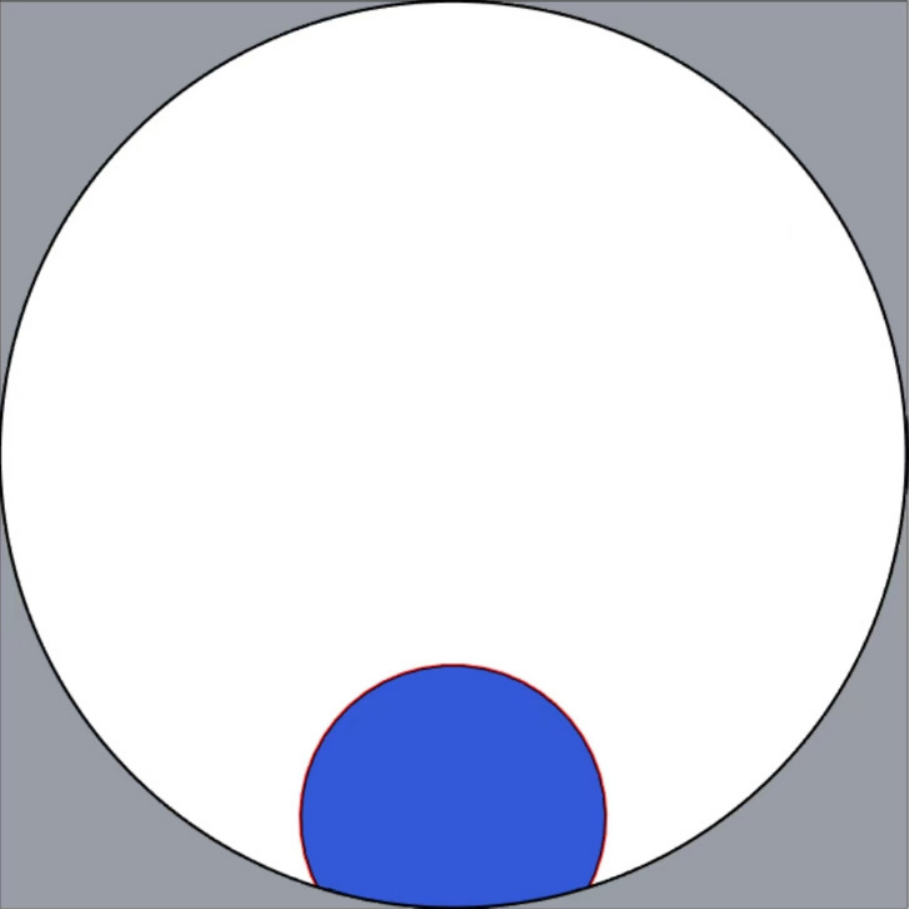}
    \hspace{10mm}
    \includegraphics[width=0.24\linewidth]{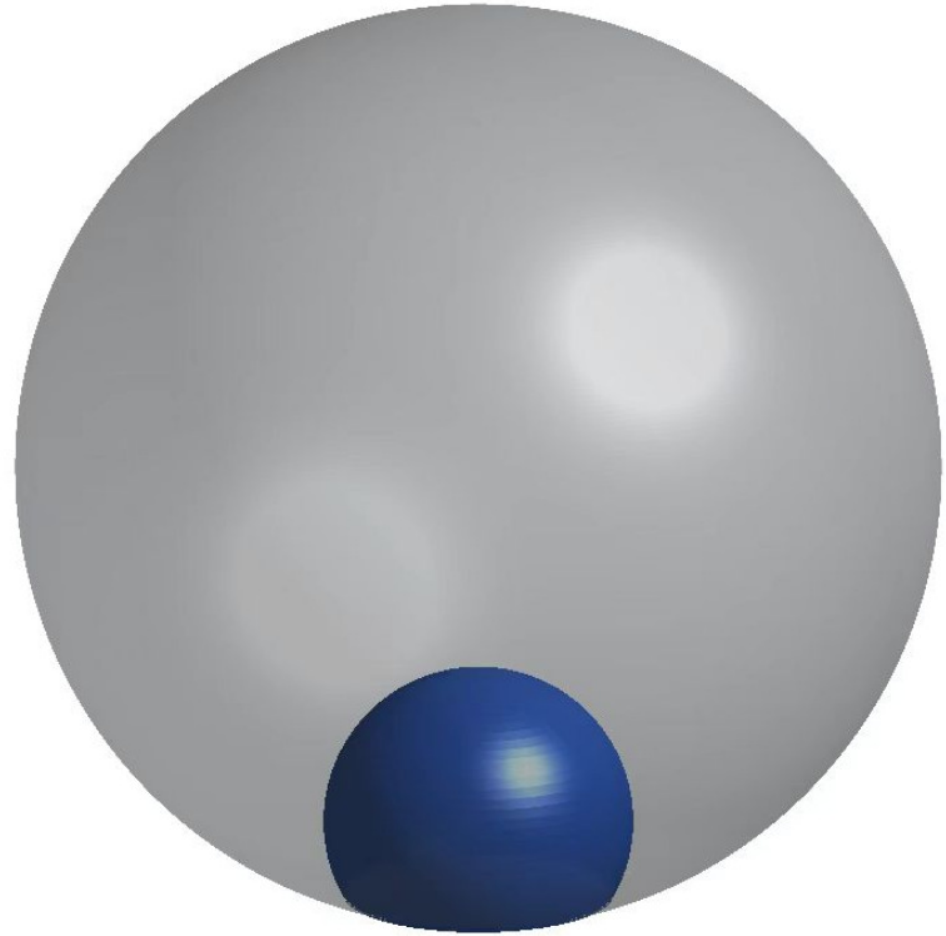}
    \caption{Particle on a curved substrate with positive/negetive curvature \cite{ZHAO2024120407} (left panel); particle on an axisymmetric curved-surface substrate generated by positive/negetive-curvature curve (right panel)
    .}
    \label{fig:1}
\end{figure}

 The rest of this manuscript is outlined as follows. 
 In Section \ref{sec2}, we derive a sharp interface model for three-dimensional axisymmetric SSD problem on the axisymmetric curved-surface substrate via thermodynamic variation and clarify the boundary constraints at the contact lines. 
 In Section \ref{sec3}, we show a novel symmetrized variational formulation by introducing a symmetrized surface energy matrix. 
In Section \ref{sec4}, an energy-stable method and a structure-preserving method are established for the variational formulation.
Section \ref{sec5} presents extensive experiments for the proposed numerical methods. Finally, we draw some conclusions in Section \ref{sec6}.

\section{The sharp-interface model}\label{sec2}
\subsection{The total free energy}\label{subsec2_1}

The open surface \(\mS\) is used to describe the interface that separates the vapor and the film, forming two closed curves, \(\Gamma_i\) and \(\Gamma_o\), with the substrate. The original interfacial energy for the three-dimensional SSD can be defined as follows:
\begin{align}\label{eq:ener0}
    W = \iint_S \gamma_{FV}(\mathcal{N}) \, d\mS + \underbrace{\left( \gamma_{FS} - \gamma_{VS} \right) A\left( \Gamma_o / \Gamma_i \right)}_{\text{Substrate energy}},
\end{align}
where \(A\left( \Gamma_o / \Gamma_i \right)\) denotes the surface area enclosed by the two contact lines on the curved-surface substrate, \(\gamma_{FS}\) and \(\gamma_{VS}\) represent the film/substrate and vapor/substrate surface energy densities, respectively, and \(\gamma_{FV}(\mathcal{N})\) is the surface energy density of the thin film, with \(\mathcal{N}\) representing the unit outward normal vector of the surface.

In this work, we study a special three-dimensional SSD problem on a curved-surface substrate, assuming that the shape of the thin film remains axisymmetric during evolution, and the substrate is also axisymmetric. 
In this case, the surface evolution problem can reduce to a curve evolution problem by considering the evolution of the cross-section profile of the thin film along its redial direction, as illustrated in Figure \ref{fig:2}.
\begin{figure}
    \centering
\includegraphics[width=0.33\linewidth]{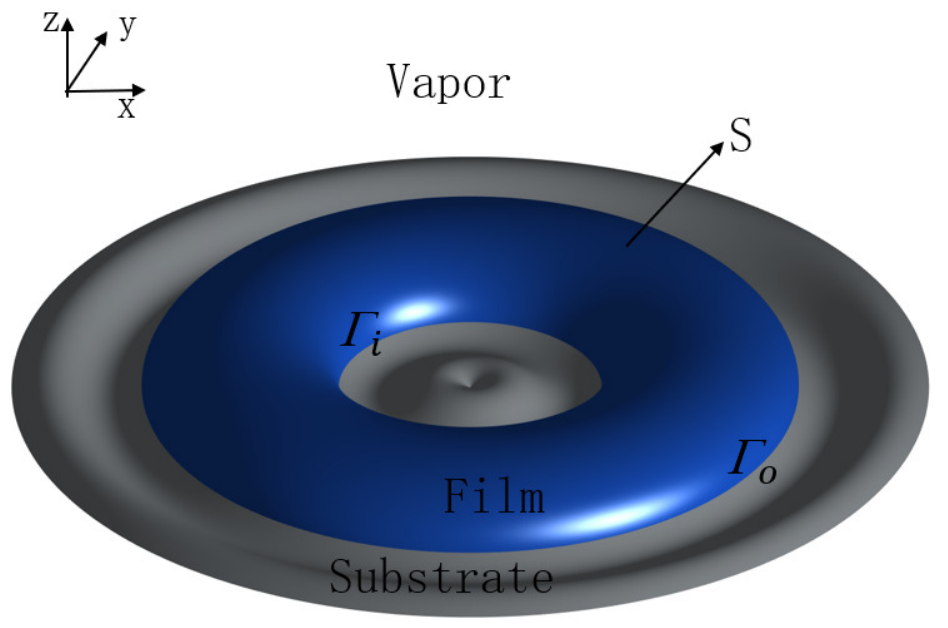}
    \hspace{10mm}
\includegraphics[width=0.33\linewidth]{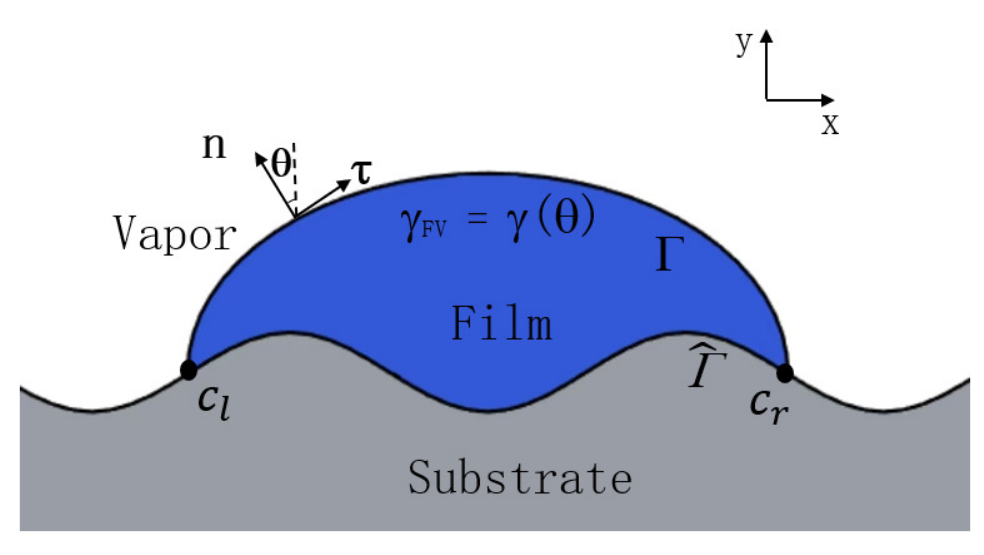}
    \caption{A schematic illustration of SSD: (1) a toroidal thin film on a curved-surface substrate (left panel); (2) the cross-section of an axisymmetric thin film in the cylindrical coordinate system \(r, z\) (right panel). \(c_i\) and \(c_o\) represent the arc lengths of the inner and outer contact points, respectively.}
    \label{fig:2}
\end{figure}
The open surface \( \mathcal{S}\) can be parameterized as follows 
\begin{align*}
    \mathcal{S}(s, \varphi) := \bigg(r(s)\cos \varphi, r(s)\sin \varphi, z(s)\bigg),
\end{align*}
where $s \in [0,L]$ represents the arc length along the radial direction curve, $r(s)$ is the radial distance, $\varphi$ is the azimuth angle, and $z(s)$ is the film height. The surface energy density of the film/vapor can be expressed as \(\gamma(\theta) = \gamma_{FV}(\mathcal{N})\), which satisfies that 
\begin{align}
    \theta = arctan\frac{z_s}{r_s}; ~~~ \gamma(\theta) = \gamma(-\theta), ~~ \forall \theta \in [0, \pi]; ~~~ \gamma(\theta) \in C^2([0, \pi]),
\end{align}
where the subscript \(s\) denotes the derivative with respect to the variable \(s\). The profile of the curved rigid substrate is defined as \(\hat{\Gamma} := \hat{\mat {X}}(c) = (\hat{x}(c), \hat{y}(c))\), where the arc length \(c\) ranges from \(c_l\) to \(c_r\). Therefore, the total energy \eqref{eq:ener0} can be expressed as
\begin{align}\label{eq:ener1}
    W = \iint_{\mS} \gamma(\theta)d\mS + \underbrace{2\pi(\gamma_{FS}-\gamma_{VS})\int_{c_l}^{c_r} \hat x(c)dc}_{\text{Substrate\, energy}}.
\end{align}

We denote the film/vapor-interface profile as $\Gamma = \mat X(s) = (x(s),y(s)), s \in [0,L]$. The unit tangent vector $\mat \tau$ and unit outer normal vector $\mat n$ of the film/vapor-interface curve $\Gamma$ can be expressed as $\mat \tau := (x_{s}, y_{s})$ and $\mat n :=(-y_{s}, x_{s})$, respectively. $\theta^{\,l}_{e}$ and $\theta^{\,r}_{e}$ are the left and right contact angles of the curve $\Gamma$.
Additionally, \(\hat{\mat{\tau}}\) and \(\hat{\mat{n}}\) denote the unit tangent vector and the unit outer normal vector of the curved substrate \(\hat{\Gamma}\), while \(\hat \theta\) represents the angle between the local unit normal vector and \(y\) axis. Furthermore, \(\hat \theta^{\,l} := \hat \theta(c = c_l),~ \hat \theta^{\,r} := \hat \theta(c = c_r)\). On the curved rigid substrate, $\hat {r}(c)$ is the radial distance, and $\hat {z}(c)$ is the film height.

\subsection{Thermodynamic variation}\label{subsec2_2}
The two tangential vectors could be calculated directly as
\begin{align*}\label{eqn:tangential}
\mathcal T _{1} = \bigg(r_s\cos \varphi, r_s\sin \varphi, z_s\bigg), \qquad \mathcal T _{2} = \bigg(-r\sin \varphi, r\cos \varphi, 0\bigg). 
\end{align*}
Then the unit outer normal vector of the surface is as follows
\begin{align*}
\mathcal N = \frac{\mathcal T_1 \times \mathcal T_2}{\left | \mathcal T_1 \times \mathcal T_2 \right | } = \bigg(-z_{s}\cos \varphi, -z_{s}\sin \varphi, r_{s}\bigg). 
\end{align*}
We define \(s = x_1\) as the first parameter and \(\varphi = x_2\) as the second parameter for the surface. Consequently, the first fundamental form can be expressed as 
\begin{align*}
I = Eds^2 + 2Fdsd\varphi + Gd\varphi^2,
\end{align*}
with $E = r_{s}^2 + z_{s}^2 = 1$, $F = 0$ and $G = r^2$. It can also be written in the metric tensor notation as follows
\begin{align*}
(g_{ij}) = \begin{pmatrix}
E & F\\
F & G
\end{pmatrix}
= \begin{pmatrix}
1 & 0\\
0 & r^2
\end{pmatrix}, \quad g = \det (g_{ij}) = r^2.
\end{align*}

Let \(\mS^{\,\varepsilon}\) denote the perturbed surface, obtained by adding a small axisymmetric perturbation to the original surface \(\mS\). The perturbed surface \(\mS^{\,\varepsilon}\) is defined as
\begin{align}
    \mS^{\,\varepsilon} := \bigg(r^{\,\varepsilon}(s) \cos \varphi, r^{\,\varepsilon}(s) \sin \varphi, z^{\,\varepsilon}(s)\bigg),
\end{align}
where \(r^{\,\varepsilon}(s)\) and \(z^{\,\varepsilon}(s)\) represent the perturbed radial and axial coordinates, respectively.
Due to the axisymmetry, this perturbation can be viewed as the corresponding perturbation of the curve in the radial direction. Figure \ref{fig:3} illustrates the perturbations of the generated curves for both the toroidal thin film and the island film, where the red lines represent the perturbed curves.
\begin{figure}
    \centering
    \includegraphics[width=0.4\linewidth]{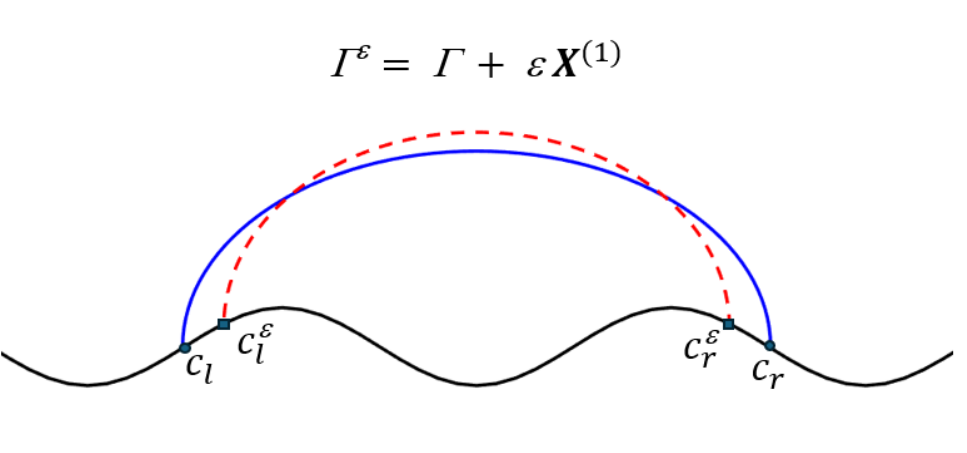}
    \hspace{8mm}
    \includegraphics[width=0.3\linewidth]{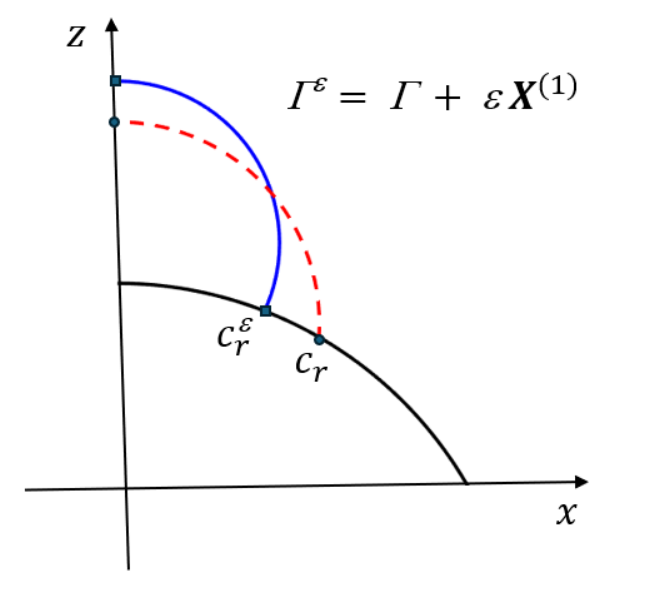}
    \caption{A schematic illustration of an infinitesimal perturbation (represented by the red line) of the generated curve in the radial direction: toroidal thin film (left panel) and island film (right panel) on a curved-surface substrate, with \( \mat X^{(1)} = (r^{(1)}(s), z^{(1)}(s)) = \varphi(s) {\mat n} + \psi(s) {\mat \tau} \in (\text{Lip}[0, L])^2 \).}
    \label{fig:3}
\end{figure}

Consider an infinitesimal perturbation of the interface curve \(\Gamma\) along its normal and tangent directions,
\begin{align*}
\Gamma^{\,\varepsilon } := \Gamma + {\,\varepsilon} \varphi(s) {\mat n} + {\,\varepsilon} \psi(s) {\mat \tau},
\end{align*}
where $\varepsilon$ is a small perturbation parameter, and $\varphi(s), \psi(s)$ are smooth functions with respect to the arc length $s$. 
Given that \(s\) and \(\varphi\) are still taken as the two parameters of the surface \(\mS^{\,\varepsilon}\), the matrix tensor can be expressed as  
\[
g^{\,\varepsilon} = \left[\left(r^{\,\varepsilon}_{s}\right)^2 + \left(z^{\,\varepsilon}_{s}\right)^2\right](r^{\,\varepsilon})^2.
\]
Additionally, the perturbed curve \(\Gamma^{\,\varepsilon}\) can be written as
\[
\Gamma^{\,\varepsilon} = \mat{X}(s) + \varepsilon \, \mat{\vartheta}(s),
\]
where \(\mat{\vartheta}(s) := (u(s), v(s))\) corresponds to the direction of the position increment. This increment vector represents the displacement along the \(x\)-axis and \(y\)-axis, given by  
\[
\begin{aligned}
u(s) &= -y_{s}(s) \, \varphi(s) + x_{s}(s) \, \psi(s), \\
v(s) &= x_{s}(s) \, \varphi(s) + y_{s}(s) \, \psi(s).
\end{aligned}
\]
Equivalently, the smooth functions \(\varphi(s)\) and \(\psi(s)\) can be expressed in terms of the components of the increment vector as  
\[
\begin{aligned}
\varphi(s) &= x_{s}(s) \, v(s) - y_{s}(s) \, u(s) = \mat{\vartheta}(s) \cdot \mat{n}(s), \\
\psi(s) &= x_{s}(s) \, u(s) + y_{s}(s) \, v(s) = \mat{\vartheta}(s) \cdot \mat{\tau}(s).
\end{aligned}
\]
Due to the movement of the contact points on the curved rigid substrate, the incremental vectors at the points must align with the unit tangent direction of the substrate curve at the respective locations. In other words, these incremental vectors should lie along the tangent direction of the curve. This relationship can be expressed as  
\begin{align}
\mat{\vartheta}(0) = \lambda_{l} \, \hat{\mat \tau}(c_{l}), \qquad  
\mat{\vartheta}(L) = \lambda_{r} \, \hat{\mat \tau}(c_{r}),
\end{align}
where \(\lambda_{l}\) and \(\lambda_{r}\) represent the magnitudes of the incremental vectors.

The total free energy of the system with a perturbed surface can be expressed as
\begin{align}
\label{eqn:Wepsilon}
W^\varepsilon
&= \iint_{\mS^\varepsilon } \gamma(\theta^{\,\varepsilon})d\mS^{\,\varepsilon} + 2\pi(\gamma_{FS}-\gamma_{VS})\int_{c^\varepsilon_l}^{c^\varepsilon_r} \hat x(c)dc\nonumber \nonumber\\
&= \int_{0}^{2\pi}\int_{0}^{L}\gamma(\theta ^{\,\varepsilon})\sqrt{g^{\,\varepsilon}}dsd\varphi + 2\pi(\gamma_{FS}-\gamma_{VS})\int_{c^\varepsilon_l}^{c^\varepsilon_r} \hat x(c)dc\nonumber \nonumber\\
&= 2\pi\int_{0}^{L}\gamma(\theta^{\,\varepsilon})\left|\mat X^{\,\varepsilon}_{s}\right|r^{\,\varepsilon}ds + 2\pi(\gamma_{FS}-\gamma_{VS})\left[\int_{0}^{c^\varepsilon_r} \hat x(c)dc-\int_{0}^{c^\varepsilon_l} \hat x(c)dc\right].
\end{align}
Since $\mat\tau = \mat X_s$ and $\mat n = -\mat\tau^\bot = -\mat X_s^\bot$, we expand the following terms at $\varepsilon = 0$:
\begin{align*}
&r^{\,\varepsilon} = r + r^{(1)}\varepsilon + O(\varepsilon ^2), \\
&\left|\mat X^{\,\varepsilon}_{s}\right| = 1 + (\mat \tau \cdot \mat X^{(1)}_{s})\varepsilon + O(\varepsilon ^2), \\
&\gamma(\theta^ \varepsilon) = \gamma(\theta) + \gamma'(\theta)(\mat n \cdot \mat X^{(1)}_{s})\varepsilon + O(\varepsilon ^2),\\
&\int_{0}^{c^\varepsilon_r} \hat x(c)dc
=\int_{0}^{c_r} \hat x(c)dc+r(L)\lambda_r\varepsilon+ O(\varepsilon ^2),\\
&\int_{0}^{c^\varepsilon_l} \hat x(c)dc
=\int_{0}^{c_l} \hat x(c)dc+r(0)\lambda_l\varepsilon+ O(\varepsilon ^2).
\end{align*}
Taking above equations into \eqref{eqn:Wepsilon}, we can write the total energy as
\begin{align}
W^\varepsilon
&= W + W^{(1)}\varepsilon + O(\varepsilon ^2) \nonumber\\
&= 2\pi \int_{0}^{L} \left[\gamma(\theta) + \gamma'(\theta)(\mat n \cdot \mat X^{(1)}_{s})\varepsilon + O(\varepsilon ^2)\right] \left[1 + (\mat \tau \cdot \mat X^{(1)}_{s})\varepsilon + O(\varepsilon ^2)\right]  \left[r + r^{(1)}\varepsilon + O(\varepsilon ^2)\right]ds \nonumber\\
&~~~~ + 2\pi(\gamma_{FS}-\gamma_{VS})\left[\left(\int_{0}^{c_r} \hat x(c)dc+r(L)\lambda_r\varepsilon + O(\varepsilon ^2)\right) - \left(\int_{0}^{c_l} \hat x(c)dc+r(0)\lambda_l\varepsilon + O(\varepsilon ^2)\right)\right].
\end{align}
Here $W^{(1)}$ is the coefficient of $\varepsilon$, given by 
\begin{align}
\label{eq:W_one}
W^{(1)} 
&= 2\pi \int_{0}^{L} \left[\gamma(\theta)r^{(1)} + \gamma'(\theta)(\mat n \cdot \mat X^{(1)}_{s})r + \gamma(\theta)(\mat \tau \cdot \mat X^{(1)}_{s})r\right]ds + 2\pi \left(\gamma_{FS}-\gamma_{VS}\right)\bigg[r(L)\lambda_r - r(0)\lambda_l\bigg].
\end{align}
By using integration by parts in \eqref{eq:W_one}, we obtain
\begin{align}\label{eqn:W(1)}
W^{(1)}
&= 2\pi \int_{0}^{L} \left[\gamma(\theta)r^{(1)} - (r\gamma(\theta) \mat \tau + r\gamma'(\theta) \mat n)_{s} \cdot \mat X^{(1)}\right]ds  + 2\pi \left[(\gamma(\theta) \mat \tau + \gamma'(\theta) \mat n)\cdot \mat X^{(1)}r\right]\Big|_{s=0}^{s=L}  \nonumber\\
&~~~~+ 2\pi (\gamma_{FS}-\gamma_{VS})\bigg[r(L)\lambda_r - r(0)\lambda_l\bigg] \nonumber\\
&= 2\pi \int_{0}^{L} \left[\gamma(\theta)r^{(1)} - (\gamma(\theta) \mat \tau + \gamma'(\theta) \mat n)_{s} \cdot \mat X^{(1)}r - (\gamma(\theta) \mat \tau + \gamma'(\theta) \mat n)\cdot \mat X^{(1)}r_{s}\right]ds + 2\pi \left[(\gamma(\theta) \mat \tau + \gamma'(\theta) \mat n)\cdot \mat X^{(1)}r\right]\Big|_{s=0}^{s=L} \nonumber\\
&~~~~ + 2\pi (\gamma_{FS}-\gamma_{VS})\bigg[r(L)\lambda_r - r(0)\lambda_l\bigg].
\end{align}

Moreover, we assume that $r_{i}$ and $r_{o}$ are two contact points of the curve $\Gamma$, and denote $\theta^{\,l}_{i} := \theta^{\,l}_{e} - \hat \theta^{\,l}$ and $\theta^{\,r}_{i} := \theta^{\,r}_{e} - \hat \theta^{\,r}$ as the extrinsic inner and outer contact angles. Since $\kappa = -\mat X_{ss} \cdot \mat n$, we can obtain 
\begin{align}
&\mat n_{s} = \kappa\mat\tau,\qquad \mat \tau \big|_{s = 0} = (\cos\theta_i^{\,l}, \sin\theta_i^{\,l}), \qquad \mat n \big|_{s = 0} = (-\sin\theta_i^{\,l}, \cos\theta_i^{\,l}), \qquad \theta_{s} = -\kappa,  \nonumber \\
&\mat \tau_{s} = -\kappa \mat n,\qquad \mat \tau \big|_{s = L} = (\cos\theta_i^{\,r}, \sin\theta_i^{\,r}), \qquad \mat n \big|_{s = L} = (-\sin\theta_i^{\,r}, \cos\theta_i^{\,r}), \qquad \mat \tau \cdot \mat X^{(1)}r_{s} - r^{(1)} = z_{s}\mat X^{(1)} \cdot \mat n.  \nonumber 
\end{align}
Despite adding the small cylindrical perturbation, we require that the two contact lines remain on the curved-surface substrate at all times, i.e.,
\begin{align}
\mat X^{(1)}\big|_{s=0} = (\lambda_{l}, 0),\qquad \mat X^{(1)}\big|_{s=L} = (\lambda_{r}, 0). 
\end{align}
Substituting above relations into \eqref{eqn:W(1)}, we obtain
\begin{align}\label{eqn:W(1)a}
W^{(1)}
&= 2\pi \int_{0}^{L} \left[\gamma(\theta)r^{(1)} + (\gamma'(\theta)\kappa \mat\tau + \gamma(\theta)\kappa \mat n + \gamma''(\theta)\kappa \mat n - \gamma'(\theta)\kappa \mat\tau)\mat X^{(1)}r - (\gamma(\theta) \mat \tau + \gamma'(\theta) \mat n)\cdot \mat X^{(1)}r_{s}\right]ds \nonumber\\
&~~~~ + 2\pi r(L)\lambda_{r}\left[\gamma(\theta^{\,r}_{e})\cos(\theta^{\,r}_{i}) - \gamma(\theta^{\,r}_{e})\sin(\theta^{\,r}_{i})\right] - 2\pi r(0)\lambda_{l}\left[\gamma(\theta^{\,l}_{e})\cos(\theta^{\,l}_{i}) - \gamma(\theta^{\,l}_{e})\sin(\theta^{\,l}_{i})\right]
 \nonumber\\
&~~~~ + 2\pi (\gamma_{FS}-\gamma_{VS})\Big(r(L)\lambda_r - r(0)\lambda_l\Big) \nonumber\\
& = 2\pi \int_{0}^{L} \left[(\gamma(\theta) + \gamma''(\theta))\kappa (\mat n \cdot \mat X^{(1)}) r - \gamma(\theta)z_{s} (\mat n \cdot \mat X^{(1)}) + \gamma'(\theta)(\mat n \cdot \mat X^{(1)}) r_{s}\right]ds \nonumber\\
&~~~~ + 2\pi r(L)\lambda_{r}\left[\gamma(\theta^{\,r}_{e})\cos(\theta^{\,r}_{i}) - \gamma(\theta^{\,r}_{e})\sin(\theta^{\,r}_{i}) + (\gamma_{FS}-\gamma_{VS})\right] \nonumber\\
&~~~~ - 2\pi r(0)\lambda_{l}\left[\gamma(\theta^{\,l}_{e})\cos(\theta^{\,l}_{i}) - \gamma(\theta^{\,l}_{e})\sin(\theta^{\,l}_{i}) + (\gamma_{FS}-\gamma_{VS})\right] \nonumber\\
&= \iint_{S} \left[(\gamma(\theta) + \gamma''(\theta))\kappa - \frac{\gamma(\theta)z_s + \gamma'(\theta)r_{s}}{r}\right]\mat n \cdot \mat X^{(1)}dS \nonumber\\
&~~~~ +\int_{\Gamma_o} \left[\gamma(\theta^{\,r}_{e})\cos(\theta^{\,r}_{i}) - \gamma(\theta^{\,r}_{e})\sin(\theta^{\,r}_{i}) + (\gamma_{FS}-\gamma_{VS})\right]\lambda_r d\Gamma \nonumber\\
&~~~~ -\int_{\Gamma_i} \left[\gamma(\theta^{\,l}_{e})\cos(\theta^{\,l}_{i}) - \gamma(\theta^{\,l}_{e})\sin(\theta^{\,l}_{i}) + (\gamma_{FS}-\gamma_{VS})\right]\lambda_l d\Gamma.
\end{align}
From \eqref{eqn:W(1)a}, we can immediately obtain the variations of the total energy with respect to the surface and the two contact lines, given by 
\begin{subequations}\label{W}
\begin{align}
&\frac{\delta W}{\delta \mS} = \left[(\gamma(\theta) + \gamma''(\theta))\kappa - \frac{\gamma(\theta)z_s + \gamma'(\theta)r_{s}}{r}\right],
\label{eqn:W_S}\\
&\frac{\delta W}{\delta \Gamma_o} = \gamma(\theta^{\,r}_{e})\cos(\theta^{\,r}_{i}) - \gamma(\theta^{\,r}_{e})\sin(\theta^{\,r}_{i}) + (\gamma_{FS}-\gamma_{VS}),  \label{eqn:W_Gamma_o}\\
&\frac{\delta W}{\delta \Gamma_i} = -\left[\gamma(\theta^{\,l}_{e})\cos(\theta^{\,l}_{i}) - \gamma(\theta^{\,l}_{e})\sin(\theta^{\,l}_{i}) + (\gamma_{FS}-\gamma_{VS})\right].\label{eqn:W_Gamma_i}
\end{align}
\end{subequations}

Note that the variations in \eqref{W} are considered when there are two contact lines on the island film. However, for the island film with only one outer contact line, these variations reduce to a simpler case. In this scenario, after introducing the perturbation, the island film at the boundary must satisfy the following conditions:
\begin{align}
\mat X^{(1)}\big|_{s=0} = (0, z^{(1)}(0)),\qquad \mat X^{(1)}\big|_{s=L} = (\lambda_{r}, 0).
\end{align}
Moreover, applying integration by parts does not yield boundary terms at \(s = 0\).

\subsection{The model and its properties}
From the anisotropic Gibbs–Thomson relation \cite{Sutton95}, the chemical potential can be defined as
\begin{equation}
  \mu = \Omega_{\,0}\frac{\delta W}{\delta \mS} = \Omega_{\,0}\left[ (\gamma(\theta) + \gamma''(\theta))\,\kappa - \frac{\gamma(\theta)z_s + \gamma'(\theta)r_s}{r} \right], 
\end{equation}
where $\Omega_{\,0}$ denotes the atomic volume of the thin film material. According to Fick's laws of diffusion, we can obtain the normal velocity given by surface diffusion \cite{Mullins57, Cahn74}
\begin{equation}
\mat j = -\frac{D_s\,v}{k_B\,T_e}\bigtriangledown_s\mu, \qquad v_n = -\Omega_0(\bigtriangledown_s\cdot\mat j) = \frac{D_s\,v\,\Omega_0}{k_B\,T_e}\bigtriangledown_s^2\mu, 
\end{equation}
where $\mat j$ represents mass flux, $D_s$ denotes surface diffusivity, $k_B\,T_e$ is thermal energy, $v$ is the number of diffusing atoms per unit area, and $\bigtriangledown_s$ represents surface gradient. 
The two contact lines \(\Gamma_i\) and \(\Gamma_o\) move along the substrate with velocities \(v_c^i\) and \(v_c^o\), respectively, governed by energy gradient flows. These velocities are described by the time-dependent Ginzburg–Landau kinetic equations:  
\begin{subequations}
    \begin{align}
        v_c^o &= -\eta \frac{\delta W}{\delta \Gamma_o} = -\eta\left[ \gamma(\theta_d^{\,o}) \cos \theta_d^{\,o} - \gamma'(\theta_d^{\,o}) \sin \theta_d^{\,o} - (\gamma_{VS} - \gamma_{FS}) \right], \\
        v_c^i &= -\eta \frac{\delta W}{\delta \Gamma_i} = \eta\left[ \gamma(\theta_d^{\,i}) \cos \theta_d^{\,i} - \gamma'(\theta_d^{\,i}) \sin \theta_d^{\,i} - (\gamma_{VS} - \gamma_{FS}) \right],
    \end{align}
\end{subequations}
where \(\eta \in (0, +\infty)\) denotes the mobility of the contact lines.

We denote \(\overline{L}\) as the characteristic length scale and \(\gamma_0\) as the characteristic surface energy scale.
The time scale is chosen as \(\frac{\overline{L}^4}{B\,\gamma_0}\), with \(B = \frac{D_s\,v\,\Omega_0^2}{k_B\,T_e}\). Additionally, the contact line mobility is taken as \(\frac{B}{\overline{L}^3}\). Since the model is axisymmetric, we have the following surface Laplacian of the chemical potential \(\mu\):
$$\bigtriangledown_s^2\mu = \frac{1}{\sqrt{g}}\partial_i(\sqrt{g}g^{ij}\partial_j\mu) = \frac{1}{r}(r\mu_s)_s.$$ 
Then the sharp-interface model for the SSD can be described in the following dimensionless
form:
\begin{subequations}\label{governingeq}
\begin{align}\label{eqn:model_a}
 &\mat X_t\cdot\mat n = \frac{1}{r}(r\,\mu_s)_s, \quad 0<s<L(t), \quad t>0,  \\ \label{eqn:model_b}
 &\mu = (\gamma(\theta) + \gamma''(\theta))\kappa - \frac{\gamma(\theta)\, z_s + \gamma'(\theta)\, r_s}{r},  \\ \label{eqn:model_c}
 &\kappa = -\mat X_{ss}\cdot\mat n, \quad \mat n = -\mat X_s^\bot,
\end{align}
\end{subequations}
where $\Gamma(t) := \mat X(s, t) = (r(s, t), z(s, t))$ is the generating curve of surface $\mS$, $L := L(t)$ denotes total arc length of open curve $\Gamma(t)$, $\mu(s, t)$ is chemical potential, $\kappa(s, t)$ is curvature of curve and $\mat n = (n_1, n_2) = (-z_s, r_s)$ is outward unit normal vector. The initial data is given as
\begin{equation}\label{initial}
\mat X(s, 0) := \mat X_0(s) = (r(s, 0), z(s, 0)), \qquad 0\le s\le L_0:=L(0). 
\end{equation}
The governing equation \eqref{governingeq} satisfies the following boundary conditions:
\begin{itemize}
\item [(i)] contact line condition
\begin{equation}\label{eqn:boundry1}
z(L, t) = \hat{\mat X}(c_r), \quad \left\{\begin{matrix}
 z(0, t) = \hat{\mat X}(c_l), 
&~~~~\text{if} ~~r(0, t)>0,\\
 z_s(0, t) = 0, & \text{otherwise}, 
\end{matrix}\right. \quad t\ge 0. 
\end{equation}
The above describes the case where both contact points lie on the curved substrate. In the alternative case, only the outer contact point is on the curved substrate, such that  \begin{equation}\label{eqn:boundry1_1}
z(L, t) = \hat{\mat X}(c_r), \quad \left\{\begin{matrix}
 r(0, t) = 0, 
&~~~~\text{if} ~~r(0, t)>0,\\
 z_s(0, t) = 0, & \text{otherwise}, 
\end{matrix}\right. \quad t\ge 0. 
\end{equation} 
\item[(ii)] relaxed contact angle condition 
\begin{equation}\label{eqn:boundry2}
\frac{dc_r}{dt} = -\eta f(\theta_e^{\,r}; \theta_i^{\,r}; \sigma), \quad \left\{\begin{matrix}
 \frac{dc_l}{dt} = \eta f(\theta_e^{\,l};\theta_i^{\,l}; \sigma), & ~~~~\text{if}~~r(0, t)>0,\\
 r(0, t) = 0, & \text{otherwise}, 
\end{matrix}\right. \quad t\ge 0, 
\end{equation}
where the function $f(\theta_e; \theta_i; \sigma)$ is defined by
\begin{equation}
f(\theta_e;\theta_i; \sigma) = \gamma(\theta_e)\cos\theta_i - \gamma'(\theta_e)\sin\theta_i - \sigma, \quad \theta\in [-\pi, \pi], \quad \sigma = \frac{\gamma_{VS} - \gamma_{FS}}{\gamma_0}. \nonumber
\end{equation}
\item[(iii)]
zero-mass flux condition 
\begin{equation}\label{eqn:boundry3}
\mu_s(0, t) = 0, \quad \mu_s(L ,t) = 0, \quad t\ge 0. 
\end{equation}
\end{itemize}

\begin{rem}
The contact line condition in \eqref{eqn:boundry1}-\eqref{eqn:boundry1_1} ensures that the moving contact lines remain on the curved-surface substrate, considering two cases: either both the inner and outer contact lines or only the outer contact line is situated on the surface.
The relaxed contact angle condition \eqref{eqn:boundry2} is the requirement for the contact angles existing on the curved-surface substrate. 
Furthermore, the zero-mass flux condition \eqref{eqn:boundry3} indicates the total volume/mass conservation of the thin film throughout the entire evolution process, which is equivalent to having no mass flux at the contact lines.
\end{rem}

In what follows, we will demonstrate that the total volume of the thin film, enclosed by the surface \( S(t) \) and the substrate, is conservative, and the total energy of the system is dissipative. To this end, a new parameter $\rho \in \bI = [0, 1]$ is introduced to parameterize the evolution curves $\Gamma (t)$ such that $\Gamma (t)$, $t \in [0, T]$, is a family of open curves:
\begin{align*}
\Gamma(t)=\mat{X}(\rho,t)=\left( r(\rho,t), z(\rho,t) \right)^\top:\mathbb{I}\times[0,T]\to\mathbb{R}^2. 
\end{align*}
Obviously, we can obtain the relationship between $s$ and $\rho$ as $s(\rho,t)=\int_{0}^{\rho} | \partial_\rho\mat{X} |d\rho $. Furthermore, we also have $\partial_\rho s=| \partial_\rho\mat{X}  |$ and $ds=\partial_\rho s d\rho= | \partial_\rho\mat{X}  | d\rho$.

Define $V(t)$ as the volume between the surface and the substrate, and $W(t)$ as the total free energy. Then by using the surface integral calculation, we obtain
\begin{align}
&V(\mat X(t)) = \int_{0}^{2\pi} \int_{0}^{L(t)} rzr_s dsd\varphi - \int_{0}^{2\pi} \int_{c_r}^{c_l} \hat r \hat z \hat r_c dcd\varphi = 2\pi \int_{0}^{L(t)} rzr_s ds - 2\pi\int_{c_l}^{c_r} \hat r \hat z \hat r_c dc, \label{eqn:def of V}\\
&W(\mat X(t)) = 2\pi \int_{0}^{L(t)} r\gamma(\theta)ds - 2\pi\sigma\left[\int_{c_l}^{c_r} \hat x(c)dc\right]. \label{eqn:def of W}
\end{align}
We can obtain the volume conservation and energy dissipation of the system \eqref{governingeq}, together with the boundary conditions (i), (ii) and (iii). 
Indeed, it follows from \eqref{eqn:boundry2} and \eqref{eqn:boundry3} that 
\begin{align}\label{eqn:mass}
\frac{dV(\mat X(t))}{dt} &= 2\pi \frac{d}{dt} \left[\int_{0}^{1} rzr_\rho d\rho - \int_{c_l}^{c_r} \hat r \hat z \hat r_c dc\right] \nonumber\\
&= 2\pi \int_{0}^{1}\left[r_t z r_\rho + r z_t r_\rho + rz r_{\rho t}\right]d\rho - 2\pi\left(\hat r\hat z \hat r_c\frac{dc_r}{dt}\right)_{c=c_r} + 2\pi\left(\hat r\hat z \hat r_c \frac{dc_l}{dt}\right)_{c=c_l} \nonumber\\
&= 2\pi \int_{0}^{1} \left[r_t z r_\rho + r z_t r_\rho - (rz)_{\rho}r_t\right]d\rho + 2\pi \left(rzr_t\right)\Bigg |_{\rho = 0}^{\rho = 1} - 2\pi\left(\hat r\hat z \frac{d\hat r(c_r)}{dt}\right)_{c=c_r} + 2\pi\left(\hat r\hat z\frac{d\hat r(c_l)}{dt}\right)_{c=c_l} \nonumber \\
&= 2\pi \int_{0}^{1} (r z_t r_\rho - r z_\rho r_t)d\rho = 2\pi \int_{0}^{L(t)} r \mat X_t \cdot \mat n ds \nonumber \\
&= 2\pi \int_{0}^{L(t)} (r \mu_s)_s ds = 0,
\end{align}
which can obtain the volume conservation, i.e. $V(\mat X(t)) \equiv V(\mat X(0))$. Noting \eqref{W} and combining with \eqref{eqn:model_a} and \eqref{eqn:model_b}, we have
\begin{align}
\frac{dW(\mat X(t))}{dt} &= 2\pi \int_{0}^{L} r\left[(\gamma(\theta) + \gamma''(\theta))\kappa - \frac{\gamma(\theta)z_s + \gamma'(\theta)r_{s}}{r}\right]\mat n \cdot \mat X^{(1)}ds \nonumber\\
&~~~~ +\int_{\Gamma_o} \left[\gamma(\theta^{\,r}_{e})\cos(\theta^{\,r}_{i}) - \gamma(\theta^{\,r}_{e})\sin(\theta^{\,r}_{i}) + (\gamma_{FS}-\gamma_{VS})\right]\frac{dc_r}{dt} d\Gamma \nonumber\\
&~~~~ -\int_{\Gamma_i} \left[\gamma(\theta^{\,l}_{e})\cos(\theta^{\,l}_{i}) - \gamma(\theta^{\,l}_{e})\sin(\theta^{\,l}_{i}) + (\gamma_{FS}-\gamma_{VS})\right]\frac{dc_l}{dt} d\Gamma \nonumber \\
&= 2\pi \int_{0}^{L(t)} (r\mu_s)_s \mu ds - \frac{2\pi}{\eta}\left[r(L)\left(\frac{dc_r}{dt}\right)^2 + r(0)\left(\frac{dc_l}{dt}\right)^2\right] \nonumber\\
&= -2\pi \int_{0}^{L(t)} r(\mu_s)^2 ds- \frac{2\pi}{\eta}\left[r(L)\left(\frac{dc_r}{dt}\right)^2 + r(0)\left(\frac{dc_l}{dt}\right)^2\right] \le 0,
\end{align}
which immediately implies the energy dissipation, i.e. $W(\mat X(t_2)) \le W(\mat X(t_1)) \le W(\mat X(0)), ~t_2 \ge t_1 \ge 0$.

\section{Variational formulation}\label{sec3}
To build the variational formulation of the system \eqref{governingeq}, we first introduce a symmetric matrix $\mat{B}(\theta)$, given by 
\begin{equation}\label{Matrix:Bqx}
\mat{B}(\theta)=
\begin{pmatrix}
\gamma(\theta)&-\gamma'(\theta) \\
\gamma'(\theta)&\gamma(\theta)
\end{pmatrix}
\begin{pmatrix}
\cos2\theta&\sin2\theta \\
\sin2\theta&-\cos2\theta
\end{pmatrix}+
\mathscr S(\theta)\left[\frac{1}{2}\mat{I}-\frac{1}{2}\begin{pmatrix}
\cos2\theta&\sin2\theta \\
\sin2\theta&-\cos2\theta
\end{pmatrix}\right], 
\end{equation}
where $\mat I$ is a $2\times 2$ identity matrix, and the objective of the stability function \(\mathscr{S}(\theta)\) is to ensure that the matrix \(\mat{B}(\theta)\) satisfies the stability estimate given by inequality \eqref{inequality2}. Then the following equivalent relation can be obtained. 

\begin{lem}\label{lem:equivalent}
 With the matrix $\mat{B}(\theta)$, \eqref{eqn:model_b} can be written as 
\begin{equation}
r\mu\mat{n}=\partial_s\left[r\mat{B}(\theta)\partial_s\mat{X}\right]-\gamma(\theta)\mat e_1
\qquad \text{with}\qquad \mat e_1=(1, 0)^\top. \label{eqn:equiv}
\end{equation} 
\end{lem}

\begin{proof}
We can easily obain
\begin{align}\label{eqn:equiv_pf1}
    &z = \mat X\cdot\mat e_1, \quad r = \mat X\cdot\mat e_2, \quad \partial_s{z} = \partial_s \mat X\cdot\mat e_2, \quad \partial_s r = \partial_s \mat X\cdot\mat e_1, \\ \label{tangent}
    &\mat \tau = \partial_s \mat X = (\cos \theta, \sin \theta)^\top, \quad \mat n = -\mat \tau^\bot, \quad \mat n_s = -\partial_{ss} \mat X^\bot, \quad \partial_s \theta = (\sin^2 \theta + \cos^2 \theta)\partial_s \theta = \partial_{ss} \mat X\cdot \mat n,
\end{align}
where $\mat e_2=(0, 1)^\top.$ By substituting \eqref{eqn:equiv_pf1}, \eqref{tangent} and \eqref{eqn:model_c} into \eqref{eqn:model_b}, we have
\begin{align}\label{eqn:equiv_pf2}
r\mu \mat n &= r\left[\gamma(\theta) + \gamma''(\theta)\right](\partial_{ss}\mat X \cdot \mat n)\mat n  \nonumber \\
&= r\left[\gamma(\theta)(\partial_{ss}\mat X \cdot \mat n)\mat n + \gamma''(\theta)(\partial_{ss}\mat X \cdot \mat n)\mat n + \gamma'(\theta)(\partial_{ss}\mat X \cdot \mat n)\mat \tau -\gamma'(\theta)(\partial_{ss}\mat X \cdot \mat n)\mat \tau \right] \nonumber \\
&= r\left[\gamma(\theta)\mat \tau_s + \gamma''(\theta)\partial_s \theta \mat n + \gamma'(\theta)\partial_s \theta\mat \tau +\gamma'(\theta) \partial_s \mat n\right] \nonumber \\
&= r\partial_s\left[\gamma(\theta)\mat \tau + \gamma'(\theta) \mat n\right] \nonumber \\
&= \partial_s \left[r\gamma(\theta)\mat \tau + r\gamma'(\theta)\mat n\right] - \gamma(\theta)(\mat \tau \cdot \mat e_1)\mat \tau - (\mat \tau \cdot \mat e_1)\gamma'(\theta)\mat n.
\end{align}
From the definition of $\mat B(\theta)$, we can obtain 
\begin{align}
\label{eqn:equiv_pf3}
\mat{B}(\theta)\partial_s\mat X &=
\begin{pmatrix}
\gamma(\theta)&-\gamma'(\theta) \\
\gamma'(\theta)&\gamma(\theta)
\end{pmatrix}
\begin{pmatrix}
\cos2\theta&\sin2\theta \\
\sin2\theta&-\cos2\theta
\end{pmatrix}\binom{\cos\theta}{\sin\theta}+
\mathscr S(\theta)\begin{pmatrix}
\frac{1-\cos2\theta}{2}&-\frac{1}{2}\sin2\theta \\
-\frac{1}{2}\sin2\theta&\frac{1+\cos2\theta}{2}
\end{pmatrix}\binom{\cos\theta}{\sin\theta} \nonumber \\
&=\begin{pmatrix}
\gamma(\theta)&-\gamma'(\theta) \\
\gamma'(\theta)&\gamma(\theta)
\end{pmatrix}\binom{\cos\theta}{\sin\theta} =\gamma(\theta)\binom{\cos\theta}{\sin\theta} +\gamma'(\theta)\binom{-\sin\theta}{\cos\theta}\nonumber \\
&
=\gamma(\theta)\mat\tau+\gamma'(\theta)\mat n.
\end{align}
Therefore, it can be inferred from \eqref{eqn:equiv_pf2} and \eqref{eqn:equiv_pf3} that
\begin{align}\label{eqn:equiv_pf4}
r\mu \mat n = \partial_s \left[r\mat{B}(\theta)\partial_s\mat X\right] - \gamma(\theta)(\mat \tau \cdot \mat e_1)\mat \tau - (\mat \tau \cdot \mat e_1)\gamma'(\theta)\mat n.
\end{align}
In addition, we have
\begin{align}\label{eqn:equiv_pf5}
[\gamma(\theta)\partial_s z + \gamma'(\theta)\partial_s r]\mat n
&= \left[\gamma(\theta)\partial_s\mat X\cdot\mat e_2+\gamma'(\theta)(\partial_s \mat X\cdot\mat e_1)\right]\mat n \nonumber \\
&= [\gamma(\theta)\partial_s\mat X^\bot\cdot\mat e_1 + \gamma'(\theta)(\partial_s\mat X\cdot\mat e_1)]\mat n \nonumber \\
&= -\gamma(\theta)(\mat n\cdot\mat e_1)\mat n + (\mat\tau\cdot\mat e_1)\gamma'(\theta)\mat n. 
\end{align}
Decomposing the $\mat e_1$ vector can obtain $\mat e_1= (\mat\tau \cdot \mat e_1)\mat\tau + (\mat n \cdot \mat e_1)\mat n$. Then from \eqref{eqn:model_b}, \eqref{eqn:equiv_pf4} and \eqref{eqn:equiv_pf5}, we finally prove \eqref{eqn:equiv}.
\end{proof}

Next, we define the functional space on the domain $\mathbb{I}$ as 
\begin{align*}
L^2(\mathbb{I}):=\left \{ u:\mathbb{I}\to\mathbb{R}\,\bigg|\,\int\limits_{\Gamma(t)}\lvert u(s) \rvert ^2ds=\int\limits_\mathbb{I}\lvert u(s(\rho,t)) \rvert^2\partial_\rho s\,d\rho<+\infty \right \}.  
\end{align*}
The inner product $(u,v)$ is defined as
\begin{equation*}
(u,v):=\int\limits_{\Gamma(t)}u(s)\,v(s)ds=\int\limits_{\mathbb{I}}u(s(\rho,t))\,v(s(\rho,t))\,\partial_\rho s\,d\rho, \quad\forall u,v\in L^2(\mathbb{I}). 
\end{equation*}
We can directly extend the above inner product to $[L^2(\mathbb{I})]^2$. Additionally, we  define the Sobolev spaces 
\begin{align*}
&H^1(\mathbb{I}):=\left \{ u:\mathbb{I}\to\mathbb{R},u\in L^2(\mathbb{I})\nn\
\text{and}\ \partial_\rho u \in L^2(\mathbb{I}) \right \},
\end{align*}
and two special functional spaces: 
\begin{align*}
&H^{(r)}_{a, b}(\mathbb I) = \left\{ u\in H^1(\mathbb I): u(0) = a; ~~ u(1) = b \right\}, \\
&H^{(z)}_{a, b}(\mathbb I) = \left\{ u\in H^1: u(1) = \hat{y}(c_r); ~~\text{if}~~ a>0, ~~ u(0) = \hat{y}(c_l) \right\}, 
\end{align*}
where $a$ and $b$ are the radii of the inner and outer contact lines. Obviously, we have $ H^{(r)}_{0, 0}(\mathbb I) = H^{1}_0(\mathbb I)$.

Introducing a test function $\varphi \in H^1(\mathbb{I})$, multiplying $r\varphi$ to $\eqref{eqn:model_a}$, integrating over $\Gamma(t)$, and noting \eqref{eqn:boundry3}, we have 
\begin{align}\label{eqn:varf_1}
\int_{\Gamma(t)} r \mat X_s \cdot \mat n \varphi ds &= \int_{\Gamma(t)}-(r \mu_s)_s \varphi ds \nonumber \\
&= \int_{\Gamma(t)} r \mu_s \varphi_s ds - (r \mu_s \varphi)\big |_{s = 0}^{s = L(t)} \nonumber \\
&= \int_{\Gamma(t)} r \mu_s \varphi_s ds.
\end{align}
Then, multiplying $\mat \psi = (\psi_1, \psi_2)^\top \in H_{a,b}^{(r)}(\mathbb{I}) \times H_{a,b}^{(z)}(\mathbb{I})$ to \eqref{eqn:equiv}, integrating it over $\mathbb{I}$, using integrating by part, and combining the boundary conditions \eqref{eqn:boundry2}, we obtain
\begin{align}\label{eqn:varf_2}
\int_{\Gamma(t)} r\mu\mat n\cdot \mat\psi ds &= \int_{\Gamma(t)} \partial_s\left[r\mat B(\theta)\partial_s \mat X\right]\mat \psi ds - \int_{\Gamma(t)} \gamma(\theta) \psi_1 ds \nonumber \\
&= -\int_{\Gamma(t)}\left[r \mat B(\theta)\partial_s \mat X\right]\partial_s \mat \psi ds - \int_{\Gamma(t)} \gamma(\theta)\psi_1 ds +\left[r\mat B(\theta)\partial_s \mat X\right]\cdot \mat \psi \Big{|}_{s = 0}^{s = L} \nonumber \\
&= -\int_{\Gamma(t)}\left[r \mat B(\theta)\partial_s \mat X\right]\partial_s \mat \psi ds - \int_{\Gamma(t)} \gamma(\theta)\psi_1 ds + r\begin{pmatrix}
\gamma'(\theta)&\gamma(\theta) \\
\gamma'(\theta)&\gamma(\theta)
\end{pmatrix}\binom{\cos\theta}
{\sin\theta}\mat \psi\Bigg|_{s = 0}^{s = L} \nonumber\\
&= -\int_{\Gamma(t)}\left[r \mat B(\theta)\partial_s \mat X\right]\partial_s \mat \psi ds - \int_{\Gamma(t)} \gamma(\theta)\mat\psi ds \nonumber\\
&~~~~- \frac{1}{\eta}\left[r(L)\frac{dc_r(t)}{dt} \frac{ \hat{\mat X}_c(c_r(t))}{\lvert\hat {\mat X}_c(c_r(t))\rvert^2}\cdot \mat \psi(1) + r(0)\frac{dc_l(t)}{dt}\frac{ \hat{\mat X}_c(c_r(t))}{\lvert\hat {\mat X}_c(c_r(t))\rvert^2}\cdot\mat\psi(0)\right]\nonumber \\
&~~~~+ \sigma\left[r(L)\frac{\hat{\mat X}_c(c_l(t))}{\lvert\hat {\mat X}_c(c_l(t))\rvert^2}\cdot\mat\psi(1) - r(0)\frac{\hat{\mat X}_c(c_l(t))}{\lvert\hat{\mat X}_c(c_l(t))\rvert^2}\cdot\mat\psi(0)\right].
\end{align}
Combining \eqref{eqn:varf_1}, \eqref{eqn:varf_2} and $ds=\partial_\rho s d\rho= | \partial_\rho\mat{X}  |d\rho$, we can obtain the variational formulation of the system \eqref{governingeq}. Suppose $\Gamma(0):=\mat X(\rho, 0), \rho 
\in \mathbb{I} = [0, 1]$, to find the open curves $\Gamma(t):=\mat X(\cdot,t)\in H_{a,b}^{(r)}(\mathbb{I}) \times H_{a,b}^{(z)}(\mathbb{I})$, and $\mu(\cdot,t)\in H^1(\mathbb{I})$, such that
\begin{subequations}\label{eqn:variation}
\begin{align}
&\left(r\partial_t \mat X \cdot \mat n, \varphi \lvert\partial_{\rho} \mat X\rvert\right) - \left(r\partial_{\rho}\mu, \partial_{\rho}\varphi \lvert\partial_{\rho} \mat X \rvert^{-1}\right) = 0, ~~\forall \varphi \in H^1(\mathbb{I}), \label{eqn:variation_1}\\
&\left(r \mu \mat n, \mat \psi \lvert\partial_{\rho}\mat X\rvert\right) + \left(r \mat B(\theta)\partial_{\rho}\mat X, \partial_{\rho}\mat \psi \lvert\partial_{\rho}\mat X\rvert^{-1}\right) + \left(\gamma(\theta), \psi_1\lvert \partial_{\rho}\mat X\rvert\right) \nonumber\\
&+\frac{1}{\eta}\left[r(L)\frac{dc_r(t)}{dt} \frac{\hat{\mat X}_c(c_r(t))}{\lvert\hat{\mat X}_c(c_r(t))\rvert^2}\cdot \mat \psi(1) + r(0)\frac{dc_l(t)}{dt}\frac{\hat{\mat X}_c(c_r(t))}{\lvert\hat {\mat X}_c(c_r(t))\rvert^2}\cdot\mat\psi(0)\right] \nonumber\\
&-\sigma\left[r(L)\frac{\hat{\mat X}_c(c_l(t))}{\lvert\hat {\mat X}_c(c_l(t))\rvert^2}\cdot\mat\psi(1) - r(0)\frac{\hat{\mat X}_c(c_l(t))}{\lvert\hat{\mat X}_c(c_l(t))\rvert^2}\cdot\mat\psi(0)\right] = 0, ~~\forall \mat \psi \in H_{a,b}^{(r)}(\mathbb{I}) \times H_{a,b}^{(z)}(\mathbb{I}). \label{eqn:variation_2}
\end{align}
\end{subequations}
Next, we will prove that the variational formulation \eqref{eqn:variation} maintains volume conservation and energy stability.
\begin{thm}\label{lem:structure}(Volume conservation and energy stability). Assume $(\mat X(\cdot, t),\mu(\cdot, t))\in \left(H_{a,b}^{(r)}(\mathbb{I}) \times H_{a,b}^{(z)}(\mathbb{I}), H^1(\mathbb{I})\right)$ is the solution of variational formulation \eqref{eqn:variation}. Then there will be the following conclusion
\begin{equation}
\label{eqn:structure_vf}
V(\mat X(t))\equiv V(\mat X(0)), \qquad W(\mat X(t_2))\le W(\mat X(t_1))\le W(\mat X(0)), \qquad t_2\ge t_1\ge 0,
\end{equation}
i.e., volume conservation and energy dissipation. 
\end{thm}
\begin{proof}
Following the same process as in the proof of \eqref{eqn:mass}, by differentiating $V(\mat X(t))$ with respect to $t$, we have 
\begin{align} \label{eqn:volume conservation}
\frac{dV(\mat X(t))}{dt} = 2\pi \int_{0}^{1} r \partial_t \mat X \cdot \mat n d\rho,\quad t \ge 0.
\end{align}
Taking $\varphi = 1$ in \eqref{eqn:variation_1}, naturally obtaining $\partial_{\rho}\varphi = 0$, then we get
\begin{align*}
\left(r\partial_t \mat X \cdot \mat n, 1\right) = 0,\quad t \ge 0,
\end{align*}
which combining with \eqref{eqn:volume conservation} implies the volume conservation.

Next, by taking the derivative of $W(\mat X(t))$ on $t$, we have 
\begin{align}\label{energy stability}
\frac{dW(\mat X(t))}{dt} &= \frac{d}{dt}\left[ 2\pi \int_{0}^{L(t)}r\gamma(\theta)\lvert\partial_s \mat X\rvert ds - 2\pi\sigma\int_{c_l}^{c_r} \hat x(c)dc\right] = \frac{d}{dt}\left[2\pi \int_{0}^{1}r\gamma(\theta)\lvert\partial_{\rho} \mat X\rvert d\rho - 2\pi\sigma\int_{c_l}^{c_r} \hat x(c)dc\right] \nonumber\\
&= 2\pi \int_{0}^{1} \partial_t r \gamma(\theta)\lvert\partial_{\rho} \mat X\rvert d \rho + 2\pi \int_{0}^{1}r \gamma'(\theta)\partial_t \theta \lvert\partial_{\rho} \mat X\rvert d \rho + 2\pi \int_{0}^{1} r \gamma(\theta)\partial_t \left(\lvert\partial_{\rho} \mat X\rvert \right) d \rho \nonumber\\
&~~~~- 2\pi\sigma\Bigg[r(L)\frac{dc_r}{dt} - r(0)\frac{dc_l}{dt}\Bigg] \nonumber\\
&= 2\pi \int_{0}^{1} \partial_t r \gamma(\theta)\lvert\partial_{\rho} \mat X\rvert d \rho + + 2\pi \int_{0}^{1}r \left[ \gamma'(\theta)\mat n +\gamma(\theta)\mat \tau\right] \cdot \partial_{\rho} \partial_t\mat X d \rho - 2\pi\sigma\left[r(L)\frac{dc_r(t)}{dt} - r(0)\frac{dc_l(t)}{dt}\right] \nonumber \\
&= 2\pi \int_{0}^{1} \left[\partial_t r \gamma(\theta)\lvert\partial_{\rho} \mat X\rvert + r \mat B(\theta)\partial_{\rho} \mat X \cdot \partial_{\rho} \partial_t \mat X \right]d \rho - 2\pi\sigma\left[r(L)\frac{dc_r(t)}{dt} - r(0)\frac{dc_l(t)}{dt}\right],\qquad t \ge 0.
\end{align}
Denoting $\varphi = \mu$ in \eqref{eqn:variation_1} and $\psi = \partial_t \mat X$ in \eqref{eqn:variation_2}, we obtain
\begin{align*}
\frac{dW(\mat X(t))}{dt} = -2\pi\left(r \partial_{\rho}\mu,\partial_{\rho} \mu \big|\partial_{\rho} \mat X\big|^{-1}\right) - \frac{2\pi}{\eta}\left[r(L)\left(\frac{dc_r(t)}{dt}\right)^2 + r(0)\left(\frac{dc_l(t)}{dt}\right)^2\right] \le 0,
\end{align*}
which indicates the energy dissipation. Therefore, we have completed the proof.
\end{proof}

\section{Parametric finite element approximation}\label{sec4}

In this section, we present two types of PFEMs for the variational formulation \eqref{eqn:variation}. To this end, we first discretize time by dividing it into \(M\) intervals, where the time step is defined as \(\Delta t_m = t_{m+1} - t_m\).
Additionally, we consider a uniform partition of the domain \(\bI = [0, 1]\) as \(\bI = \bigcup_{j=1}^N \bI_j\), where each interval is defined by \(\bI_j = [\rho_{j-1}, \rho_j]\) with the mesh size \(h = \frac{1}{N}\). Then, we define the following finite element spaces:
\begin{align*}
&K(\mathbb I) :=\left \{ u\in C(\mathbb{I}): u \big|_{\mathbb{I}_j} \in \mathbb{P}_1, j = 1, 2, . . . ,N \right \} \subset H^1(\mathbb{I}),\\
&\mat K^{m+1}_{c_l, c_r}(\mathbb I) := \left\{ \mat \eta\in \left[K(\mathbb I)\right]^2: \mat \eta(0)\cdot \left(\hat{\mat X}(c_l^{m+1}) - \hat{\mat X}(c_l^{m})\right)^\bot = 0,~\mat \eta(1)\cdot \left(\hat{\mat X}(c_r^{m+1}) - \hat{\mat X}(c_r^{m})\right)^\bot = 0\right\},
\end{align*}
where $\mathbb P_1$ represents all polynomials with a maximum degree of $1$.

We use $\Gamma^m := \mat X^m, \mu^m, \theta^m$ and $\mat n^m$ to approximate the numerical value of the moving curve $\Gamma(t_m) := \mat X(\cdot, t_m), \mu, \theta, \mat n$ at time $t_m$. Then, we have 
\begin{align*}
\Gamma^m = \bigcup_{j=1}^N h_j^m, ~~~~\left\{h_j^m\right\}_{j = 1}^N \text{are connected line segments of the curve}~ \Gamma_m. 
\end{align*}
Define the following mass-lumped $L^2$-inner product 
as 
\begin{align*}
\Big(\mat v, \mat w\Big)_{\Gamma^m}^h = \frac{1}{6} h \sum_{j=1}^N \left[(\mat v \cdot \mat w)(\rho_j^-) + 4(\mat v \cdot \mat w)(\rho_{j-\frac{1}{2}}) + (\mat v \cdot \mat w)(\rho_{j-1}^+)\right],
\end{align*}
where \(\mat v(\rho^-)\) and \(\mat v(\rho^+)\) are one-sided limits defined as \(\mat v(\rho_j^{\pm}) = \lim\limits_{\delta \to 0^+} \mat v(\rho \pm \delta)\).

The unit tangent and normal vectors can be numerically calculated as
\begin{align*}
\mat \tau^m = \partial_s \mat X^m = \frac{\partial_{\rho} \mat X^m}{\lvert\partial_{\rho} \mat X^m\rvert}, \quad \mat n^m = -(\partial_s \mat X^m)^\bot = -\frac{(\partial_{\rho} \mat X^m)^\bot}{\lvert\partial_{\rho} \mat X^m\rvert}.
\end{align*}

We use the chain rule $\hat{\mat X}_t = \hat{\mat X}_c c_t$ at the two contact points $c_l, c_r$, and introduce the approximations as follows:
\begin{subequations}\label{eqn:G}
\begin{align}
&r^m(0)\frac{\hat{\mat X}_c (c_l^m)}{\lvert\hat{\mat X}_c (c_l^m)\rvert^2} \approx \mat G(c_l^m,c_l^{m+1}) = \frac{(\hat{\mat X} (c_l^{m+1}) - \hat{\mat X} (c_l^m))\int_{c_l^m}^{c_l^{m+1}}\hat x(c) dc}{\lvert\hat{\mat X} (c_l^{m+1}) - \hat{\mat X} (c_l^m)\rvert^2},\\
&r^m(L)\frac{\hat{\mat X}_c (c_r^m)}{\lvert\hat{\mat X}_c (c_r^m)\rvert^2} \approx \mat G(c_r^m,c_r^{m+1}) = \frac{(\hat{\mat X} (c_r^{m+1}) - \hat{\mat X} (c_r^m))\int_{c_r^m}^{c_r^{m+1}}\hat x(c) dc}{\lvert\hat{\mat X} (c_r^{m+1}) - \hat{\mat X} (c_r^m)\rvert^2}.
\end{align}
\end{subequations}
Taking $\Gamma^0 = \mat X^0 \in \mat K^{m+1}_{c_l, c_r}(\mathbb I)$ satisfying $\mat X^0 (0) = \hat{\mat X}(c_l^0)$ and $\mat X^0 (1) = \hat{\mat X}(c_r^0)$, 
then the energy-stable method of the variational formulation \eqref{eqn:variation} is to seek for $\mat X^{m+1} \in \mat K^{m+1}_{c_l, c_r}(\mathbb I)$ and $\mu^{m+1} \in K(\mathbb I)$, for $m \ge 0$, such that
\begin{subequations}\label{eqn:full discrete}
\begin{align}
&\frac{1}{\Delta t_m}\left( \mat X^{m+1}-\mat X^m, \varphi ^h\mat f^{m+\frac{1}{2}} \right) - \left( r^m\partial_\rho\mu ^{m+1}, \partial_\rho\varphi^h\left | \partial_\rho\mat X ^m \right |^{-1} \right)=0, \qquad\forall\varphi\in K(\mathbb I),  
\label{eqn:structure_a}\\
&\left( \mu ^{m+1}\mat f^{m+\frac{1}{2}}, \mat\psi ^h \right) + \left( \gamma(\theta ^{m+1}), \psi_1 ^h \left | \partial_\rho\mat X^{m+1} \right | \right) + 
\left( r^m\mat{B}(\theta ^m)\partial_\rho\mat X^{m+1},  
 \partial_\rho\mat\psi ^h\left | \partial_\rho\mat X ^m \right |^{-1}\right) \nonumber \\
 &~~~~+\frac{1}{\eta \Delta t_m}\left[(c^{m+1}_r - c^m_r) \mat G(c_r^m, c_r^{m+1})\cdot \mat \psi^h(1) + (c^{m+1}_l - c^m_l) \mat G(c_l^m, c_l^{m+1})\cdot\mat\psi^h(0)\right]
 \nonumber\\
 &~~~~-\sigma\left[\mat G(c_r^m, c_r^{m+1})\cdot \mat \psi^h(1) - \mat G(c_l^m,c_l^{m+1})\cdot\mat\psi^h(0)\right] = 0, \qquad\forall \mat \psi^h \in \mat K^{m+1}_{c_l, c_r}(\mathbb I),
\label{eqn:structure_b}
\end{align}
\end{subequations}
where $\mat G(\cdot, \cdot)$ as an approximation of $ r^m\hat{\mat X}_c(c_{l,r}(t))$ is introduced in \eqref{eqn:G} for the stability of the substrate energy, and $\mat f^{m+\frac{1}{2}}\in [L^\infty(\bI)]^2$ represents a time-integrated approximation of $\mat f=r\,|\partial_\rho\mat X|\,\mat n$, given by
\begin{align}
\mat f^{m+\frac{1}{2}}= -\frac{1}{6}\Bigl[2r^m\,\partial_\rho\mat X^m+2r^{m+1}\,\partial_\rho\mat X^{m+1} + r^m\,\partial_\rho\mat X^{m+1} + r^{m+1}\,\partial_\rho\mat X^m\Bigr]^\perp.\label{eq:weightednormal}
\end{align}

\begin{lem}
Assume $3\gamma(\theta)\ge \gamma(\pi + \theta)$ and 
\begin{align*}
\mathscr S(\theta)\ge \mathscr S_0(\theta):=inf\left \{ \alpha\ge 0:P_\alpha(\theta, \hat{\theta})-Q(\theta, \hat{\theta})\ge 0, \forall\hat{\theta}\in[-\pi, \pi] \right \}, \quad\theta\in[-\pi, \pi], 
\end{align*}
where $P_\alpha(\theta, \hat{\theta})$ and $ Q(\theta, \hat{\theta})$ are defined by  
\begin{align*}
&P_\alpha(\theta, \hat{\theta}):=2\sqrt{(-\gamma(\theta)+\alpha(-\sin\hat{\theta}\cos\theta+\cos\hat{\theta}\sin\theta)^2 + f(\theta, \hat{\theta}))\gamma(\theta)}, \quad\forall\theta, \hat{\theta}\in[-\pi, \pi], \quad\alpha\ge 0, \\
&Q(\theta, \hat{\theta}):=\gamma(\hat{\theta})+\gamma(\theta)(\sin\theta\sin\hat{\theta}+\cos\theta\cos\hat{\theta}) -\gamma'(\theta)(-\sin\hat{\theta}\cos\theta+\cos\hat{\theta}\sin\theta), \quad\forall\theta, \hat{\theta}\in[-\pi, \pi], 
\end{align*}
with $f(\theta, \hat{\theta}) := 2(\sin\theta\sin\hat{\theta} + \cos\theta\cos\hat{\theta}) -\gamma'(\theta)(-\sin\hat{\theta}\cos\theta+\cos\hat{\theta}\sin\theta).$
Then there holds 
\begin{equation}\label{inequality2}
\frac{1}{\left | \mat v \right |}\left(\mat{B}(\theta)\mat w\right)\cdot(\mat w-\mat v)\ge\left | \mat w \right |\gamma(\hat{\theta})-\left | \mat v \right |\gamma(\theta), 
\end{equation}
where $\frac{\mat v}{\left | \mat v \right |} = (-\sin \theta, \cos \theta)$.
\end{lem}
\begin{proof}
We omit the proof here, as it follows a similar approach to those found in Ref. \cite{bao2024}.
\end{proof}
\begin{thm}\label{Energy stability}(Energy stability). Let $(\mat X^{m+1}, \mu^{m+1})$ be the solution of the energy-stable method \eqref{eqn:full discrete}. Then the scheme is energy stable in the sense that 
\begin{align}\label{eqn:ES}
W(\mat X^{m+1}) - W(\mat X^{m}) \le 0,\quad 0 \le m \le M-1.
\end{align}
\begin{proof}
Selecting $\varphi^h = \Delta t_m\mu^{m+1}$ in \eqref{eqn:structure_a}, $\mat\psi^h = \mat X^{m+1} - \mat X^m$ in \eqref{eqn:structure_b}, then by rearranging these three expressions, we can obtain
\begin{align}\label{enerequa}
&\left( r^m\,\mat B(\theta^m)\mat X_\rho^{m+1}, \left( \mat X^{m+1} - \mat X^m \right)_\rho \left| \mat X_\rho^m \right|^{-1} \right) + \left( \gamma(\theta^{m+1}), (r^{m+1} - r^m)\left| \mat X_\rho^{m+1} \right| \right)- \sigma\left( \int_{c_r^m}^{c_r^{m+1}}\hat x(c) dc - \int_{c_l^m}^{c_l^{m+1}}\hat x(c) dc \right) \nonumber \\
& = -\left( r^m\,\mu_\rho^{m+1}, \mu_\rho^{m+1}\left| \mat X_\rho^m \right| \right) -\frac{1}{\eta\Delta t_m}\left[ \left( c_r^{m+1} - c_r^m \right)\int_{c_r^m}^{c_r^{m+1}}\hat x(c) dc
+ \left( c_l^{m+1} - c_l^m \right)\int_{c_l^m}^{c_l^{m+1}}\hat x(c) dc \right]. 
\end{align}
Then, by choosing $\mat w = \mat X_\rho^{m+1}$ and $\mat v = \mat X_\rho^m$ in \eqref{inequality2},  we have
\begin{align} \label{inequalityzhu}
\mat B(\theta^m)\mat X_\rho^{m+1}\cdot\left( \mat X^{m+1} - \mat X^m \right)_\rho\left| \mat X_\rho^m \right|^{-1}\ge \gamma(\theta^{m+1})\left| \mat X_\rho^{m+1} \right| - \gamma(\theta^m)\left| \mat X_\rho^m \right|.
\end{align}
By utilizing the properties of matrix $\mat B(\theta^m)$ in \eqref{inequalityzhu}, and using the inequality $(a-b)a\ge \frac{1}{2}a^2 - \frac{1}{2}b^2$ and the definition of the energy \(W(\mat X)\) in \eqref{eqn:def of W}, we obtain
\begin{align}\label{enerinequa}
&\left( r^m\,\mat B(\theta^m)\mat X_\rho^{m+1}, \left( \mat X^{m+1} - \mat X^m \right)_\rho \left| \mat X_\rho^m \right|^{-1} \right) + \left( \gamma(\theta^{m+1}), (r^{m+1} - r^m)\left| \mat X_\rho^{m+1} \right| \right) \nonumber - \sigma\left( \int_{c_r^m}^{c_r^{m+1}}\hat x(c) dc - \int_{c_l^m}^{c_l^{m+1}}\hat x(c) dc \right) \nonumber \\
& \ge \left( r^{m+1}, \gamma(\theta^{m+1})\left| \mat X_\rho^{m+1} \right| \right) - \left( r^m, \gamma(\theta^m)\left| \mat X_\rho^m \right| \right) \nonumber - \sigma\left( \int_{c_r^m}^{c_r^{m+1}}\hat x(c) dc - \int_{c_l^m}^{c_l^{m+1}}\hat x(c) dc \right) \nonumber \\
&\ge \left( r^{m+1}, \gamma(\theta^{m+1})\left| \mat X_\rho^{m+1} \right| \right) - \sigma\int_{c_l^{m+1}}^{c_r^{m+1}}\hat x(c) dc \nonumber - \left( r^m, \gamma(\theta^m)\left| \mat X_\rho^m \right| \right) + \sigma\int_{c_l^{m}}^{c_r^{m}}\hat x(c) dc \nonumber \\
&= \frac{1}{2\pi}\left( W(\mat X^{m+1}) - W(\mat X^m) \right).
\end{align}
Finally, thanks to  
\begin{align}\label{eqn:gezero}
    \left( c_{l,r}^{m+1} - c_{l,r}^m\right)\int_{c_{l,r}^m}^{c_{l,r}^{m+1}}\hat x(c) dc = \left( c_{l,r}^{m+1} - c_{l,r}^m \right)^2 r(\cdot) \ge 0,
\end{align}
and by using \eqref{enerequa} and \eqref{enerinequa}, we can obtain
\begin{align*}
W(\mat X^{m+1}) - W(\mat X^m)\le -2\pi\left( r^m\,\mu_\rho^{m+1}, \mu_\rho^{m+1}\left| \mat X_\rho^m \right| \right) -\frac{\pi}{\eta\Delta t_m}\left[ \left( c_r^{m+1} - c_r^m \right)\int_{c_r^m}^{c_r^{m+1}}\hat x(c) dc
+ \left( c_l^{m+1} - c_l^m \right)\int_{c_l^m}^{c_l^{m+1}}\hat x(c) dc \right]\le 0, 
\end{align*}
which implies the energy stability. Therefore, we have completed the proof.
\end{proof}
\end{thm}
\begin{thm} 
(Volume-conservation estimate). Let $(\mat X^{m+1}, \mu^{m+1}, c_l^{m+1}, c_r^{m+1})$ be a solution of the energy-stable method \eqref{eqn:full discrete}. Then it holds that
\begin{align}\label{eqn:DFdiffenceV}
    V(\mat X^{m+1}) - V(\mat X^m) = -H(c_l^m, c_l^{m+1}) + H(c_r^m, c_r^{m+1}),
\end{align}
where we introduce
\begin{align} \label{def H}
    H(c_1,c_2) &:= \frac{\pi \left(\hat{\mat X}_2(c_2) - \hat{\mat X}_2(c_1)\right)}{6k^2} \left[\left(\hat{\mat X}_2(c_1) - b\right)^2 + 4\left(\frac{\hat{\mat X}_2(c_1) + \hat{\mat X}_2(c_2)}{2} - b\right)^2 + \left(\hat{\mat X}_2(c_2) - b\right)^2\right]
    - 2\pi \int_{c_1}^{c_2}\hat{r} \hat{z} \partial_{c}\hat{r} dc,
\end{align}
as the volume of the region obtained by rotating the domain between $\hat{\Gamma}$
and the line segment connecting the two points $\hat{\mat{X}}(c_1)$ and $\hat{\mat{X}}(c_2)$, as shown in Figure \ref{fig:4},
where \(k\) and \(b\) are the slope and intercept of $\hat {\mat X}(c_1)$ and $\hat {\mat X}(c_2)$ connecting lines respectively.
\begin{figure}
    \centering
    \includegraphics[width=0.45\linewidth]{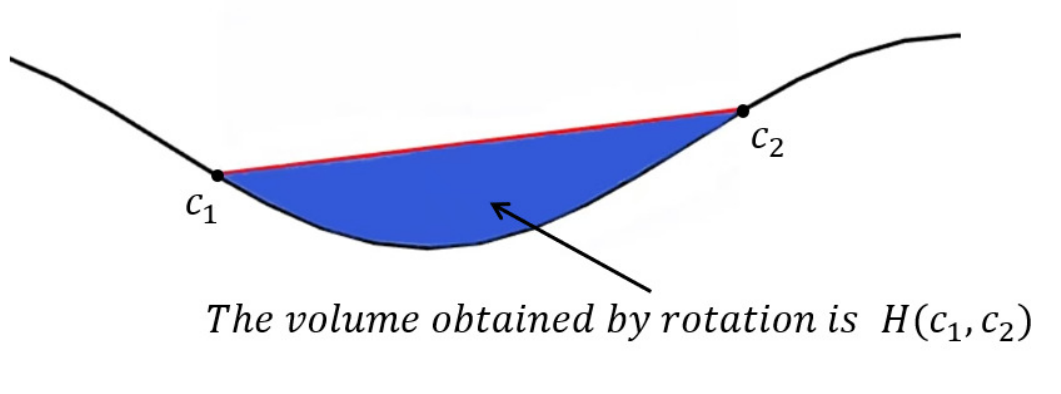}
    \caption{Illustration of the volume \(H(c_1, c_2)\) in \eqref{def H}}
    \label{fig:4}
\end{figure}
\begin{proof}
We consider appropriate extensions of $\mat X^m, \mat X^{m+1}$ to $[0, 4]$ as follows:
\begin{align}
    \mat X^m_{upd}(\rho) = \left\{\begin{matrix} \mat X^m (\rho),&\quad \rho \in [0, 1],\nonumber\\
    \mat X^m(1),&\quad \rho \in (1,2],\nonumber\\
    \mat X^m(3 - \rho),&\quad \rho \in (2,3],\nonumber\\
    \mat X^m(0),&\quad \rho \in (3,4],\nonumber\\
    \end{matrix}\right.
    \quad \mat X^{m+1}_{upd}(\rho) = \left\{\begin{matrix} \mat X^{m+1} (\rho),&\quad \rho \in [0, 1],\nonumber\\
    (2 - \rho)\mat X^{m+1}(1) + (\rho - 1)\mat X^{m}(1),&\quad \rho \in (1,2],\nonumber\\
    \mat X^m(3 - \rho),&\quad \rho \in (2,3],\nonumber\\
    (4 - \rho)\mat X^m(0) + (\rho - 3)\mat X^{m+1}(0),&\quad \rho \in (3,4].\end{matrix}\right.
\end{align}
Let $\Gamma^m_{upd} = \mat X^m_{upd}\left([0, 4]\right)$ and $\Gamma^{m+1}_{upd} = \mat X^{m+1}_{upd}\left([0, 4]\right)$.
We can get that $\Gamma^m_{upd}$ and $\Gamma^{m+1}_{upd}$ are two closed polygonal curves, and $\Gamma^m_{upd}$ is degenerated. The volumes of the regions enclosed by 
 rotating $\Gamma^m_{upd}$ and $\Gamma^{m+1}_{upd}$ are denoted as $\mathcal{V}^m_{upd}$ and $\mathcal{V}^{m+1}_{upd}$, respectively.

 Then we introduce the intermediate closed curve between $\Gamma^m_{upd}$ and $\Gamma^{m+1}_{upd}$ as
\begin{align}
    \mat X^h_{upd}(\rho, t) := \frac{t_{m+1} - t}{\Delta t_m}\mat X^m_{upd}(\rho) + \frac{t - t_{m}}{\Delta t_m}\mat X^{m+1}_{upd}(\rho), \qquad \forall \rho \in [0, 4], ~~ t \in [t_m, t_{m+1}],
\end{align}
and denote $\Gamma^h_{upd}(t) = \mat X^h_{upd}([0, 4])$. Let $\mathcal{V}_{upt}(t)$ be the volume of the region by rotating $\Gamma^h_{upd}(t)$. By Reynolds transport theorem, we get
\begin{align}\label{eqn:extend V}
    \frac{d}{dt}\mathcal{V}_{upd}(t) &= 2\pi\frac{d}{dt}\int_0^4 r^hz^h r^h_{\rho}d\rho = 2\pi\int_0^4 \left[r^h z^h_t\partial_{\rho}r^h - r^h z^h_{\rho} r^h_t\right]d\rho \nonumber\\
    &= 2\pi\int_{\Gamma^h_{upd} (t)} r^h  \left(\mat X_{upd}^h\right)_t  \cdot\mat n_{upd}^h ds = \frac{2\pi}{\Delta t}\int_0^4 r^h \left(\mat X_{upd}^{m+1} - \mat X_{upd}^{m}\right)\cdot\left[-(\mat X_{upd}^{h})_{\rho}\right]^\bot d\rho \nonumber\\
    &= \frac{2\pi}{\Delta t}\int_0^1 r^h \left(\mat X^{m+1} - \mat X^{m}\right)\cdot\left[-(\mat X_{upd}^{h})_{\rho}\right]^\bot d\rho \nonumber\\
    &~~~ + \frac{2\pi}{\Delta t_m}\int_1^2 r^h \left(2 - \rho\right)\left(\mat X^{m+1}(1) - \mat X^{m}(1)\right)\cdot\left[\frac{t_{m+1} - t}{\Delta t_m}\left(\mat X^{m+1}(1) - \mat X^{m}(1)\right)\right]^\bot d\rho \nonumber\\
    &~~~ + \frac{2\pi}{\Delta t_m}\int_3^4 r^h \left(\rho - 3\right)\left(\mat X^{m+1}(0) - \mat X^{m}(0)\right)\cdot\left[\frac{t - t_{m}}{\Delta t_m}\left(\mat X^{m+1}(0) - \mat X^{m}(0)\right)\right]^\bot d\rho \nonumber\\
    &= \frac{2\pi}{\Delta t_m}\int_0^1 r^h \left(\mat X^{m+1} - \mat X^{m}\right)\cdot\left[-\frac{t_{m+1} - t}{\Delta t_m}\mat X_{\rho}^m(\rho) - \frac{t - t_{m}}{\Delta t_m}\mat X_{\rho}^{m+1}(\rho)\right]^\bot d\rho,
\end{align}
where we have used that $\mat n_{upd}^h = -\frac{\left(\left(\mat X_{upd}^{h}\right)_{\rho}\right)^\bot}{\left|\left(\mat X_{upd}^{h}\right)_{\rho}\right|}$ and $\left(\mat X_{upd}^{h}\right)_{t} = \frac{\mat X_{upd}^{m+1} - \mat X_{upd}^{m}}{\Delta t_m}$. Integrating \eqref{eqn:extend V} for $t$ from $t_m$ to $t_{m+1}$ then obtains
\begin{align}\label{eqn:extendV-V}
    \mathcal{V}_{upd}^{m+1} - \mathcal{V}_{upd}^{m} &= \int_{t_m}^{t_{m+1}} \frac{2\pi}{\Delta t_m} \int_0^1 r^h\left(\mat X^{m+1} - \mat X^{m}\right)\cdot\left[-\frac{t_{m+1} - t}{\Delta t_m}\mat X_{\rho}^m(\rho) - \frac{t - t_{m}}{\Delta t_m}\mat X_{\rho}^{m+1}(\rho)\right]^\bot d\rho dt \nonumber\\
    &= \int_0^1 \left(\mat X^{m+1} -\mat X^{m}\right) \frac{2\pi}{\Delta t_m} \int_{t_m}^{t_{m+1}} r^h\left(-\left(\mat X^h_{upd}\right)_{\rho}\right)^\bot dt d\rho \nonumber\\
    &= 2\pi\int_0^1 \left(\mat X^{m+1} -\mat X^{m}\right)\cdot \mat (-\frac{1}{6})\left[r^m \mat X^m_\rho + 4r^{m+\frac{1}{2}} \mat X^{m+\frac{1}{2}}_\rho + r^{m+1} \mat X^{m+1}_\rho\right]^\perp d\rho \nonumber\\
    &= 2\pi\int_0^1 \left(\mat X^{m+1} -\mat X^{m}\right)\cdot (-\frac{1}{6})\left[2r^m\,\mat X^m_\rho + 2r^{m+1}\, \mat X^{m+1}_\rho + r^m\, \mat X^{m+1}_\rho + r^{m+1}\, \mat X^m_\rho\right]^\perp d\rho\nonumber\\
    &= 2\pi\left(\left(\mat X^{m+1} -\mat X^{m}\right)\cdot \mat f^{m+\frac{1}{2}}, 1\right).
\end{align}
In addition, we know $\mathcal{V}_{upd}^m = 0$, and $\mathcal{V}_{upd}^{m+1}$ can be given as
\begin{align}
    \mathcal{V}_{upd}^{m+1} &= 2\pi\int_0^1 r_{upd}^{m+1} z_{upd}^{m+1} \partial_{\rho}r_{upd}^{m+1} d\rho = 2\pi\int_0^1 r^{m+1} z^{m+1} \partial_{\rho}r^{m+1} d\rho - 2\pi\int_0^1 r^m z^m \partial_{\rho}r^m d\rho \nonumber\\
    &- \frac{\pi \left(y^{m+1}(\rho_N) - y^{m}(\rho_N)\right)}{6k_1^2} \left[\left(y^{m}(\rho_N) - b_1\right)^2 + 4\left(\frac{y^{m}(\rho_N) + y^{m+1}(\rho_N)}{2} - b_1\right)^2 + \left(y^{m+1}(\rho_N) - b_1\right)^2\right] \nonumber\\
    &+ \frac{\pi \left(y^{m+1}(\rho_0) - y^{m}(\rho_0)\right)}{6k_2^2} \left[\left(y^{m}(\rho_0) - b_2\right)^2 + 4\left(\frac{y^{m}(\rho_0) + y^{m+1}(\rho_0)}{2} - b_2\right)^2 + \left(y^{m+1}(\rho_0) - b_2\right)^2\right],
\end{align}
where \(k_1\) and \(b_1\) represent the slope and intercept of the line connecting \(\mat X^{m+1}(\rho_N)\) and \(\mat X^m(\rho_N)\), while \(k_2\) and \(b_2\) are the slope and intercept of the line connecting \(\mat X^{m+1}(\rho_0)\) and \(\mat X^m(\rho_0)\).
Together with \eqref{eqn:extendV-V}, the definition of $V$ in \eqref{eqn:def of V}, and the definition of \(H\) in \eqref{def H}, we can obtain
\begin{align}\label{eqn:difference V}
    V(\mat X^{m+1}) - V(\mat X^m) &= 2\pi \int_{\Gamma^{m+1}} r^{m+1}z^{m+1}r^{m+1}_s ds -2\pi \int_{\Gamma^{m}} r^m z^m r^m_s ds - 2\pi\int_{c_l^{m+1}}^{c_r^{m+1}} \hat r \hat z \hat r_c dc + 2\pi\int_{c_l^{m}}^{c_r^{m}} \hat r \hat z \hat r_c dc \nonumber\\
    &= 2\pi\left(\left(\mat X^{m+1} -\mat X^{m}\right)\cdot \mat f^{m+\frac{1}{2}}, 1\right) - H(c_l^{m},c_l^{m+1}) + H(c_r^{m},c_r^{m+1}).
\end{align}
Choosing $\varphi^h = 1$ in \eqref{eqn:structure_b} and combining \eqref{eqn:difference V} then prove \eqref{eqn:DFdiffenceV}.
\end{proof} 
\end{thm}

Even though we employ a time-integrated approximation, we are still unable to achieve accurate volume conservation with the energy-stable method.
Our next step is to correct $\mat f^{m+\frac{1}{2}}$ to construct a scheme that can ensure accurate volume conservation. We introduce $\delta \mat f^{m+\frac{1}{2}} \in K(\mathbb{I})$ as follows
\begin{align}\label{eqn:deltaf}
    \delta \mat f^{m+\frac{1}{2}}(q_j) = \left\{\begin{matrix} \frac{3H (c_r^m, c_r^{m+1})}{\pi \left|\left(\hat{\mat X}(c_{N-1}^{m+1}) - \hat{\mat X}(c_{N-1}^{m})\right) + 2\left(\hat{\mat X}(c_{N}^{m+1}) - \hat{\mat X}(c_{N}^{m})\right)\right|^2}\frac{\left(\left(\hat{\mat X}(c_{N-1}^{m+1}) - \hat{\mat X}(c_{N-1}^{m})\right) + 2\left(\hat{\mat X}(c_{N}^{m+1}) - \hat{\mat X}(c_{N}^{m})\right)\right)}{\left|\mat X(\rho_{N}) - \mat X(\rho_{N-1})\right|} \quad &\text{for}~~j = N, \\
    -\frac{3H (c_l^m, c_l^{m+1})}{\pi \left|2\left(\hat{\mat X}(c_{0}^{m+1}) - \hat{\mat X}(c_{0}^{m})\right) + \left(\hat{\mat X}(c_{1}^{m+1}) - \hat{\mat X}(c_{1}^{m})\right)\right|^2}\frac{\left(2\left(\hat{\mat X}(c_{0}^{m+1}) - \hat{\mat X}(c_{0}^{m})\right) + \left(\hat{\mat X}(c_{1}^{m+1}) - \hat{\mat X}(c_{1}^{m})\right)\right)}{\left|\mat X(\rho_{1}) - \mat X(\rho_{0})\right|} \quad &\text{for}~~j = 0, \\
    0\qquad \qquad &\text{otherwise}.
    \end{matrix}\right.
\end{align}
Using \eqref{eqn:deltaf} to adapt the energy-stable method \eqref{eqn:full discrete}, we can obtain the following structure-preserving method. Taking $\Gamma^0 = \mat X^0 \in \mat K^{m+1}_{c_l, c_r}(\mathbb I)$ satisfying $\mat X^0 (0) = \hat{\mat X}(c_l^0)$ and $\mat X^0 (1) = \hat{\mat X}(c_r^0)$, we seek for $\mat X^{m+1} \in \mat K^{m+1}_{c_l, c_r}(\mathbb I)$ and $\mu^{m+1} \in K(\mathbb I)$, for $m \ge 0$, such that
\begin{subequations}\label{eqn:full_discrete}
\begin{align}
&\frac{1}{\Delta t_m}\left( \mat X^{m+1}-\mat X^m, \varphi ^h\mat f^{m+\frac{1}{2}}_* \right) - \left( r^m\partial_\rho\mu ^{m+1}, \partial_\rho\varphi^h\left | \partial_\rho\mat X ^m \right |^{-1} \right)=0, \qquad\forall\varphi\in K(\mathbb I),  
\label{eqn:structure2_a}\\
&\left( \mu ^{m+1}\mat f^{m+\frac{1}{2}}_*, \mat\psi ^h \right) + \left( \gamma(\theta ^{m+1}), \psi_1 ^h \left | \partial_\rho\mat X^{m+1} \right | \right) + 
\left( r^m\mat{B}(\theta ^m)\partial_\rho\mat X^{m+1},  
 \partial_\rho\mat\psi ^h\left | \partial_\rho\mat X ^m \right |^{-1}\right) \nonumber \\
 &~~~~+\frac{1}{\eta \Delta t_m}\left[(c^{m+1}_r - c^m_r) \mat G(c_r^m, c_r^{m+1})\cdot \mat \psi^h(1) + (c^{m+1}_l - c^m_l) \mat G(c_l^m, c_l^{m+1})\cdot\mat\psi^h(0)\right]
 \nonumber\\
 &~~~~-\sigma\left[\mat G(c_r^m, c_r^{m+1})\cdot \mat \psi^h(1) - \mat G(c_l^m,c_l^{m+1})\cdot\mat\psi^h(0)\right] = 0, \qquad\forall \mat \psi^h \in \mat K^{m+1}_{c_l, c_r}(\mathbb I),
\label{eqn:structure2_b}
\end{align}
\end{subequations}
where \(\mat f^{m+\frac{1}{2}}_* = \mat f^{m+\frac{1}{2}} + \delta \mat f^{m+\frac{1}{2}}\) is the corrected time-integrated approximation of $\mat f=r\,|\partial_\rho\mat X|\,\mat n$.
\begin{thm} Let $(\mat X^{m+1}, \mu^{m+1}, c_l^{m+1}, c_r^{m+1})$ be the solution of the energy-stable method \eqref{eqn:full_discrete}. Then the energy stability holds unconditionally. In addition, the volume is conservative in the sense that
\begin{align}\label{eqn:thn_DV}
    V(\mat X^{m+1}) - V(\mat X^{m}) = 0, \qquad m = 0,...,M-1.
\end{align}
\begin{proof}
    Taking \(\varphi^h = \Delta t_m \mu^{m+1}\) in \eqref{eqn:structure2_a} and \(\mat \psi^h = \mat X^{m+1} - \mat X^{m}\) in \eqref{eqn:structure2_b}, the remaining proof of energy stability follows exactly the same steps as in Theorem \ref{Energy stability}.
Next, by substituting \(\varphi^h = \Delta t_m\) in \eqref{eqn:structure2_a}, and utilizing \eqref{eqn:DFdiffenceV}, we can obtain   
\begin{align}\label{eqn:thn_DV_pf1}
    \left( \mat X^{m+1}-\mat X^m, \mat f^{m+\frac{1}{2}} \right) = -\left( \mat X^{m+1}-\mat X^m, \delta \mat f^{m+\frac{1}{2}} \right) = \frac{1}{2\pi}\left[H(c_l^m, c_l^{m+1}) - H(c_r^m, c_r^{m+1})\right].
\end{align}
Then, by using \eqref{eqn:difference V} and \eqref{eqn:thn_DV_pf1}, we can obtain \eqref{eqn:thn_DV}.
\end{proof}
\end{thm}

\section{Numerical results}\label{sec5}

In this section, we present several numerical examples of the axisymmetric SSD problem on various axisymmetric curved-surface substrates. 
In all numerical tests, we fix the material parameters as \(\eta = 100\) and \(\sigma = -\frac{\sqrt{3}}{2}\). 
For simplicity, we assume a uniform discretization in the time direction, i.e., \(\Delta t_m = \Delta t\).
Additionally, we adopt 4-fold anisotropic surface energy function given by \(\gamma(\theta) = 1 + \beta \cos(4\theta)\), where \(\beta\) represents the degree of anisotropy. 
Specifically, when \(\beta = 0\), the system is isotropic; for \(0 \le \beta \le 1/15\), the system is weakly anisotropic; and when \(\beta > 1/15\), it is strongly anisotropic. 
In numerical experiments, we use the Newton-Raphson iteration method to solve the semi-implicit scheme \eqref{eqn:full_discrete}, with a tolerance set to \(\text{tol} = 10^{-8}\).
In the numerical tests, unless otherwise stated, we primarily use initial surfaces and substrates generated by rotating the following types of curves and curved substrates (see Figure \ref{fig:5}):
\begin{itemize}
    \item [(I)] Spherical film with radius \(r = 1.5\) on a hemisphere substrate obtained by rotating a positive curvature curve: $x^2+y^2=81$; 
    \item [(II)] Axisymmetric toroidal thin film with thickness of \(0.5\) on a hemisphere substrate obtained by rotating a positive curvature curve: $x^2+y^2=81$;
    \item [(III)] Spherical film with radius \(r = 1.8\) on a sphere substrate obtained by rotating a negative curvature curve: 
    $x^2+(y-5)^2=25$.
\end{itemize}
\begin{figure}[!ht]
    \centering
    \includegraphics[width=0.3\linewidth]{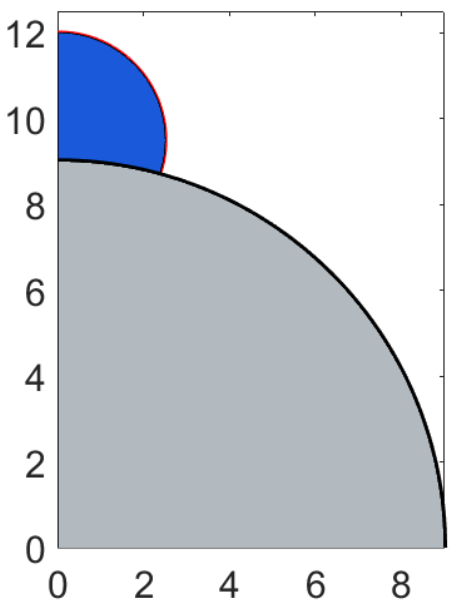}
    ~~~
    \includegraphics[width=0.305\linewidth]{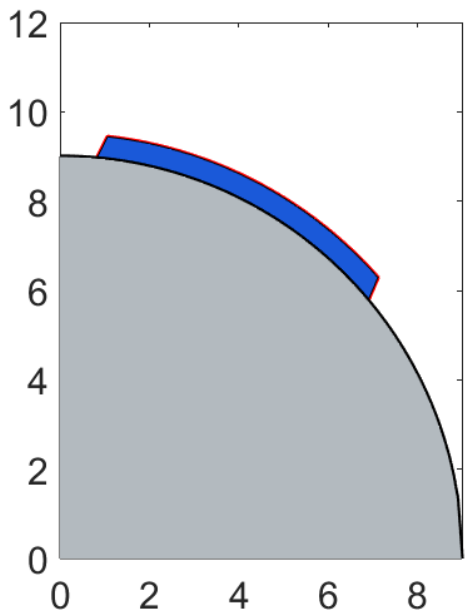}
    ~~~
    \includegraphics[width=0.225\linewidth]{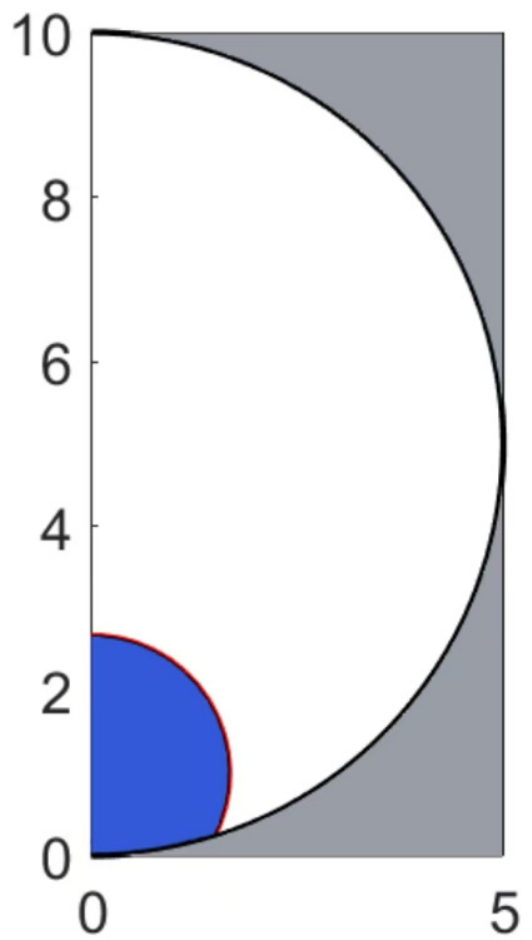}
    \caption{Three types of initial generated curves/curved substrates.}
    \label{fig:5}
\end{figure}

\textbf{Example 1} (Convergence tests) 
We begin by testing the convergence of the structure-preserving method \eqref{eqn:full_discrete} for isotropic, weakly anisotropic, and strongly anisotropic cases.
Based on the numerical solutions \(\{\mat X ^m\}_{0 \le m \le M}\) to the numerical scheme with the mesh size \(h\) and the time step \(\Delta t\), we define the numerical approximation solution as follows:
\begin{align}
    \mat X_{h,\Delta t}(t,\rho_j) := \frac{t-t_m}{\Delta t}\mat X^{m+1}(\rho_j) + \frac{t_{m+1}-t}{\Delta t}\mat X^{m}(\rho_j), \quad \forall j = 1,2,...,N-1.
\end{align}
If there are two contact points on the curved substrate, the boundary conditions are given by:
\begin{align}
    \mat X_{h,\Delta t}(t,\rho_0) := \hat{\mat X} \left(\frac{t-t_m}{\Delta t}c^{m+1}_l + \frac{t_{m+1}-t}{\Delta t}c^{m}_l\right), \quad \mat X_{h,\Delta t}(t,\rho_N) := \hat{\mat X} \left(\frac{t-t_m}{\Delta t}c^{m+1}_r + \frac{t_{m+1}-t}{\Delta t}c^{m}_r\right).
\end{align}
The boundary satisfies \(\mat X_{h,\Delta t}(t,\rho_N) := \hat{\mat X} \left(\frac{t-t_m}{\Delta t}c^{m+1}_r + \frac{t_{m+1}-t}{\Delta t}c^{m}_r\right)\) if only one outer contact point is on the curved substrate. 
We calculate the numercial error \(e^{h,\Delta t}(t)\) by comparing \(\mat X_{h,\Delta t}(t)\) with the reference solution \(\mat X_r(t)\), using the manifold distance, defined by  \cite{Zhao20} 
\begin{equation*}
e^{h,\Delta t}(t) := \text{Md}(\mat X_{h,\Delta t}(t), \mat X_r(t)) = \left |(\Omega_{h,\Delta t}(t)\backslash\Omega_r(t))\cup(\Omega_r(t)\backslash\Omega_{h,\Delta t}(t)) \right | = \left |\Omega_{h,\Delta t}(t) \right |+\left |\Omega_r(t) \right |-2\left |\Omega_{h,\Delta t}(t)\cap\Omega_r(t) \right |, \nonumber
\end{equation*}
where $\Omega_i$, $i=\{h,\Delta t\}, r$ denotes the region enclosed by $\mat X_i(t)$, and $| \cdot |$ represents the area of the region.

Figures \ref{fig:6}–\ref{fig:7} illustrate the numerical errors and the corresponding convergence order of the structure-preserving algorithm under various anisotropic strengths. Furthermore, Figure \ref{fig:8} shows the numerical errors and the order at different times for a fixed \(\beta\). From these figures, we observe that the convergence rate with respect to the mesh size \(h\) is second-order, which is in agreement with the expected result.

\begin{figure}[!ht]
    \centering
    \includegraphics[width=0.4\linewidth]{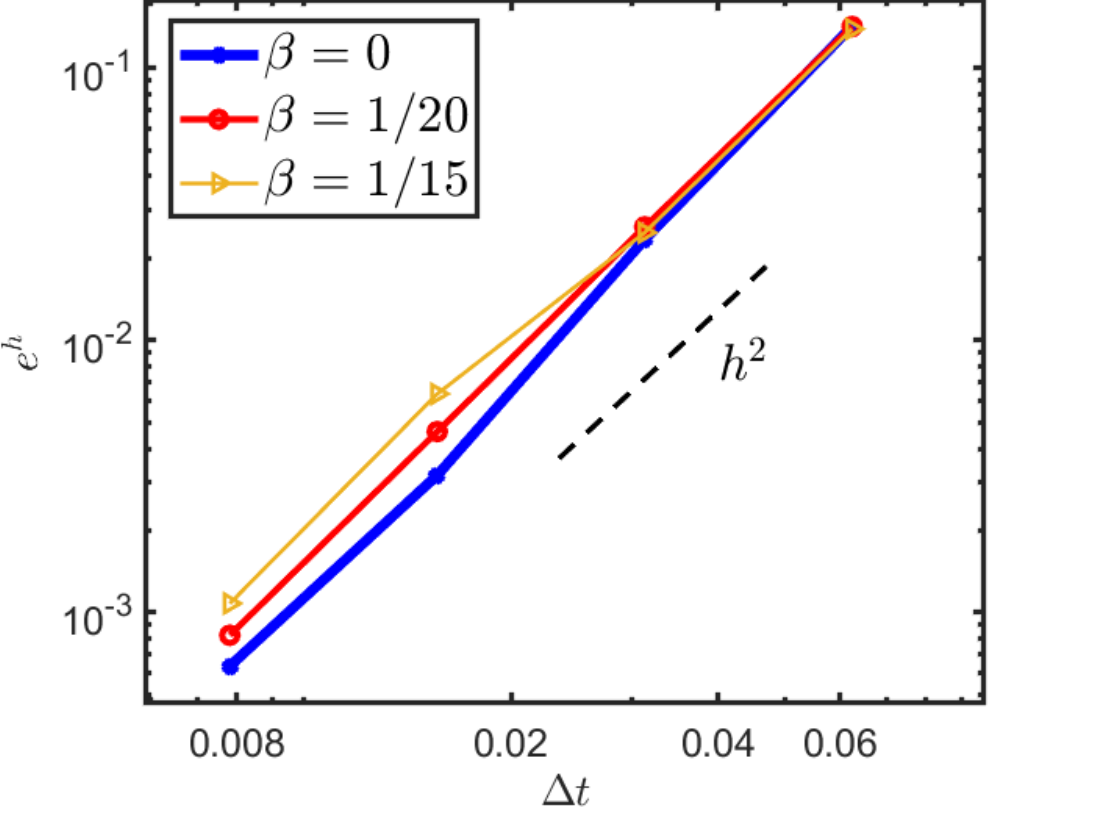}
    ~~~
    \includegraphics[width=0.4\linewidth]{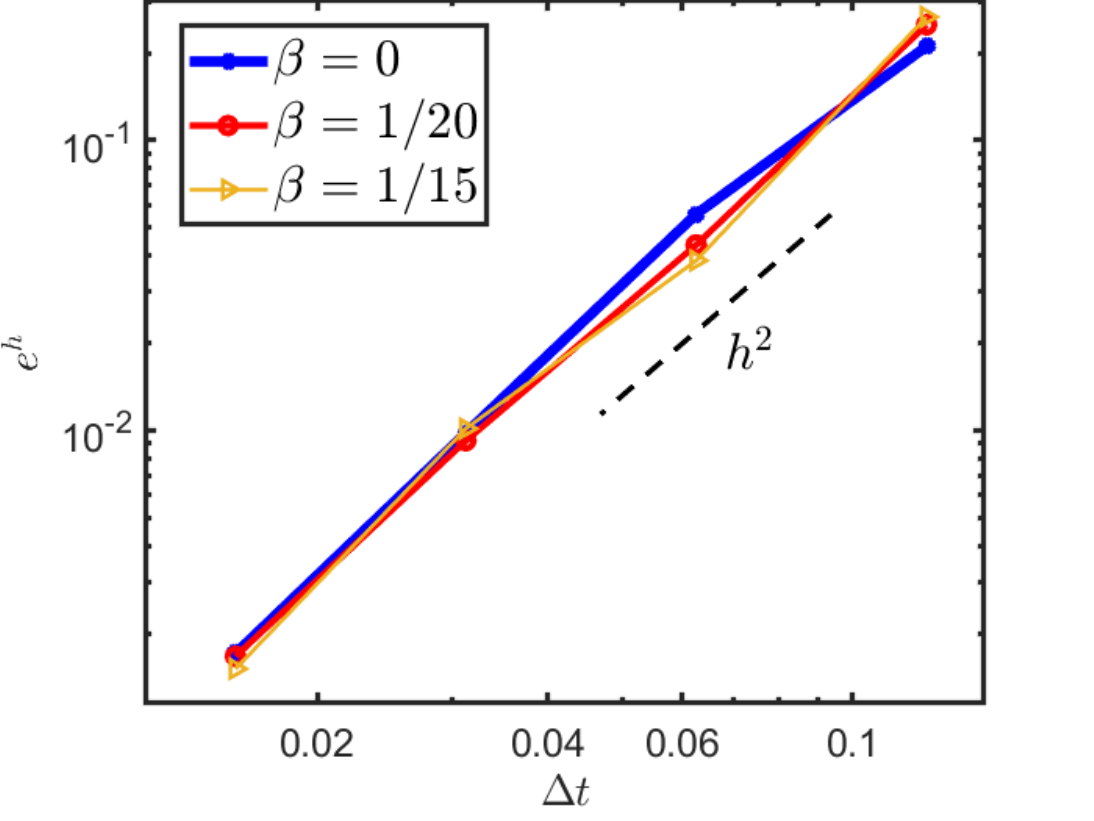}
    \caption{Convergence tests with isotropy and weak anisotropy at \(T = 1\): Case I (left panel); Case II (right panel).}
    \label{fig:6}
\end{figure}

\begin{figure}[!ht]
    \centering
    \includegraphics[width=0.4\linewidth]{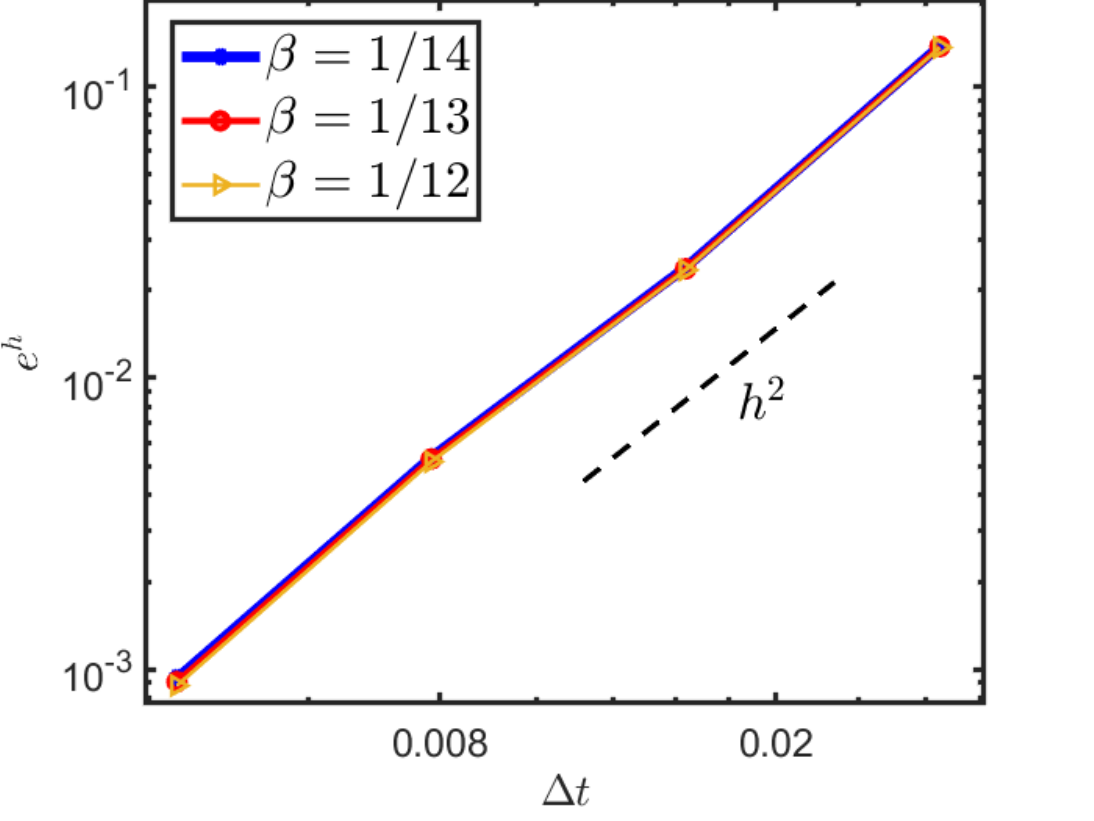}
    ~~~
    \includegraphics[width=0.4\linewidth]{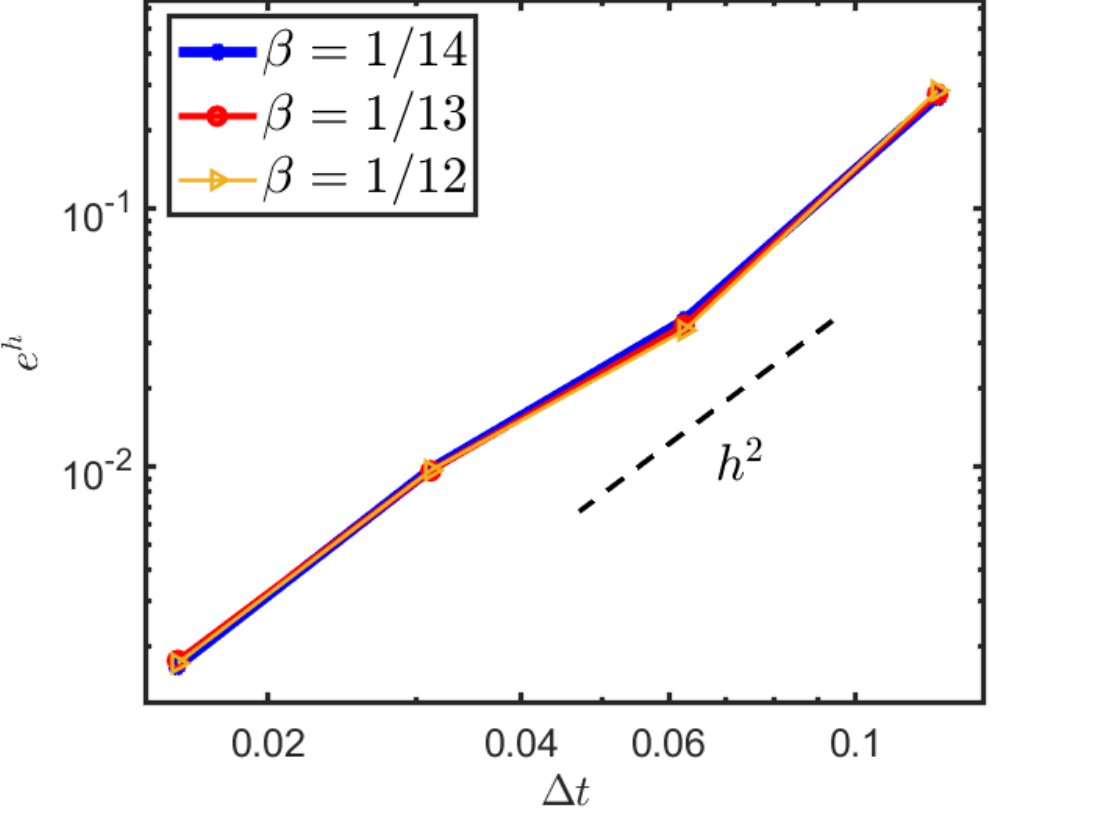}
    \caption{Convergence tests with strong anisotropy at \(T = 1\): Case I (left panel); Case II (right panel).}
    \label{fig:7}
\end{figure}

\begin{figure}[!ht]
    \centering
    \includegraphics[width=0.4\linewidth]{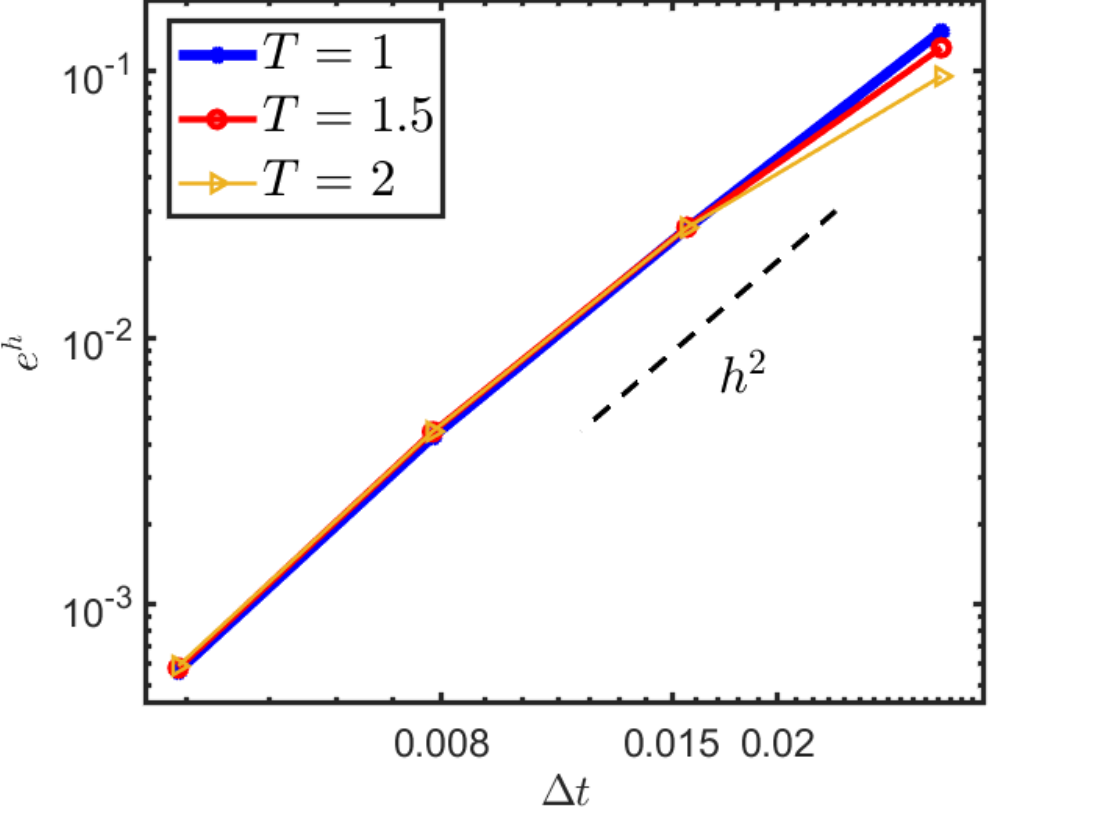}
    ~~~
    \includegraphics[width=0.4\linewidth]{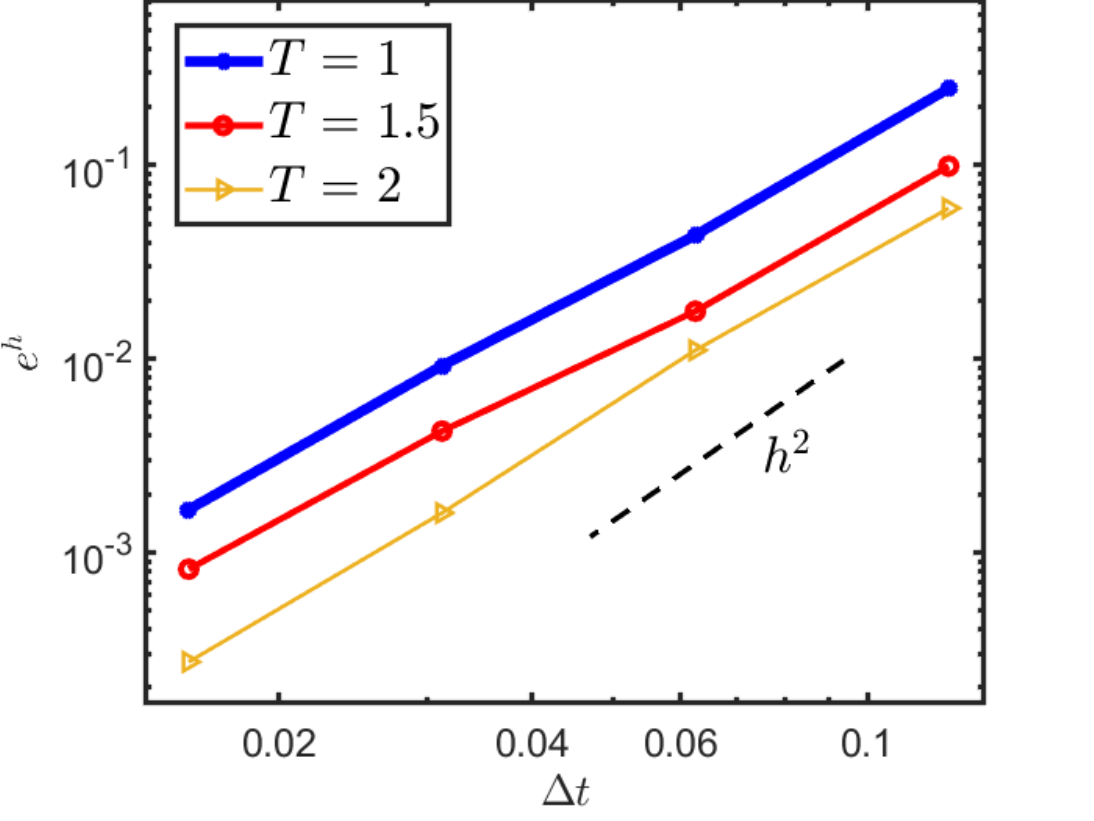}
    \caption{Convergence tests with \(\beta = \frac{1}{20}\) for different times \(T\): Case I (left panel); Case II (right panel).}
    \label{fig:8}
\end{figure}

\textbf{Example 2} (Energy stability \& Volume conservation) 
In this example, we focus on examining the energy stability and volume conservation of the structure-preserving method \eqref{eqn:full_discrete}. 
For a fixed value of \(\beta\), Figure \ref{fig:9} shows that the discrete energy decays monotonically over different time steps.
Figure \ref{fig:10} illustrates that the discrete energy maintains stability in the isotropic, weakly anisotropic, and strongly anisotropic cases.
In addition, we test the evolution of the relative volume error for the two numerical methods.
We can observe from Figure \ref{fig:11} that, after adjustment, the structure-preserving method can effectively maintain volume conservation, whereas the energy-stable method does not guarantee volume conservation.
\begin{figure}[!ht]
    \centering
    \includegraphics[width=0.4\linewidth]{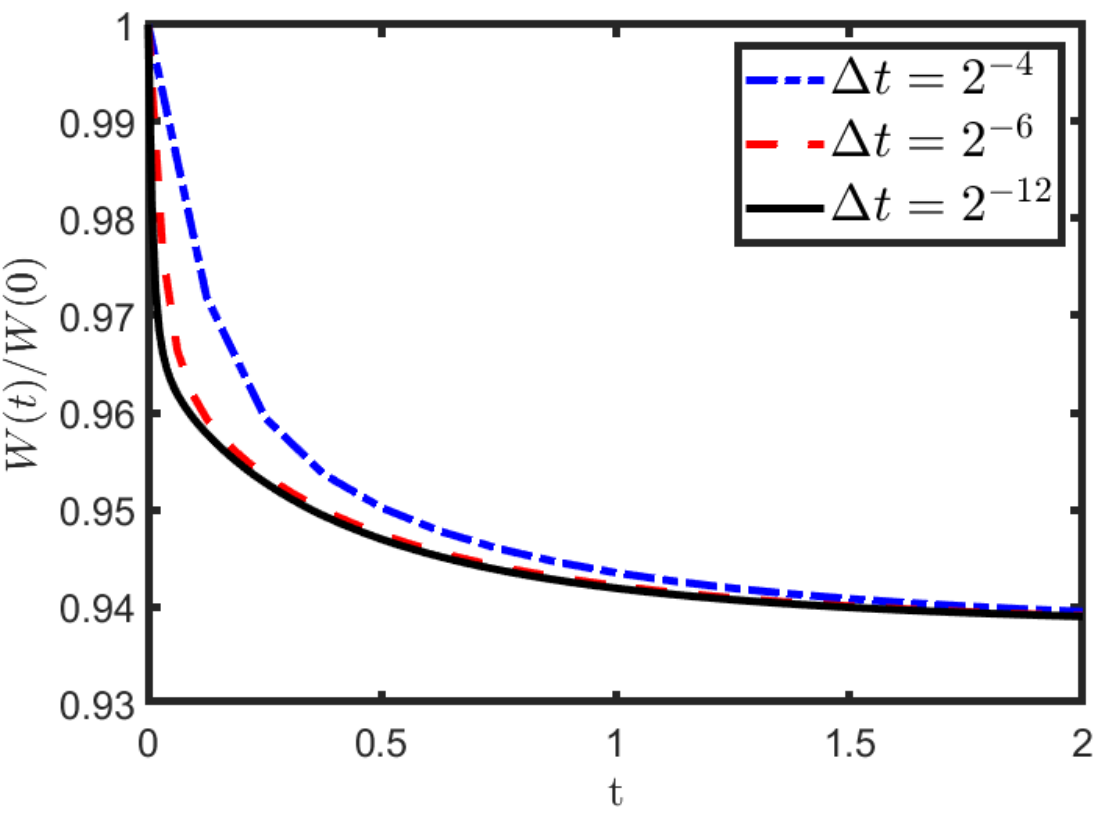}
    ~~~
    \includegraphics[width=0.4\linewidth]{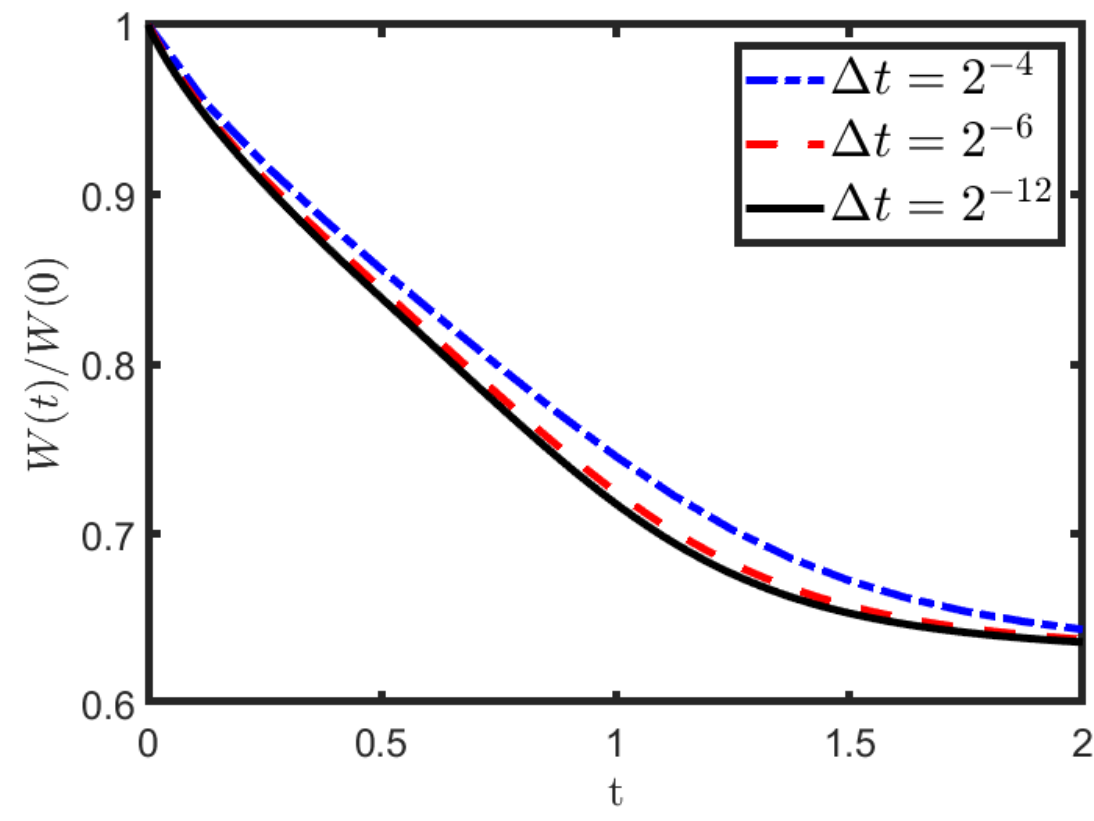}
    \caption{Time history of normalized energy of structure-preserving method under different time steps with \(\beta = \frac{1}{20}\), \(h = 2^{-7}\) and \(T = 2\): Case I (left panel); Case II (right panel).}
    \label{fig:9}
\end{figure}
\begin{figure}[!ht]
    \centering
    \includegraphics[width=0.4\linewidth]{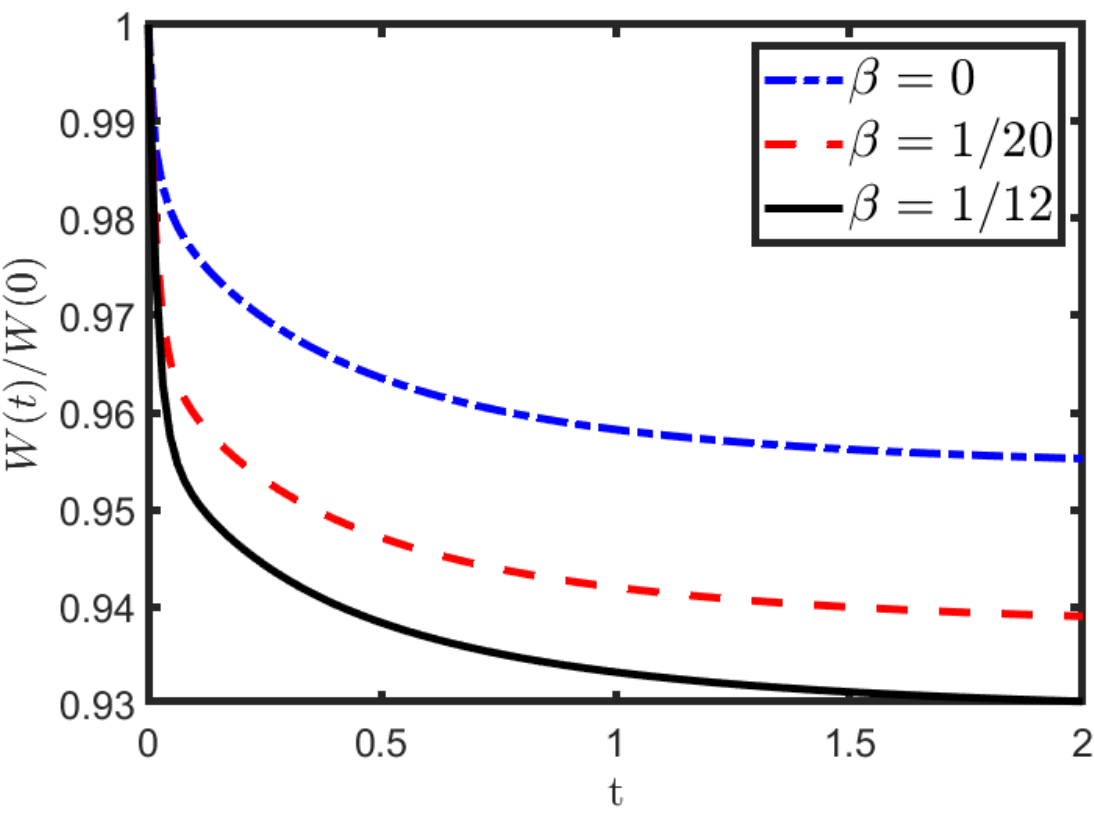}
    ~~~
    \includegraphics[width=0.4\linewidth]{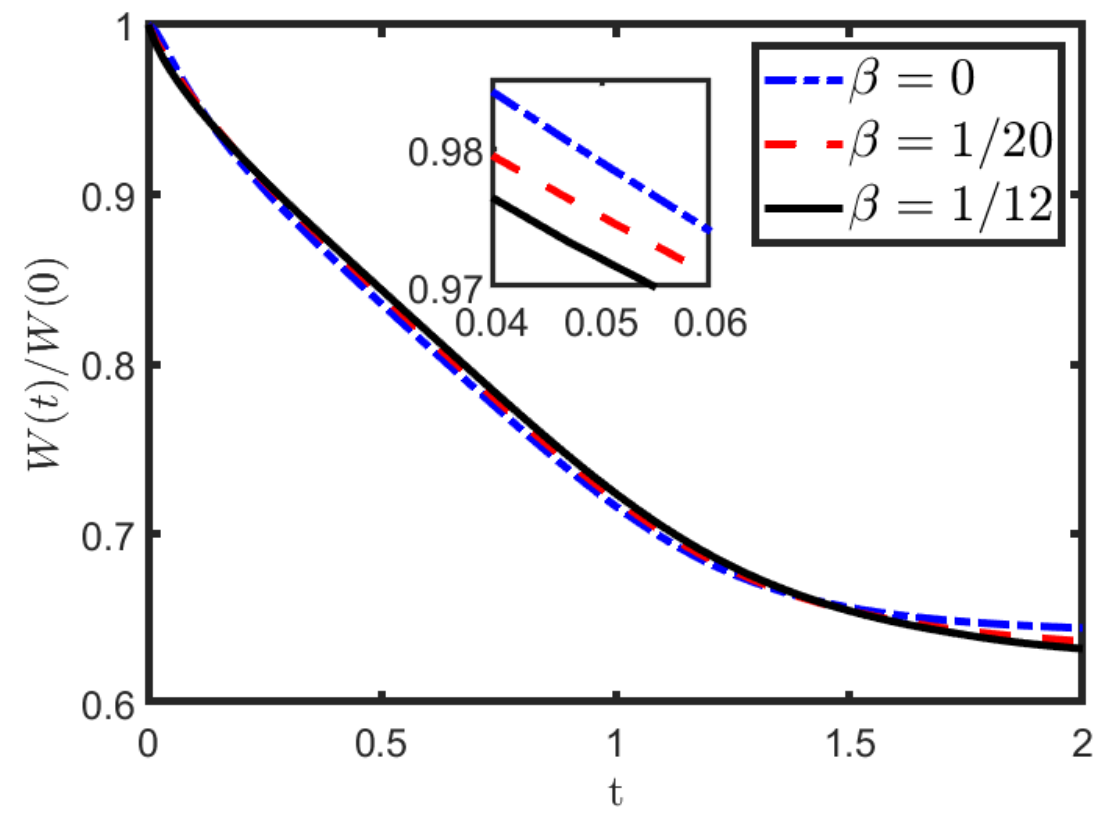}
    \caption{Time history of normalized energy of structure-preserving method under isotropic and anisotropic conditions with \(\Delta t = 2^{-9}\), \(h = 2^{-7}\) and \(T = 2\): Case I (left panel); Case II (right panel).}
    \label{fig:10}
\end{figure}
\begin{figure}[!ht]
    \centering
    \includegraphics[width=0.3\linewidth]{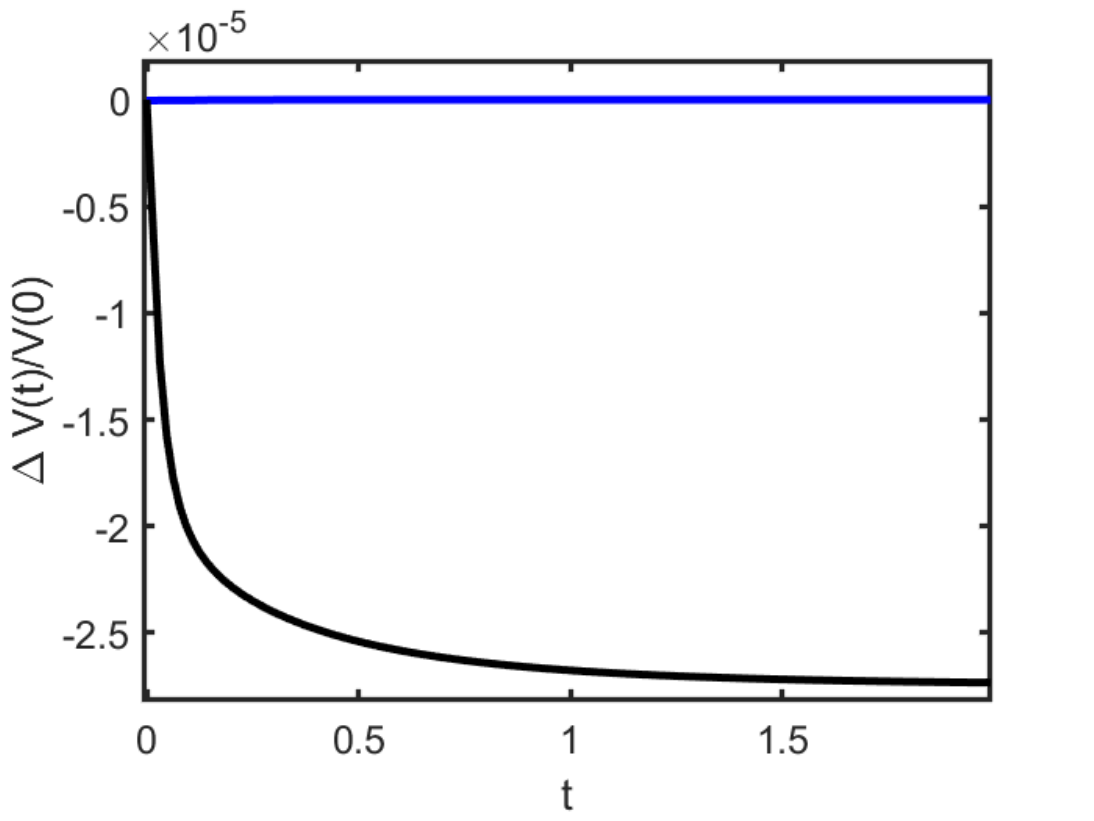}
    ~~
    \includegraphics[width=0.3\linewidth]{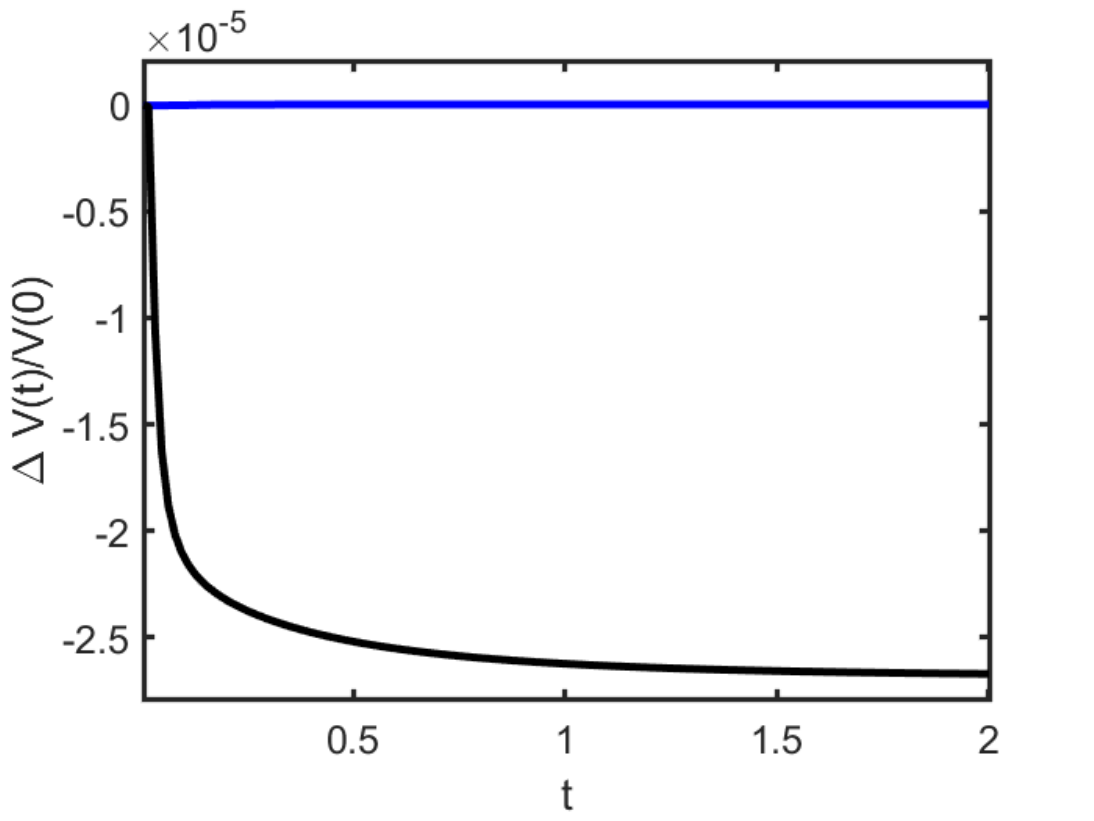}
    ~~
    \includegraphics[width=0.3\linewidth]{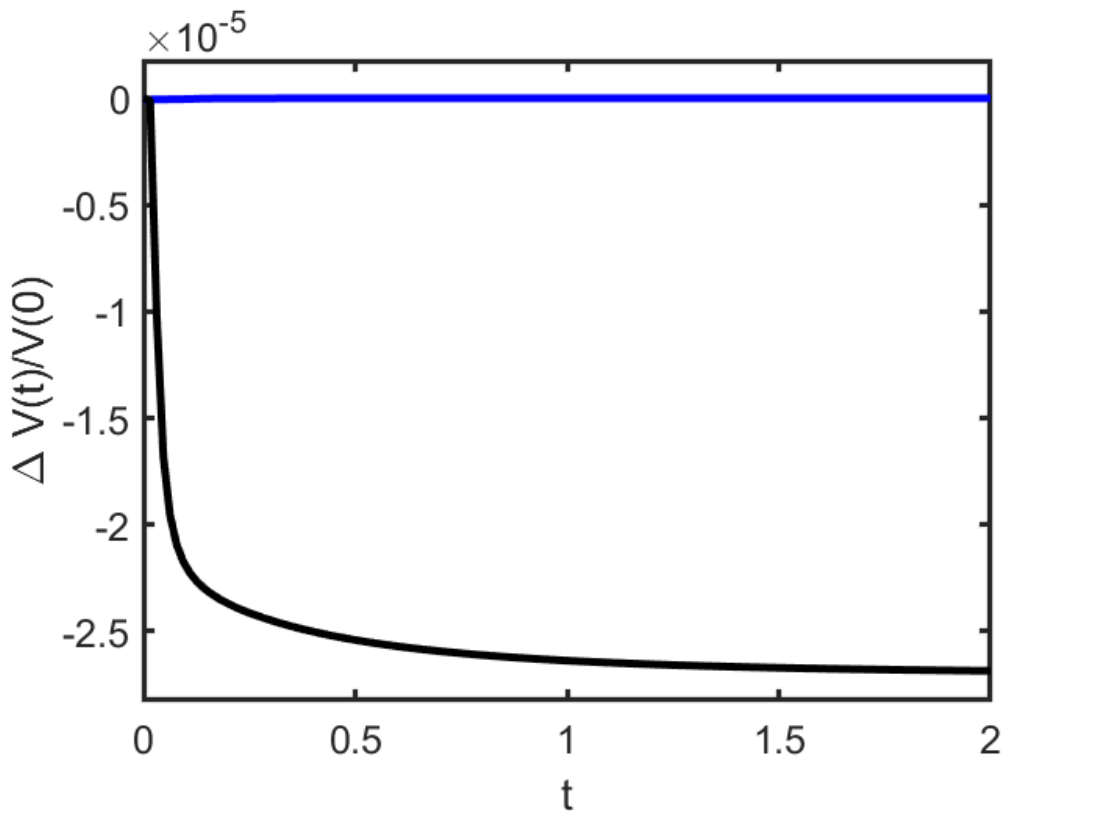}
    \\ \vspace{5pt}
    \includegraphics[width=0.3\linewidth]{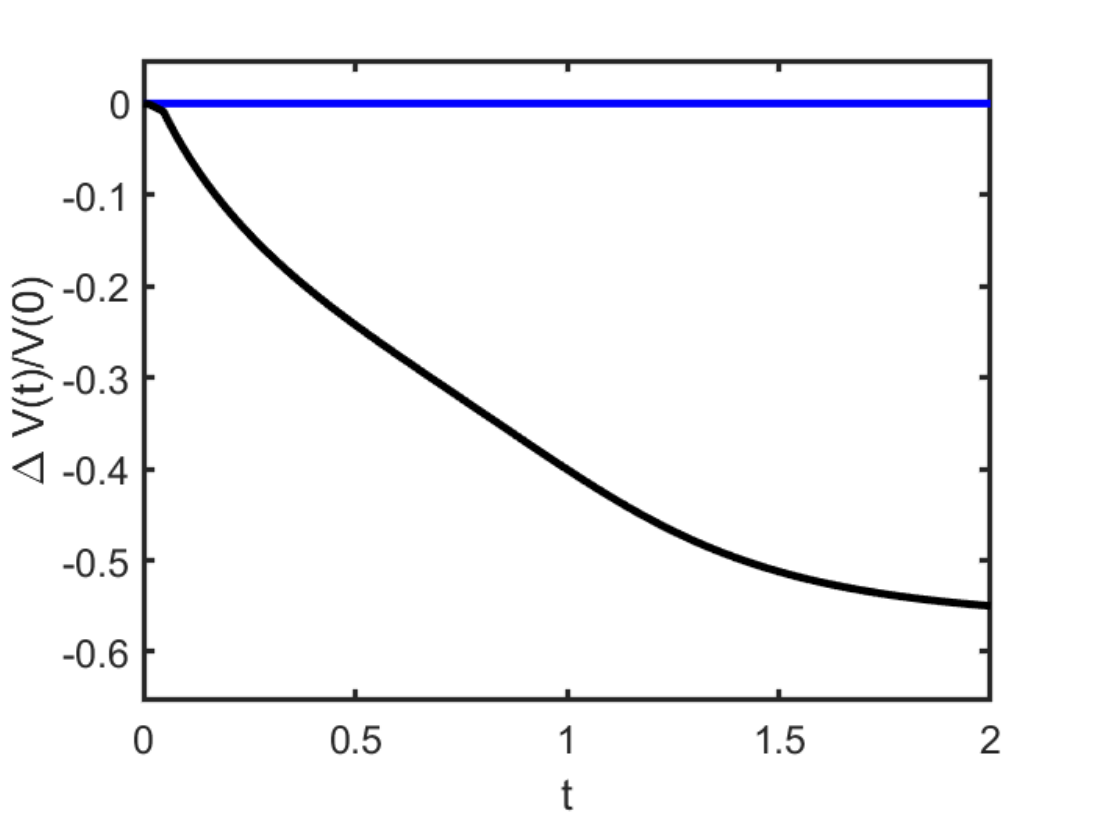}
    ~~
    \includegraphics[width=0.3\linewidth]{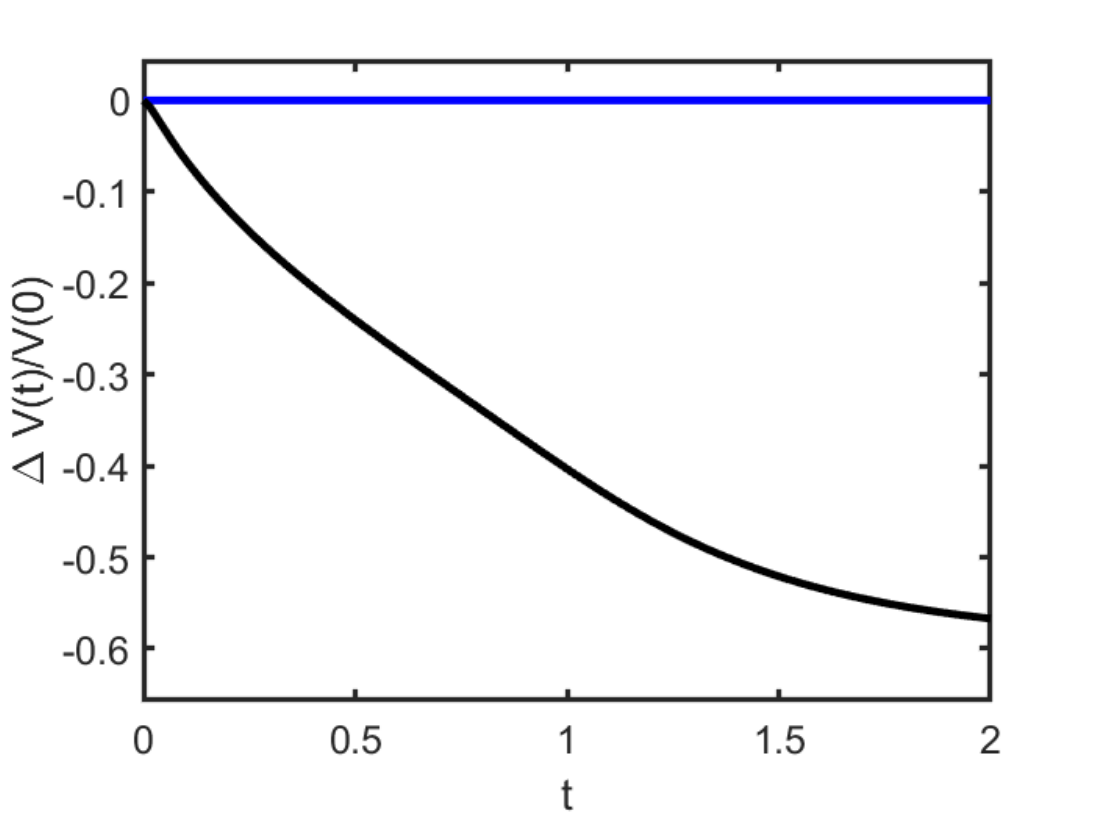}
    ~~
    \includegraphics[width=0.3\linewidth]{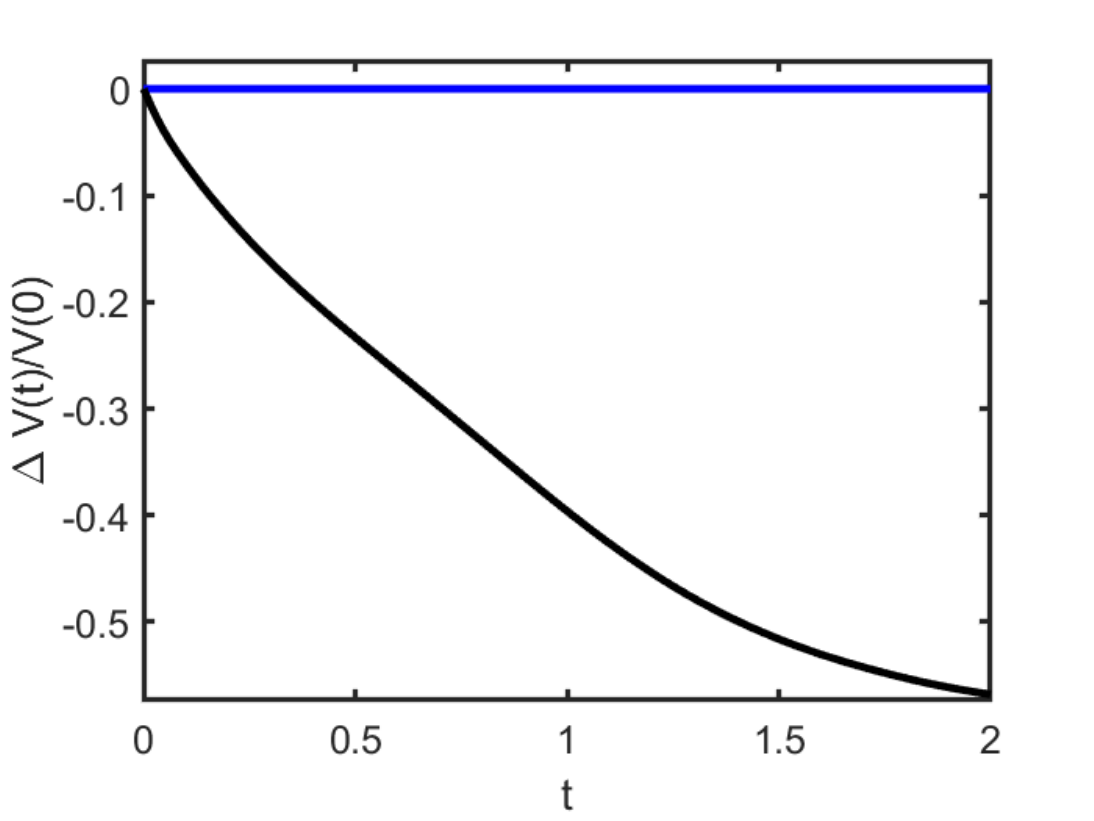}
    \caption{Relative volume errors of energy-stable method and structure-preserving method with \(\Delta t = 2^{-9}\), \(h = 2^{-7}\) and \(T = 2\): Case I (upper panel); Case II (lower panel). The black/blue line represents the relative volume evolution of energy-stable/structure-preserving method. From left to right, the cases correspond to $\beta=0, \frac{1}{20}, \frac{1}{12}$, representing isotropy, weak anisotropy, and strong anisotropy, respectively.}
    \label{fig:11}
\end{figure}

\textbf{Example 3} 
In this example, we numerically simulate the evolution of particles/islands on a larger substrate. Zhao et al. \cite{ZHAO2024120407} studied local approximations of the two-dimensional substrate/particle interaction for substrates with positive and negative curvature. In this study, we develop their work further by extending it to three dimensions and advancing from the isotropic to the weakly/strongly anisotropic cases. We specify that equilibrium means that the energy difference between two adjacent steps reaches \(10^{-8}\). As shown in Figures \ref{fig:12}-\ref{fig:13}, we present the equilibrium states evolving on curved-surface substrates (generated by curves with positive/negative curvature) under various anisotropies, while also demonstrating energy stability and volume conservation.
\begin{figure}
    \centering
    \includegraphics[width=0.25\linewidth]{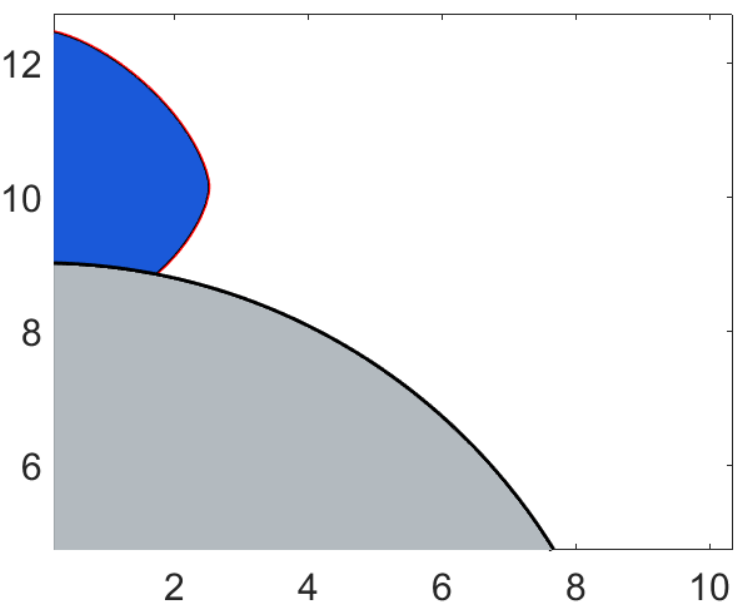}
    ~~~
     \includegraphics[width=0.25\linewidth]{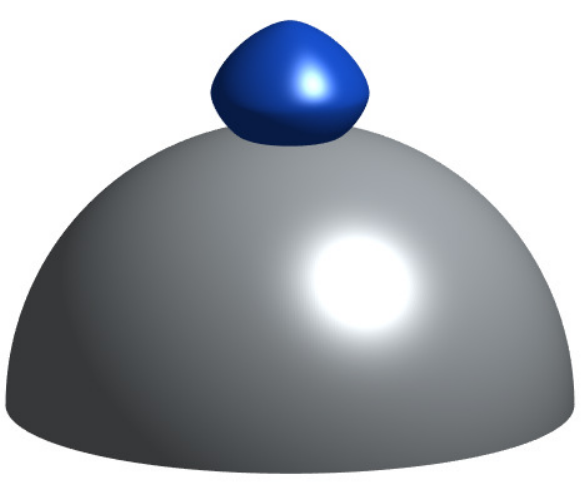}
     ~~~
     \includegraphics[width=0.3\linewidth]{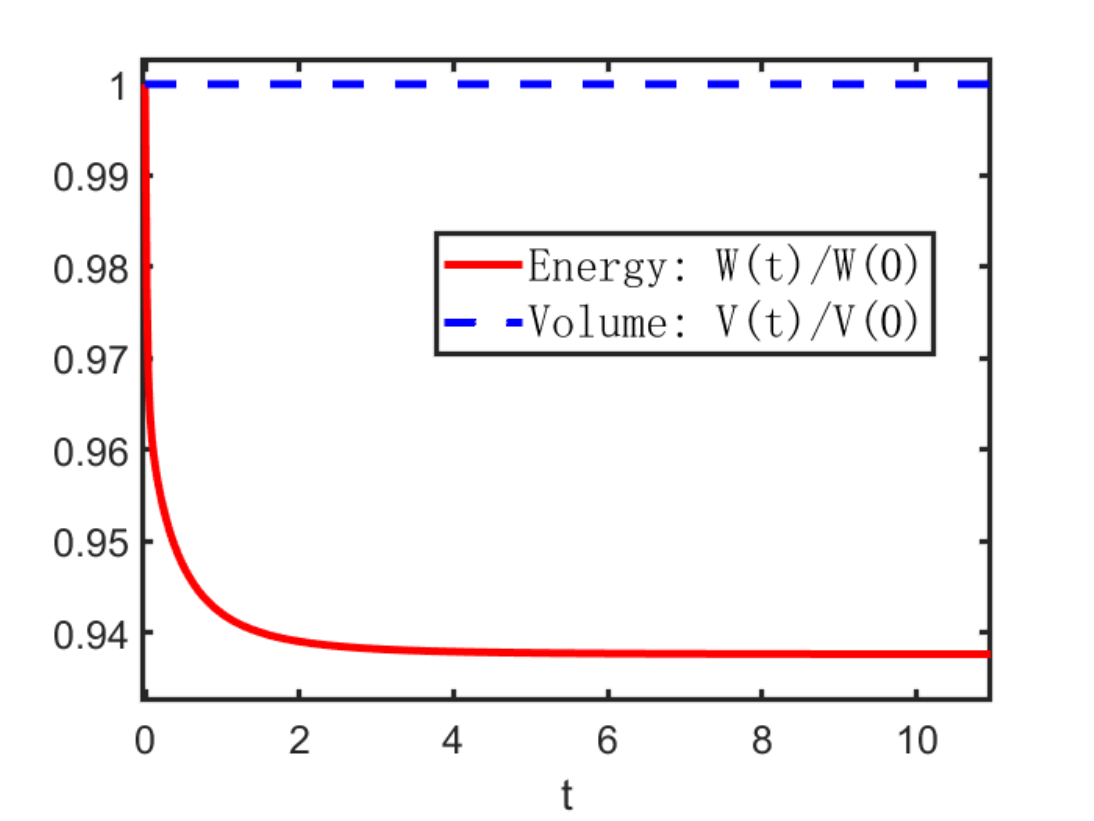}
    \\ \vspace{5pt}
    \includegraphics[width=0.25\linewidth]{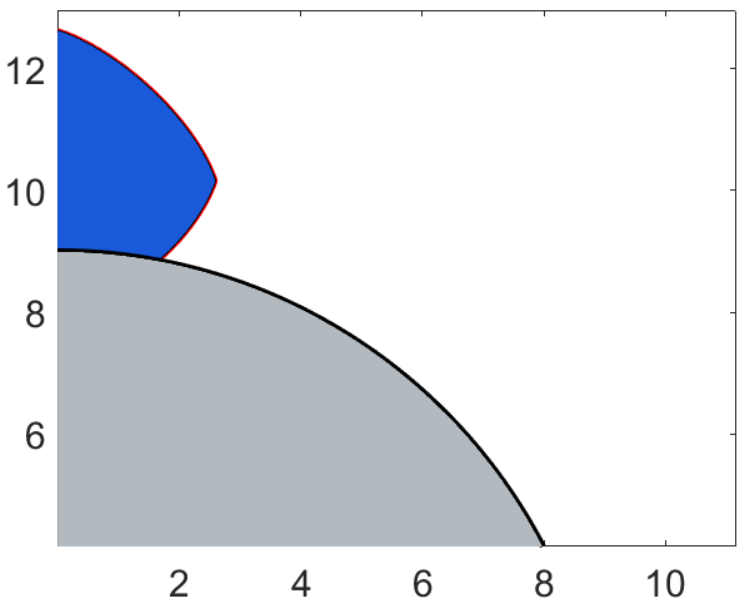}
    ~~~
    \includegraphics[width=0.25\linewidth]{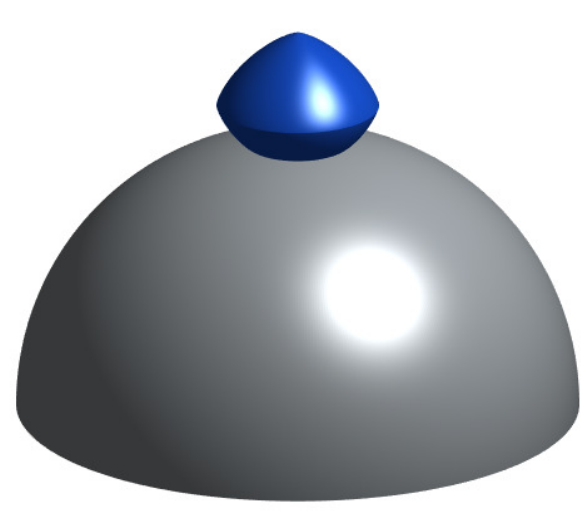}
    ~~~
    \includegraphics[width=0.3\linewidth]{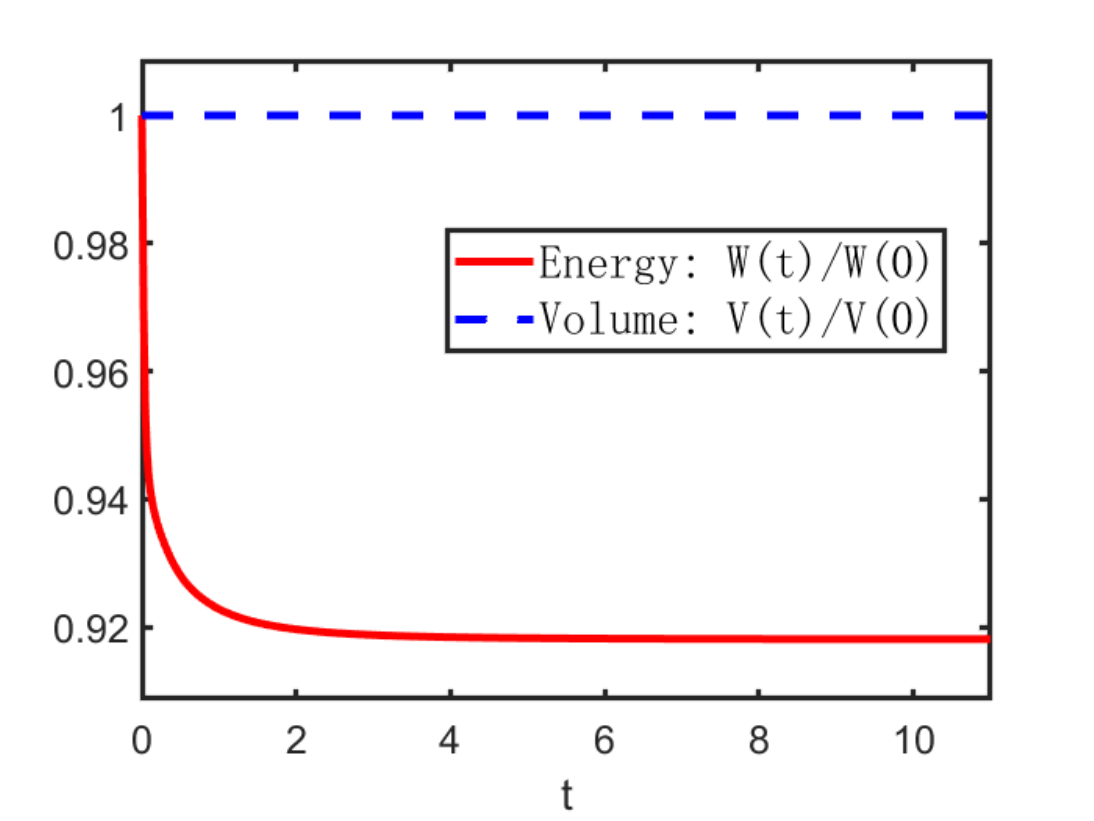}
    \caption{Evolution equilibrium states, starting with Case I.
    The left panel shows the equilibrium shape along the x-axis, while the middle panel visualizes the equilibrium shapes in three dimensions. Energy stability and volume conservation are presented in the right panel.
    The parameters are selected as \(\Delta t = 2^{-9}\), \(h = 2^{-7}\), \(T = 11\), and \(\beta = \frac{1}{20}\) (upper panel), \(\beta = \frac{1}{12}\) (lower panel). 
}
    \label{fig:12}
\end{figure}
\begin{figure}[!ht]
    \centering
    \includegraphics[width=0.25\linewidth]{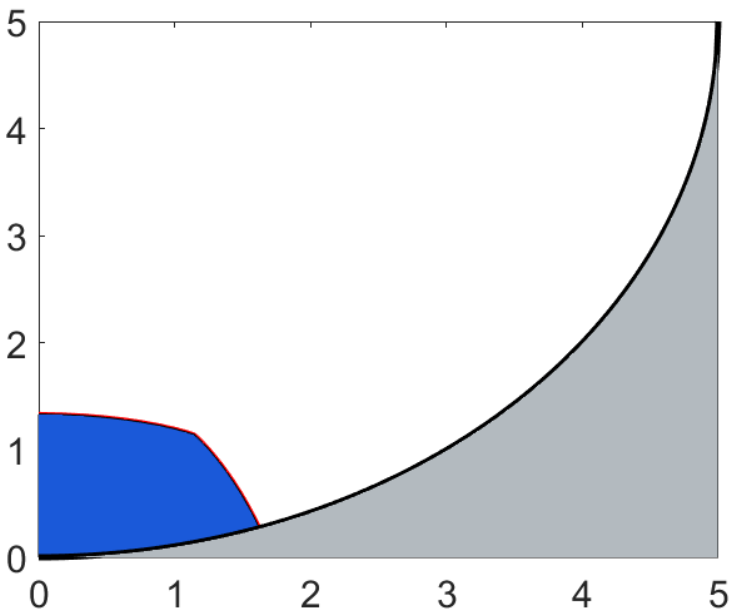}
    ~~
    \includegraphics[width=0.25\linewidth]{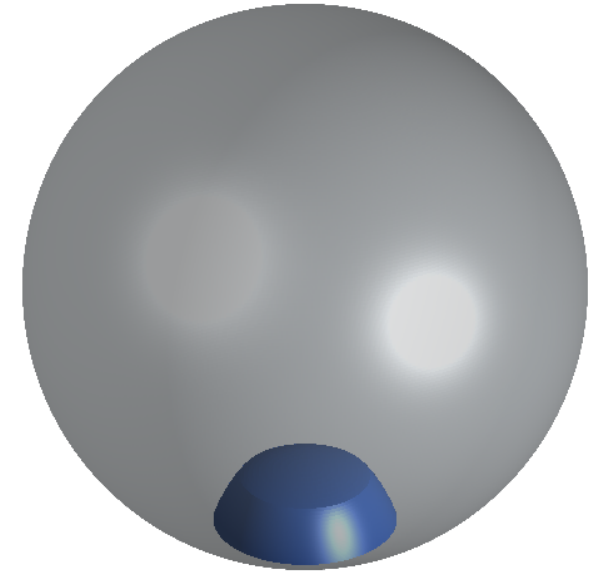}
    ~~~
    \includegraphics[width=0.3\linewidth]{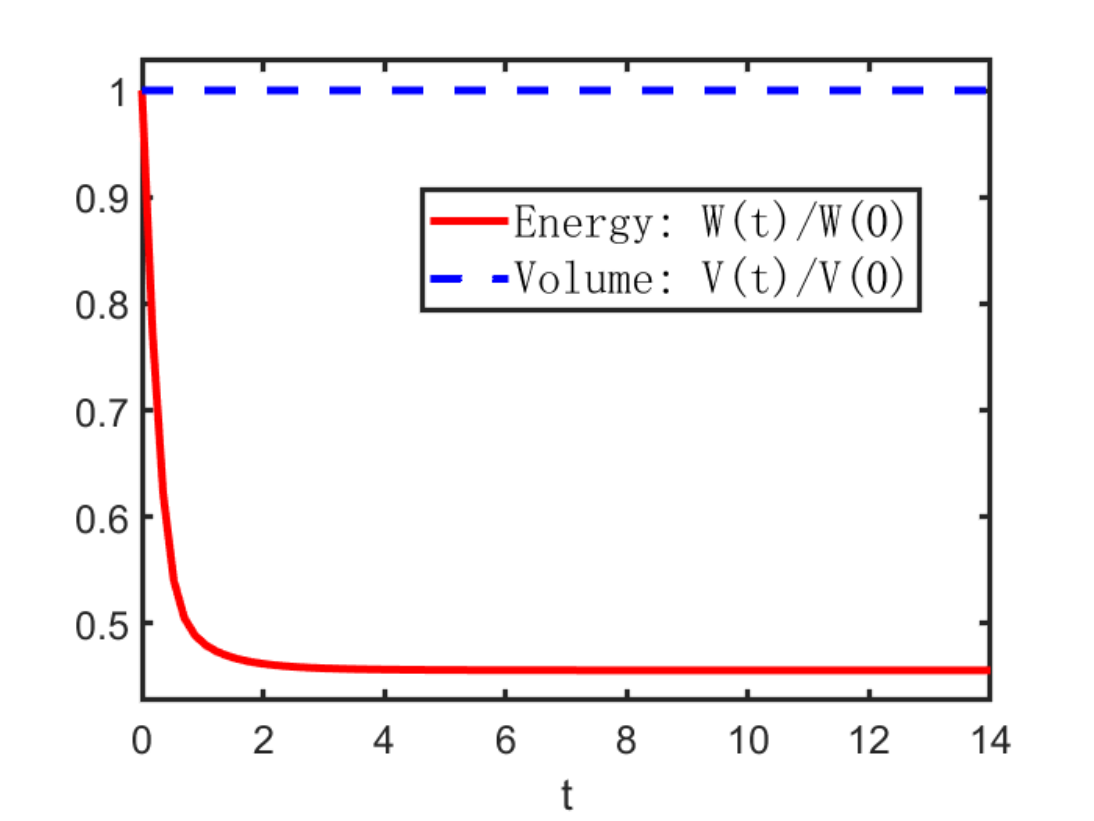}
    \\ \vspace{5pt}
    \includegraphics[width=0.25\linewidth]{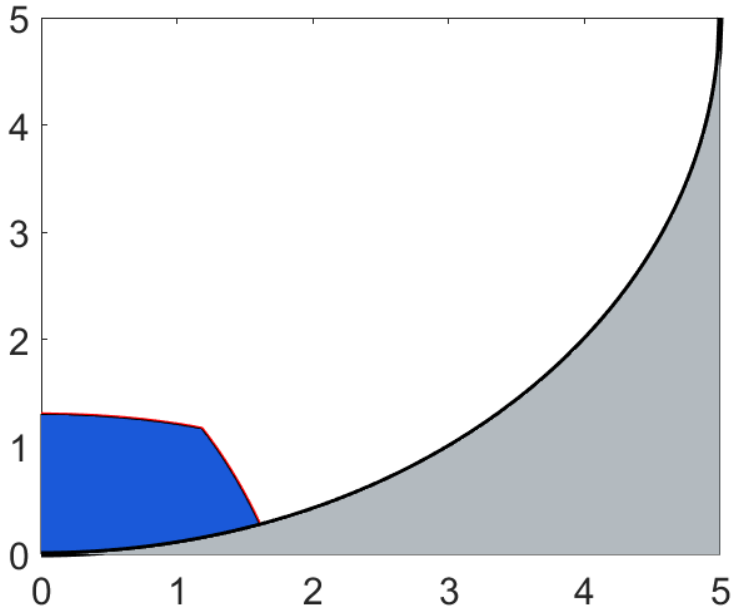}
    ~~
    \includegraphics[width=0.25\linewidth]{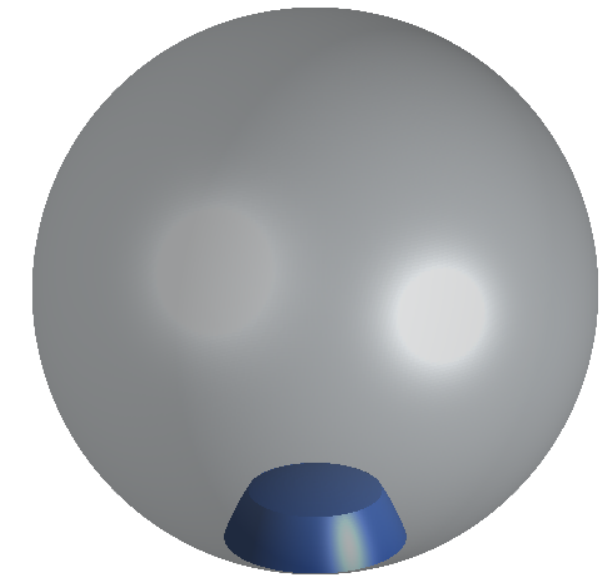}
    ~~~
    \includegraphics[width=0.3\linewidth]{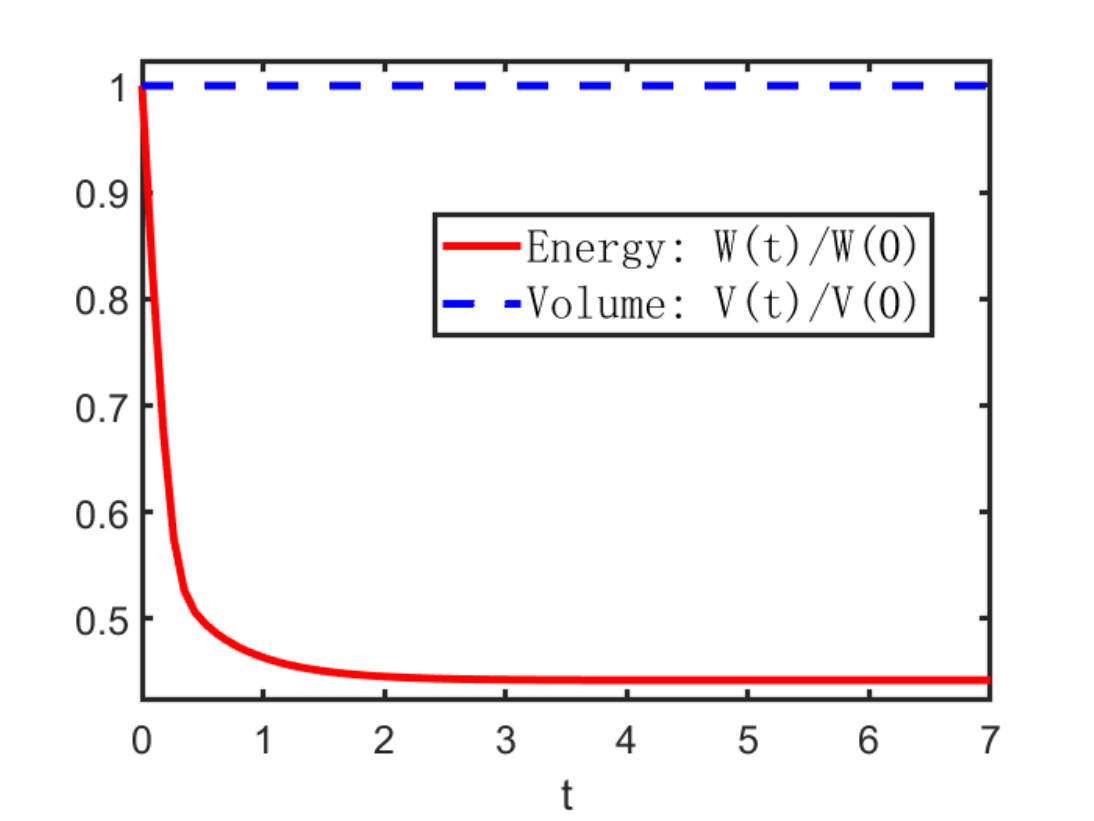}
    \caption{Evolution equilibrium states, starting with Case III. 
    The left panel shows the generated curve along with the curved substrate for the equilibrium shape. The middle panel visualizes the axisymmetric surfaces reaching the equilibrium state. Energy stability and volume conservation are depicted in the right panel.
    The parameters are selected as \(\Delta t = 2^{-9}\), \(h = 2^{-7}\), and \(T = 14,~\beta = \frac{1}{20}\) (upper panel), \(T = 7,~\beta = \frac{1}{12}\) (lower panel).  
}
    \label{fig:13}
\end{figure}

\textbf{Example 4} In this example, we study the evolution of the axisymmetric toroidal thin film on a more general axisymmetric curved-surface substrate, generated by a sinusoidal curve. 
We first focus on the case where the length of the thin film is close to the period of the generating sinusoidal curve: \(y = 0.2 sin(\pi x)\). We observe from Figures \ref{fig:16}-\ref{fig:17} that the anisotropy strength significantly affects the evolution rate of the thin film. In the case of weak anisotropy, the film continuously moves to the left, and the hole in the middle gradually disappears until it reaches a steady state. As the anisotropy strength increases, in the case of strong anisotropy, the film evolves to the steady state with minimal movement to the left.
We then focus on the scenario where the length of the thin film is much shorter than the period of the sinusoidal curve: \(y = 4 sin\frac{x}{4}\).
Figures \ref{fig:18}-\ref{fig:19} demonstrate that the thin film moves in the direction of lower curvature. Furthermore, as anisotropy increases, the rate of movement slows down.
\begin{figure}[!ht]
    \centering
    \includegraphics[width=0.28\linewidth]{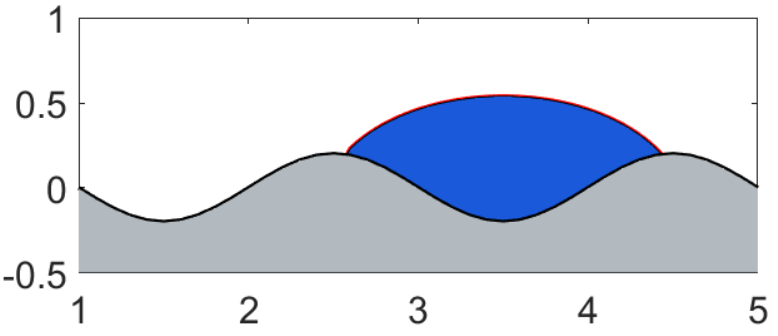}
    ~~~~
    \includegraphics[width=0.28\linewidth]{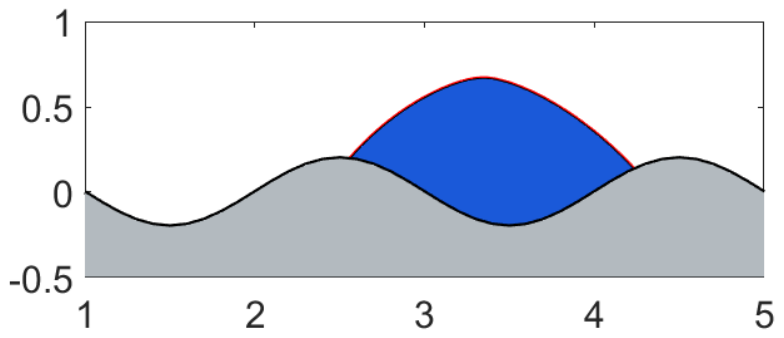}
    ~~~~
    \includegraphics[width=0.28\linewidth]{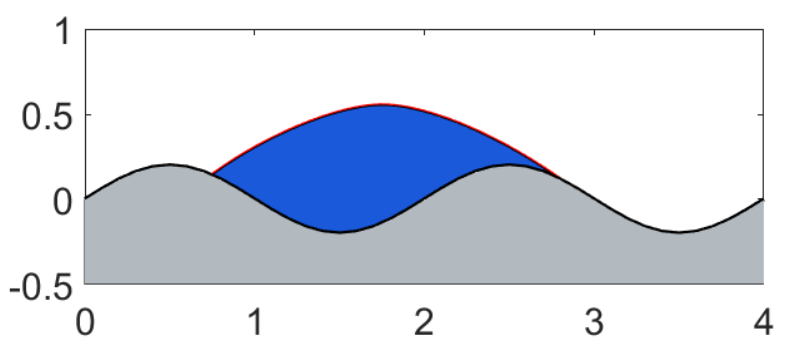}
    \\ \vspace{5pt}
    \includegraphics[width=0.83\linewidth]{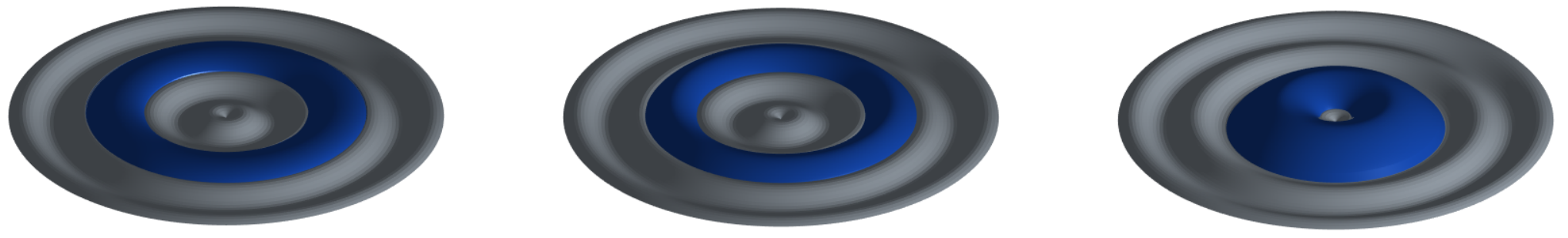}
    \caption{Evolution of 
    an axisymmetric thin film on an axisymmetric curved substrate generated by a sinusoidal curve,
    with \(\Delta t = 2^{-9}, h = 2^{-7}\) and \(\beta = \frac{1}{20}\): the generated curves $\Gamma^m$ at times \(t = 0, 0.55, 5\) (upper panel), the visualization of the corresponding axisymmetric surfaces (lower panel).}
    \label{fig:16}
\end{figure}

\begin{figure}[!ht]
    \centering
    \includegraphics[width=0.28\linewidth]{sin1_2-eps-converted-to.pdf}
    ~~~~
    \includegraphics[width=0.28\linewidth]{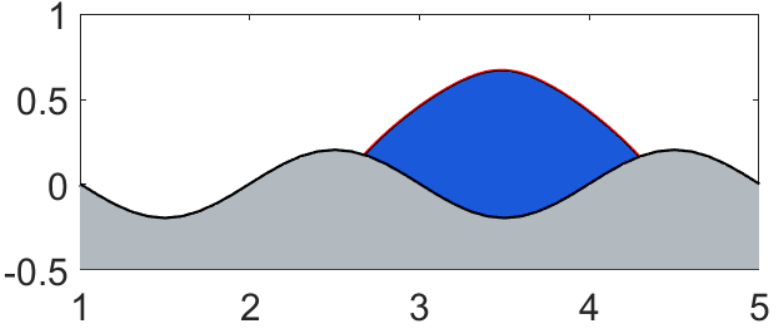}
    ~~~~
    \includegraphics[width=0.28\linewidth]{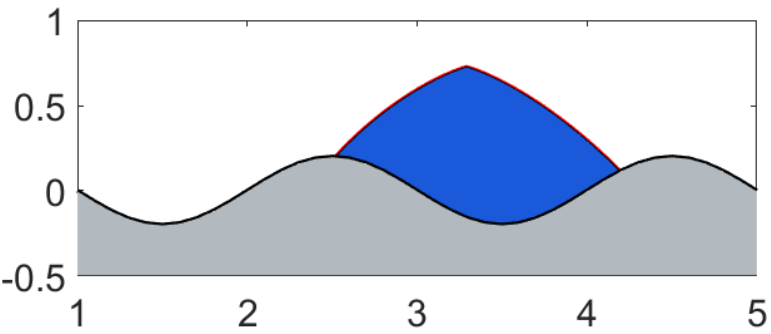}
    \\ \vspace{5pt}
    \includegraphics[width=0.83\linewidth]{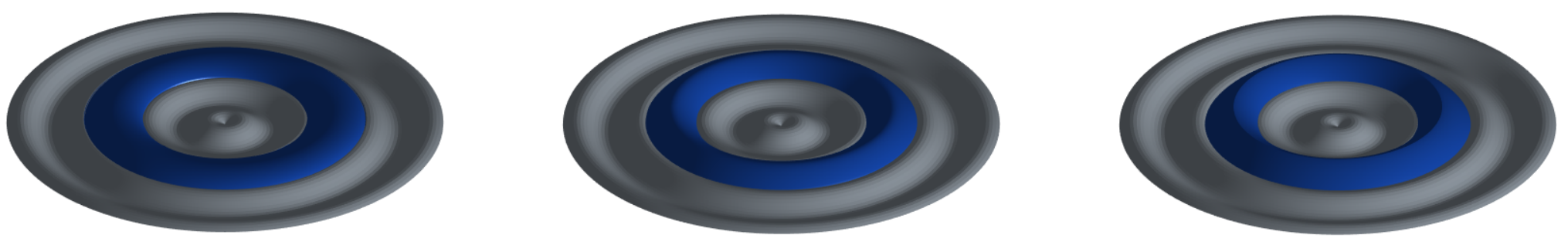}
    \caption{
    Evolution of 
    an axisymmetric thin film on an axisymmetric curved substrate generated by a sinusoidal curve,
    with \(\Delta t = 2^{-9}, h = 2^{-7}\) and \(\beta = \frac{1}{12}\): the generated curves $\Gamma^m$ at times \(t = 0, 0.55, 7\) (upper panel), the visualization of the corresponding axisymmetric surfaces (lower panel).}
    \label{fig:17}
\end{figure}
\begin{figure}[!ht]
    \centering
    \includegraphics[width=0.28\linewidth]{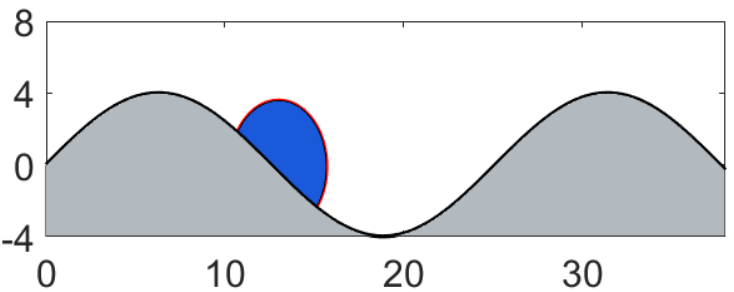}
    ~~~
    \includegraphics[width=0.28\linewidth]{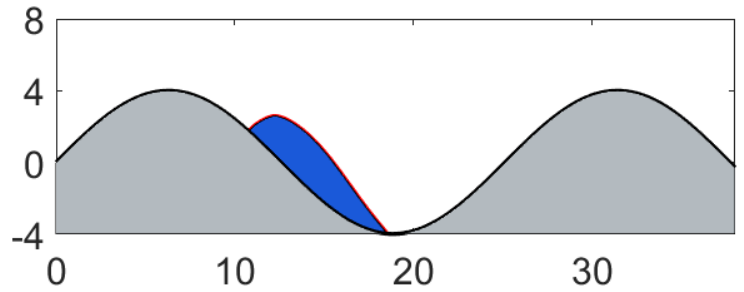}
    ~~~
    \includegraphics[width=0.28\linewidth]{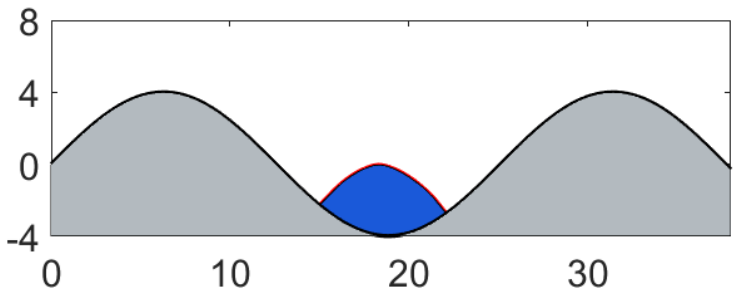}
    \\ \vspace{5pt}
    \includegraphics[width=0.83\linewidth]{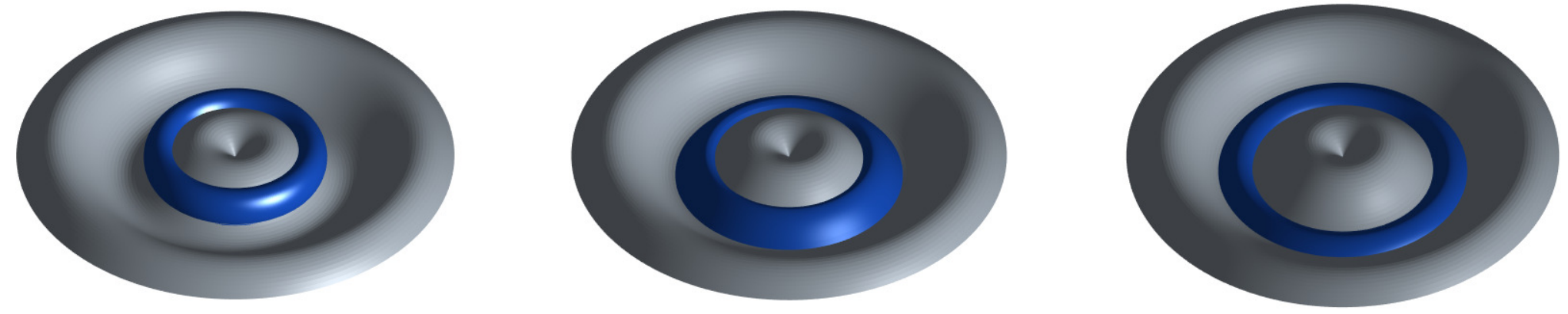}
    \caption{
    Evolution of 
    an axisymmetric thin film on an axisymmetric curved substrate generated by a sinusoidal curve,
    with \(\Delta t = 2^{-10}, h = 2^{-7}\) and \(\beta = \frac{1}{20}\): the generated curves $\Gamma^m$ at times \(t = 0, 68, 350\) (upper panel), the visualization of the corresponding axisymmetric surfaces (lower panel).}
    \label{fig:18}
\end{figure}
\begin{figure}[!ht]
    \centering
    \includegraphics[width=0.28\linewidth]{sin3_1-eps-converted-to.pdf}
    ~~~
    \includegraphics[width=0.28\linewidth]{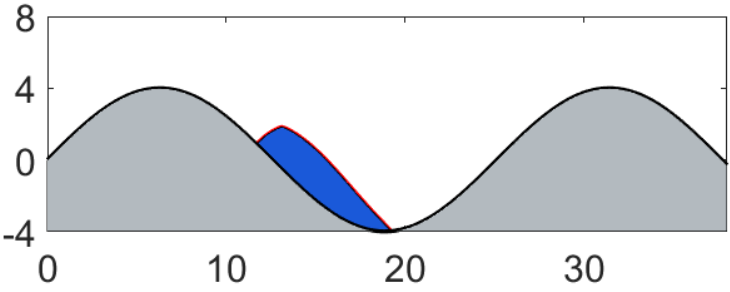}
    ~~~
    \includegraphics[width=0.28\linewidth]{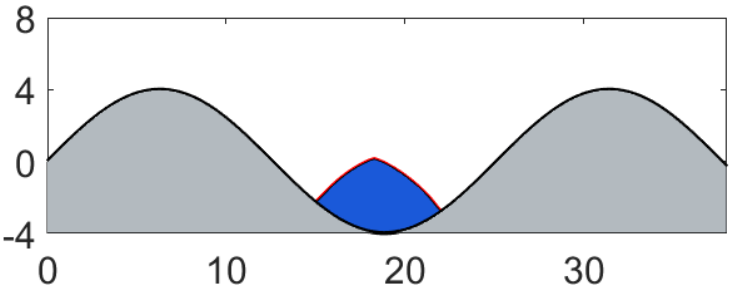}
    \\ \vspace{5pt}
    \includegraphics[width=0.83\linewidth]{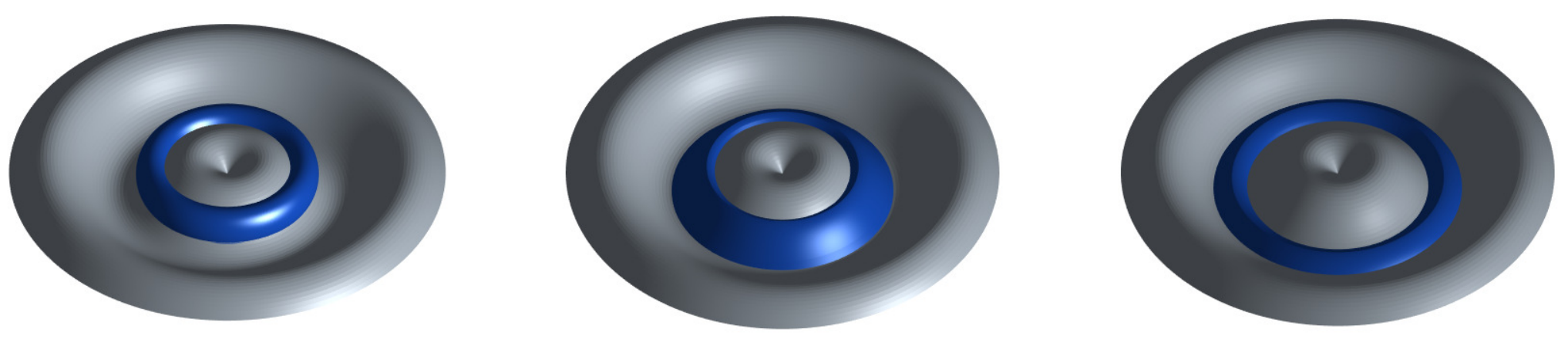}
    \caption{
    Evolution of 
    an axisymmetric thin film on an axisymmetric curved substrate generated by a sinusoidal curve,
    with \(\Delta t = 2^{-10}, h = 2^{-7}\) and \(\beta = \frac{1}{12}\): the generated curves $\Gamma^m$ at times \(t = 0, 87, 450\) (upper panel), the visualization of the corresponding axisymmetric surfaces (lower panel).
    }
    \label{fig:19}
\end{figure}

\textbf{Example 5} 
We conclude the numerical experiments by simulating
the pinch-off phenomenon and the edge contraction process. 
Specifically, we investigate the evolution of a long thin film with a thickness of $0.5$, perfectly aligned with an axisymmetric curved-surface substrate generated by the rotation of a positive-curvature curve (a circle with a radius of $30$). 
As shown in Figure \ref{fig:20}, when the film comes into contact with the substrate, we reinitialize the system and divide the thin film into two parts. These parts then continue to evolve from the initial configuration, ultimately reaching steady states and forming a smaller island and a smaller toroidal film.

Additionally, we simulate the edge retraction behavior of a semi-infinite step film across a corner. The initial substrate configuration is chosen such that the height is $1$ and the corner is smoothly connected by two arcs with a radius of $r = 0.5$. 
From Figure \ref{fig:21}, we observe that the right side of the film gradually approaches and eventually climbs over the corner. Furthermore, we compare how the contraction rate is influenced by the anisotropic strength. As shown in Figure \ref{fig:22}, a higher anisotropic strength results in a slower contraction rate.

\begin{figure}[!ht]
    \centering
    \includegraphics[width=0.23\linewidth]{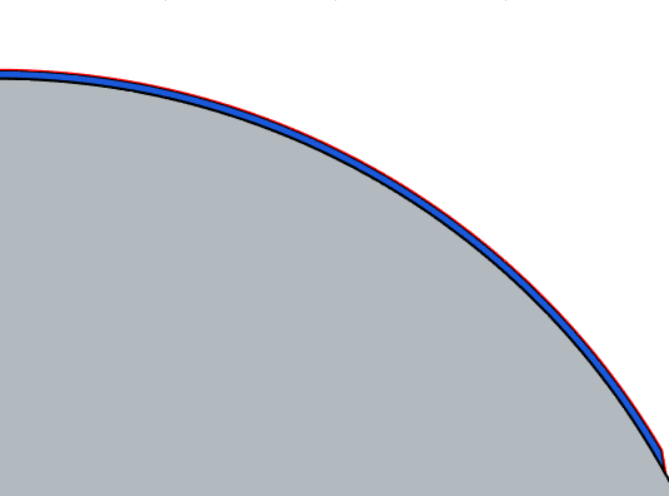}
    ~
    \includegraphics[width=0.23\linewidth]{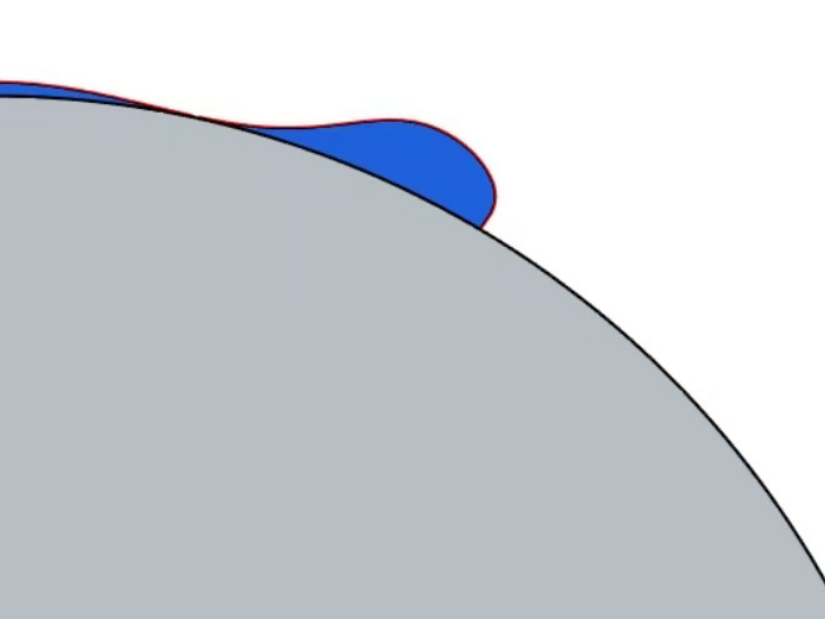}
    ~
    \includegraphics[width=0.23\linewidth]{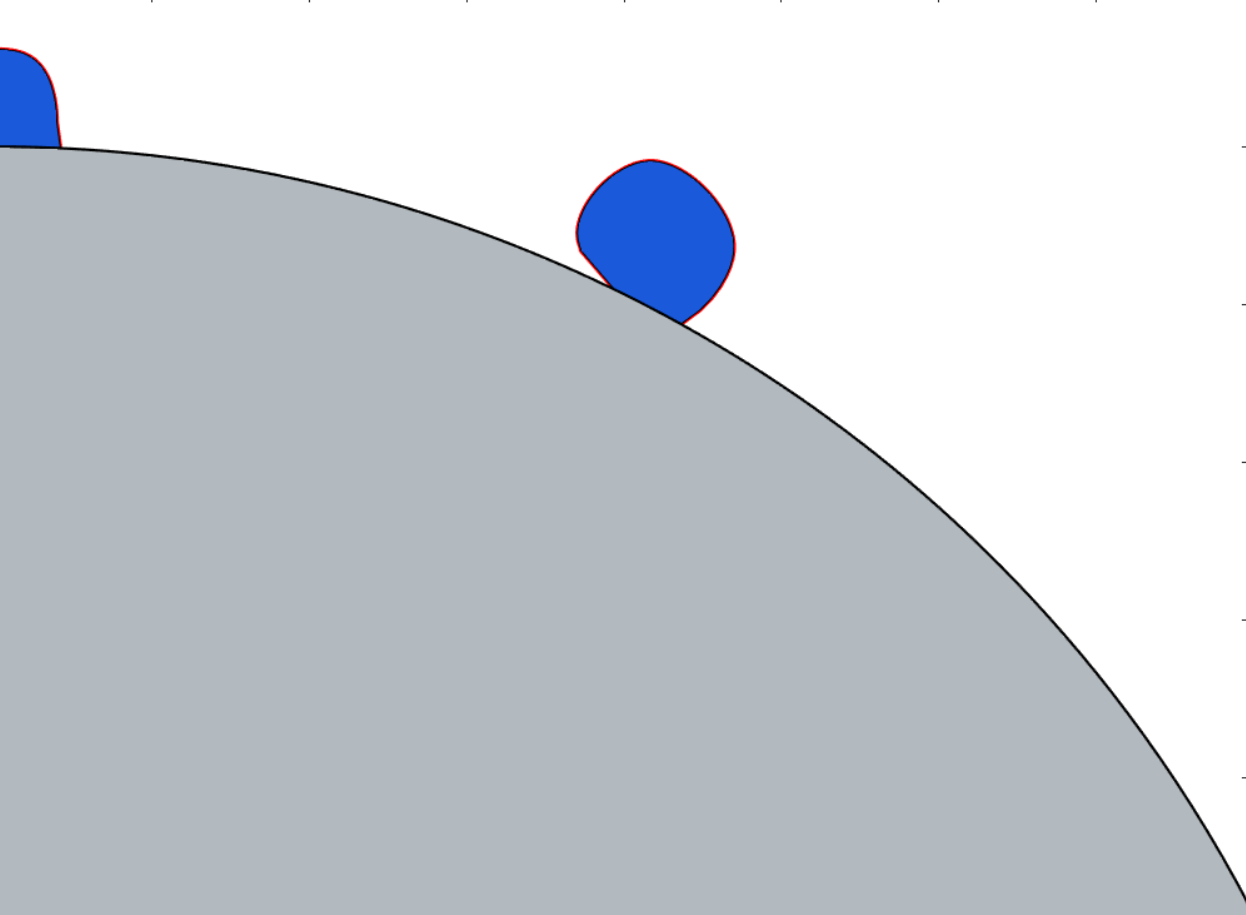}
    ~
    \includegraphics[width=0.23\linewidth]{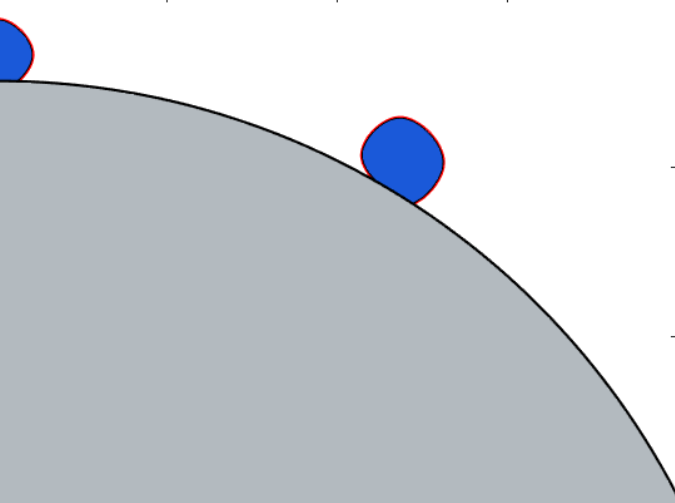}
    \\  \vspace{5pt}
    \includegraphics[width=0.23\linewidth]{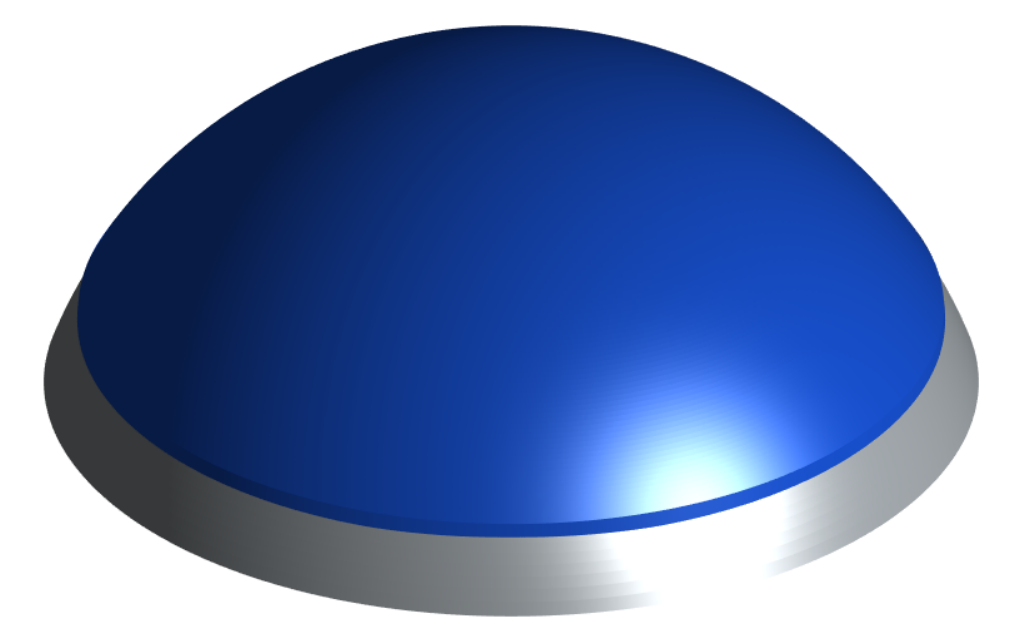}
    ~
    \includegraphics[width=0.23\linewidth]{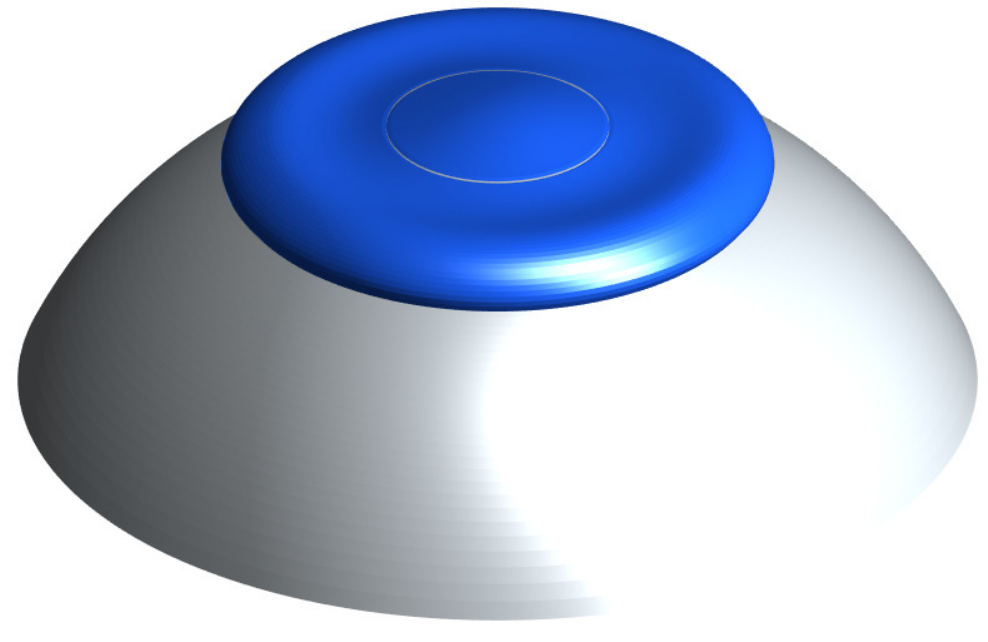}
    ~
    \includegraphics[width=0.23\linewidth]{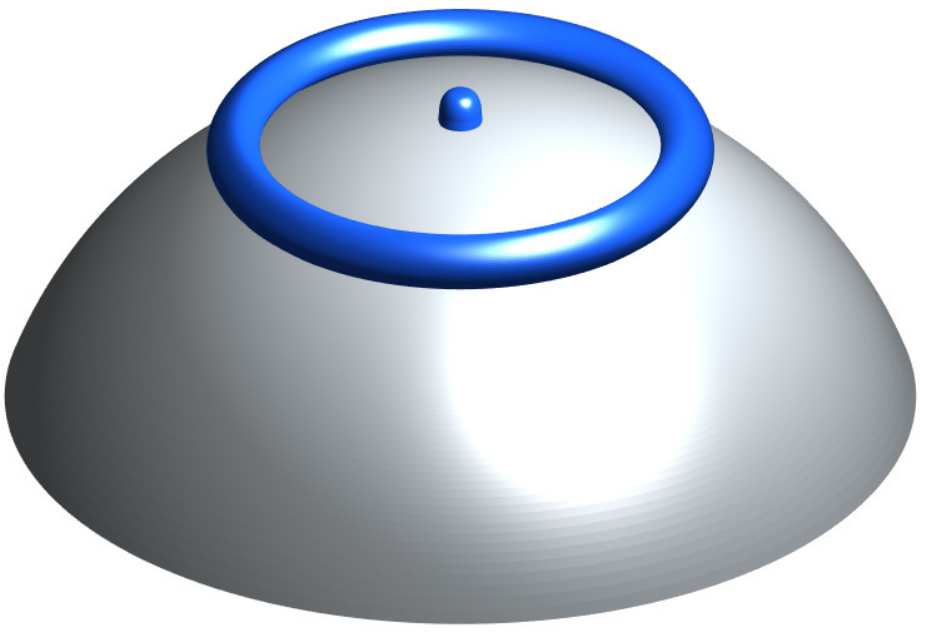}
    ~
    \includegraphics[width=0.23\linewidth]{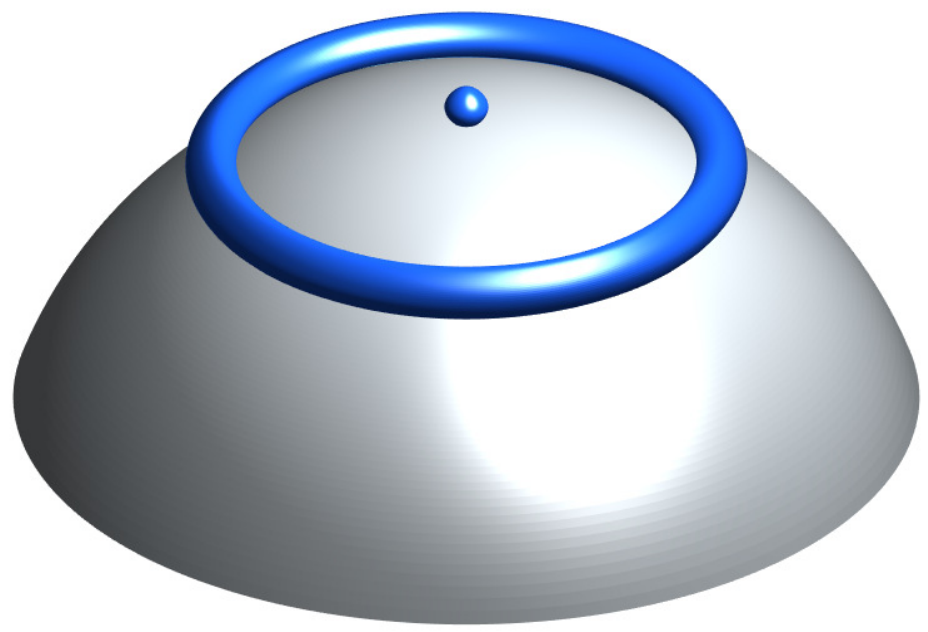}
    \caption{Visualizations of the axisymmetric thin films on hemispherical substrate at times of pinch-off with \(\Delta t = 2^{-5}, h = 2^{-7}, t = 0, 10.5, 13, 17\).}
    \label{fig:20}
\end{figure}

\begin{figure}[!ht]
    \centering
    \includegraphics[width=0.23\linewidth]{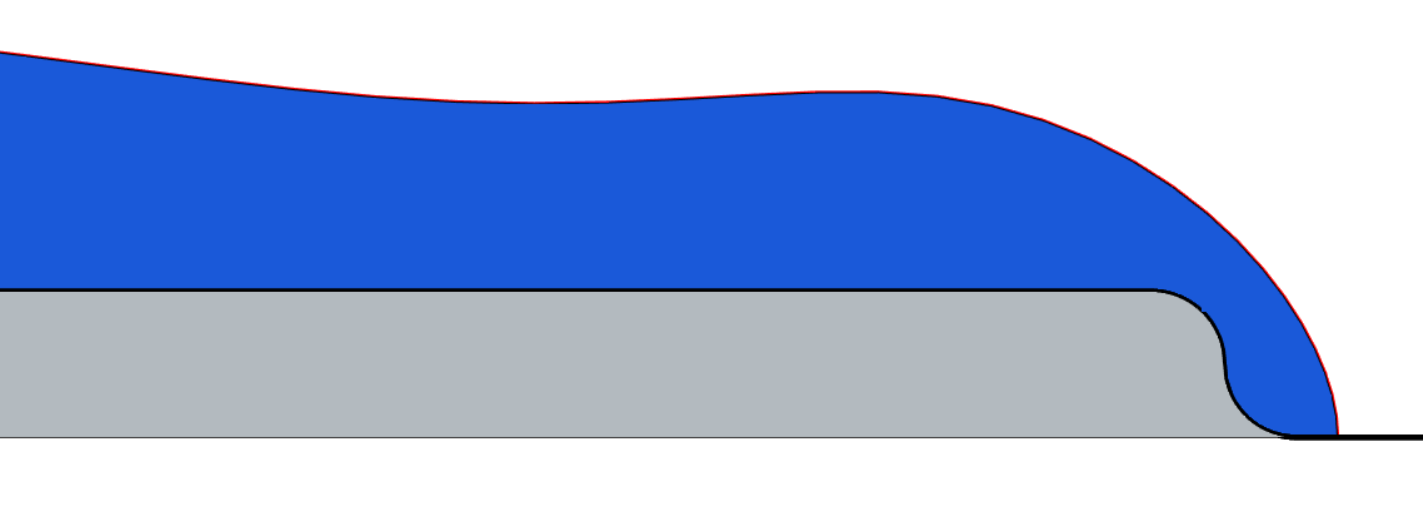}
    ~
    \includegraphics[width=0.23\linewidth]{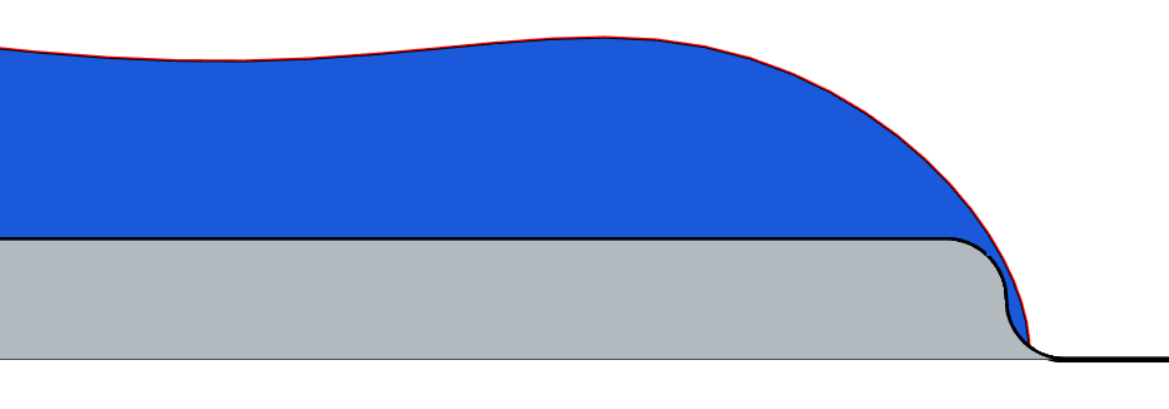}
    ~
    \includegraphics[width=0.23\linewidth]{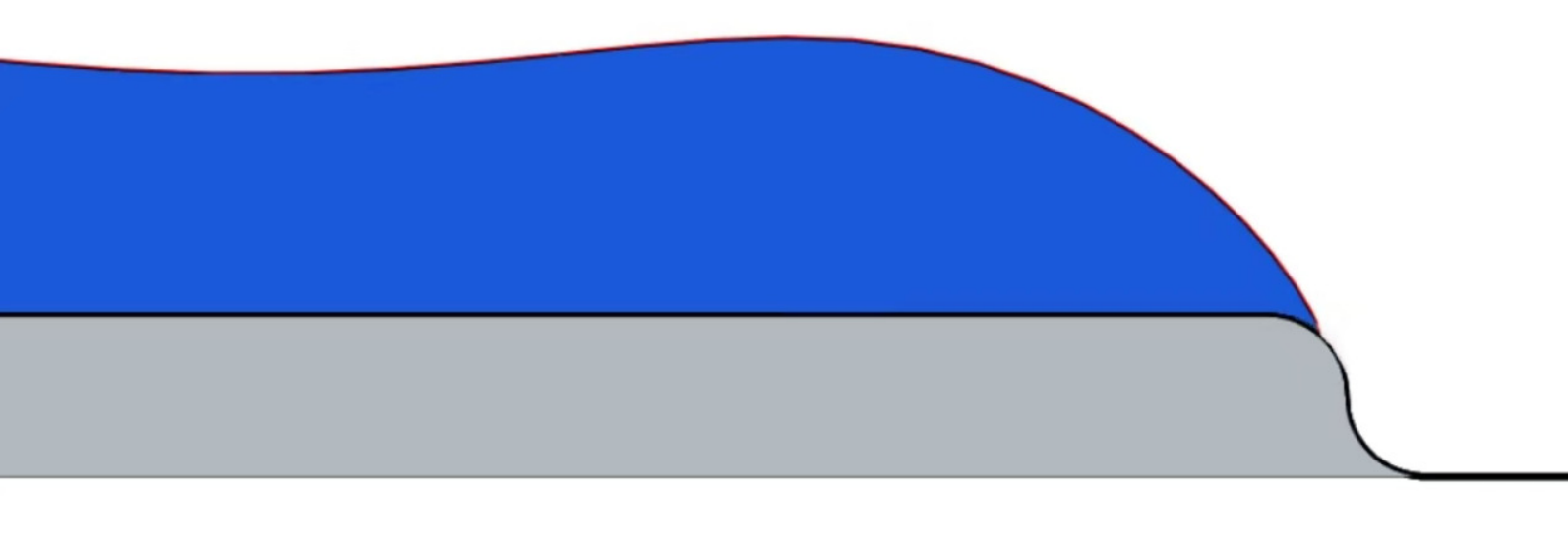}
    ~
    \includegraphics[width=0.23\linewidth]{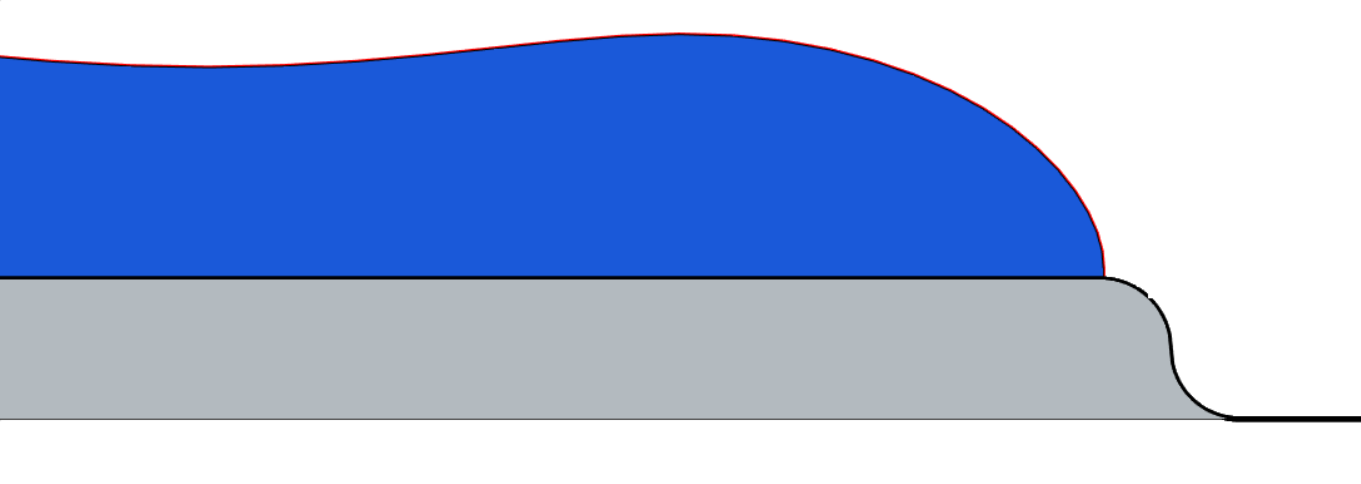}
    \\  \vspace{5pt}
    \includegraphics[width=0.23\linewidth]{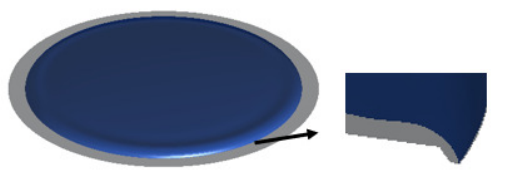}
    ~
    \includegraphics[width=0.23\linewidth]{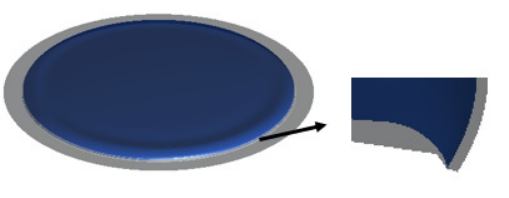}
    ~
    \includegraphics[width=0.23\linewidth]{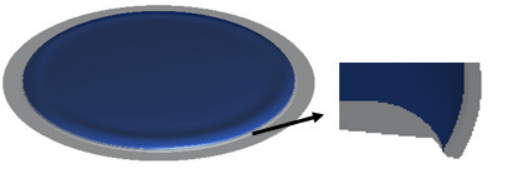}
    ~
    \includegraphics[width=0.23\linewidth]{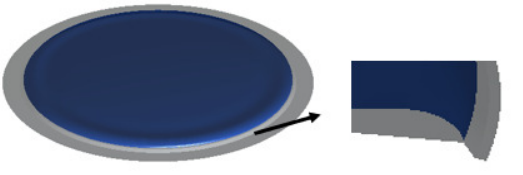}
    \caption{Snapshots in the edge retraction of a semi-infinite film moving along frustum of a cone with \(\Delta t = 2^{-6}, h = 2^{-8}, \beta = \frac{1}{20}\) at times: \(t = 5, 20, 28.5, 35\).}
    \label{fig:21}
\end{figure}

\begin{figure}[!ht]
    \centering
    \includegraphics[width=0.23\linewidth]{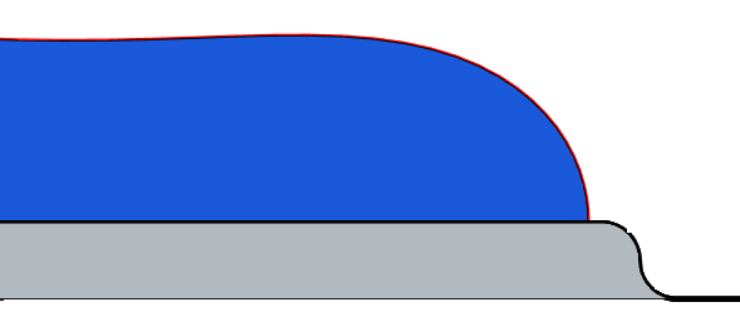}
    ~
    \includegraphics[width=0.23\linewidth]{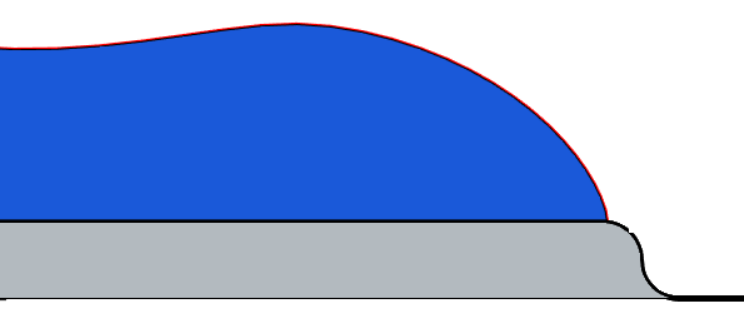}
    ~
    \includegraphics[width=0.23\linewidth]{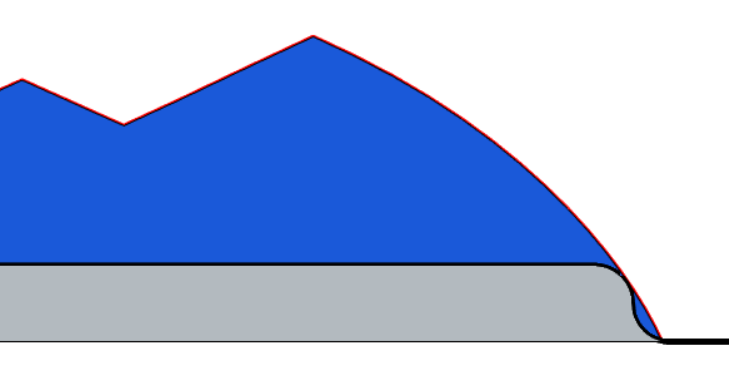}
    ~
    \includegraphics[width=0.23\linewidth]{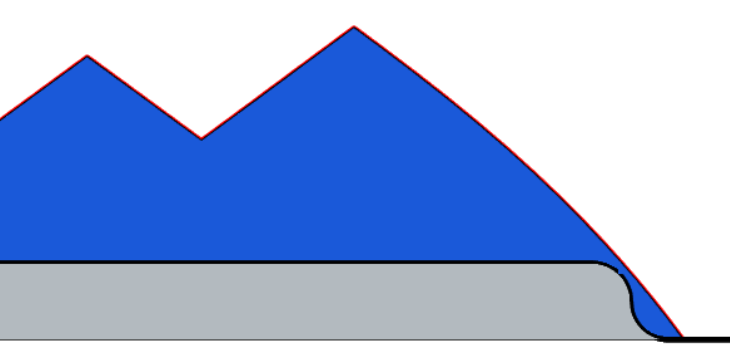}
    \caption{Comparison of isotropic and different anisotropic edge retraction phenomena at the same time \(T = 35\) with \(\Delta t = 2^{-6}, h = 2^{-8}\) and \(\beta = 0, \frac{1}{20}, \frac{1}{12}, \frac{1}{4}\) (from left to right).}
    \label{fig:22}
\end{figure}

\section{Conclusions}\label{sec6}
In this study, we focus on the SSD of thin films on axisymmetric curved-surface substrates, assuming that the film morphology is axisymmetric. Leveraging the thermodynamic variations of anisotropic surface energies, we rigorously derive a sharp-interface model governed by anisotropic surface diffusion, together with the appropriate boundary conditions. By introducing a symmetrized surface energy matrix, we obtain a novel symmetrized variational formulation. Building upon this formulation, we construct an energy-stable parametric finite element approximation by carefully discretizing the boundary terms. Furthermore, we develop an additional structure-preserving method that ensures the conservation of volume. Finally, we present a series of comprehensive numerical examples to demonstrate the convergence and structure-preserving properties of our proposed scheme. In addition, we explore several interesting phenomena, including the migration of "small" particles on curved-surface substrates, pinch-off events, and edge retraction.

\bibliographystyle{elsarticle-num}
\bibliography{thebib}

\begin{thebibliography}{10}
\expandafter\ifx\csname url\endcsname\relax
  \def\url#1{\texttt{#1}}\fi
\expandafter\ifx\csname urlprefix\endcsname\relax\def\urlprefix{URL }\fi
\expandafter\ifx\csname href\endcsname\relax
  \def\href#1#2{#2} \def\path#1{#1}\fi

\bibitem{Jiran90}
E.~Jiran, C.~Thompson, Capillary instabilities in thin films, J. Electron.
  Mater. 19~(11) (1990) 1153--1160.

\bibitem{Jiran92}
E.~Jiran, C.~Thompson, Capillary instabilities in thin, continuous films, Thin
  Solid Films. 208~(1) (1992) 23--28.

\bibitem{Ye10a}
J.~Ye, C.~V. Thompson, Mechanisms of complex morphological evolution during
  solid-state dewetting of single-crystal nickel thin films, Appl. Phys. Lett.
  97~(7) (2010) 071904.

\bibitem{Ye10b}
J.~Ye, C.~V. Thompson, Regular pattern formation through the retraction and
  pinch-off of edges during solid-state dewetting of patterned single crystal
  films, Phys. Rev. B 82~(19) (2010) 193408.

\bibitem{Ye11a}
J.~Ye, C.~V. Thompson, Anisotropic edge retraction and hole growth during
  solid-state dewetting of single crystal nickel thin films, Acta Mater. 59~(2)
  (2011) 582--589.

\bibitem{Ye11b}
J.~Ye, C.~V. Thompson, Templated solid-state dewetting to controllably produce
  complex patterns, Adv. Mater. 23~(13) (2011) 1567--1571.

\bibitem{Thompson12}
C.~V. Thompson, Solid-state dewetting of thin films, Annu. Rev. Mater. Res. 42
  (2012) 399--434.

\bibitem{Kim13}
G.~H. Kim, R.~V. Zucker, J.~Ye, W.~C. Carter, C.~V. Thompson, Quantitative
  analysis of anisotropic edge retraction by solid-state dewetting of thin
  single crystal films, J. Appl. Phys. 113~(4) (2013) 043512.

\bibitem{Zucker13}
R.~V. Zucker, G.~H. Kim, W.~C. Carter, C.~V. Thompson, A model for solid-state
  dewetting of a fully-faceted thin film, C. R. Physique 14~(7) (2013)
  564--577.

\bibitem{Zucker16}
R.~V. Zucker, G.~H. Kim, J.~Ye, W.~C. Carter, C.~V. Thompson, The mechanism of
  corner instabilities in single-crystal thin films during dewetting, J. Appl.
  Phys. 119~(12) (2016) 125306.

\bibitem{Danielson06}
D.~T. Danielson, D.~K. Sparacin, J.~Michel, L.~C. Kimerling,
  Surface-energy-driven dewetting theory of silicon-on-insulator agglomeration,
  J. Appl. Phys. 100~(8) (2006) 083507.

\bibitem{Mullins57}
W.~W. Mullins, Theory of thermal grooving, J. Appl. Phys. 28~(3) (1957)
  333--339.

\bibitem{Armelao06}
L.~Armelao, D.~Barreca, G.~Bottaro, A.~Gasparotto, S.~Gross, C.~Maragno,
  E.~Tondello, Recent trends on nanocomposites based on {Cu, Ag and Au}
  clusters: A closer look, Coord. Chem. Rev. 250~(11) (2006) 1294--1314.

\bibitem{Bollani19}
M.~Bollani, M.~Salvalaglio, A.~Benali, M.~Bouabdellaoui, M.~Naffouti,
  M.~Lodari, S.~Di~Corato, A.~Fedorov, A.~Voigt, I.~Fraj, et~al., Templated
  dewetting of single-crystal sub-millimeter-long nanowires and on-chip silicon
  circuits, Nat. Commun. 10~(1) (2019) 5632.

\bibitem{Schmidt09}
V.~Schmidt, J.~V. Wittemann, S.~Senz, U.~G{\"o}sele, Silicon nanowires: a
  review on aspects of their growth and their electrical properties, Adv.
  Mater. 21~(25-26) (2009) 2681--2702.

\bibitem{Giermann05}
A.~L. Giermann, C.~V. Thompson, Solid-state dewetting for ordered arrays of
  crystallographically oriented metal particles, Appl. Phys. Lett. 86~(12)
  (2005) 121903.

\bibitem{Kovalenko17}
O.~Kovalenko, S.~Szab{\'o}, L.~Klinger, E.~Rabkin, Solid state dewetting of
  polycrystalline {Mo} film on sapphire, Acta Mater. 139 (2017) 51--61.

\bibitem{Naffouti16}
M.~Naffouti, T.~David, A.~Benkouider, L.~Favre, A.~Delobbe, A.~Ronda,
  I.~Berbezier, M.~Abbarchi, Templated solid-state dewetting of thin silicon
  films, Small 12~(44) (2016) 6115--6123.

\bibitem{Wang11}
D.~Wang, R.~Ji, P.~Schaaf, Formation of precise 2d au particle arrays via
  thermally induced dewetting on pre-patterned substrates, Beilstein Journal of
  Nanotechnology 2~(1) (2011) 318--326.

\bibitem{dornel2006surface}
E.~Dornel, J.~Barbe, F.~De~Cr{\'e}cy, G.~Lacolle, J.~Eymery, Surface diffusion
  dewetting of thin solid films: Numerical method and application to
  {S}i/{S}io2, Phys. Rev. B 73~(11) (2006) 115427.

\bibitem{Jiang12}
W.~Jiang, W.~Bao, C.~V. Thompson, D.~J. Srolovitz, Phase field approach for
  simulating solid-state dewetting problems, Acta Mater. 60~(15) (2012)
  5578--5592.

\bibitem{Jiang18a}
W.~Jiang, Y.~Wang, D.~J. Srolovitz, W.~Bao, Solid-state dewetting on curved
  substrates, Phys. Rev. Mater. 2 (2018) 113401.

\bibitem{Jiang19a}
W.~Jiang, Q.~Zhao, Sharp-interface approach for simulating solid-state
  dewetting in two dimensions: a {Cahn-Hoffman} $\boldsymbol{\xi}$-vector
  formulation, Physical D 390 (2019) 69--83.

\bibitem{Jiang19c}
W.~Jiang, Q.~Zhao, W.~Bao, Sharp-interface model for simulating solid-state
  dewetting in three dimensions, SIAM J. Appl. Math. 80~(4) (2020) 1654--1677.

\bibitem{Srolovitz86}
D.~J. Srolovitz, S.~A. Safran, Capillary instabilities in thin films: {II.}
  {Kinetics}, J. Appl. Phys. 60~(1) (1986) 255--260.

\bibitem{Young1805}
T.~Young, An essay on the cohesion of fluids, Philos. Trans. R. Soc. London 95
  (1805) 65--87.

\bibitem{Srolovitz86a}
D.~J. Srolovitz, S.~A. Safran, Capillary instabilities in thin films: {I.}
  {Energetics}, J. Appl. Phys. 60~(1) (1986) 247--254.

\bibitem{Wong00}
H.~Wong, P.~Voorhees, M.~Miksis, S.~Davis, Periodic mass shedding of a
  retracting solid film step, Acta Mater. 48~(8) (2000) 1719--1728.

\bibitem{Dornel06}
E.~Dornel, J.~Barbe, F.~De~Cr{\'e}cy, G.~Lacolle, J.~Eymery, Surface diffusion
  dewetting of thin solid films: Numerical method and application to
  {Si/SiO$_2$}, Phys. Rev. B 73~(11) (2006) 115427.

\bibitem{Jiang16}
W.~Jiang, Y.~Wang, Q.~Zhao, D.~J. Srolovitz, W.~Bao, Solid-state dewetting and
  island morphologies in strongly anisotropic materials, Scr. Mater. 115 (2016)
  123--127.

\bibitem{ZHAO2024120407}
Q.~Zhao, W.~Jiang, Y.~Wang, D.~J. Srolovitz, T.~Qian, W.~Bao, Dynamics of small
  solid particles on substrates of arbitrary topography, Acta Materialia 281
  (2024) 120407.

\bibitem{bao2024}
W.~Bao, Y.~Li, Q.~Zhao, A structure-preserving parametric finite element method
  for solid-state dewetting on curved substrates, arXiv preprint
  arXiv:2410.00438 (2024).

\bibitem{Zhao19}
Q.~Zhao, A sharp-interface model and its numerical approximation for
  solid-state dewetting with axisymmetric geometry, J. Comput. Appl. Math. 361
  (2019) 144--156.

\bibitem{li2024structure}
M.~Li, C.~Zhou, Structure-preserving parametric finite element methods for
  simulating axisymmetric solid-state dewetting problems with anisotropic
  surface energies, arXiv preprint arXiv:2405.05844 (2024).

\bibitem{Bansch05}
E.~B{\"a}nsch, P.~Morin, R.~H. Nochetto, A finite element method for surface
  diffusion: the parametric case, J. Comput. Phys. 203~(1) (2005) 321--343.

\bibitem{bao2021structure}
W.~Bao, Q.~Zhao, A structure-preserving parametric finite element method for
  surface diffusion, SIAM J. Numer. Anal. 59~(5) (2021) 2775--2799.

\bibitem{Barrett07}
J.~W. Barrett, H.~Garcke, R.~N{\"u}rnberg, A parametric finite element method
  for fourth order geometric evolution equations, J. Comput. Phys. 222~(1)
  (2007) 441--467.

\bibitem{Barrett08JCP}
J.~W. Barrett, H.~Garcke, R.~N{\"u}rnberg, On the parametric finite element
  approximation of evolving hypersurfaces in $\mathbb{R}^3$, J. Comput. Phys.
  227~(9) (2008) 4281--4307.

\bibitem{kovacs2021convergent}
B.~Kov{\'a}cs, B.~Li, C.~Lubich, A convergent evolving finite element algorithm
  for willmore flow of closed surfaces, Numer. Math. 149~(3) (2021) 595--643.

\bibitem{Zhao20}
Q.~Zhao, W.~Jiang, W.~Bao, An energy-stable parametric finite element method
  for simulating solid-state dewetting, IMA J. Numer. Anal. 41~(3) (2021)
  2026--2055.

\bibitem{Bao17}
W.~Bao, W.~Jiang, Y.~Wang, Q.~Zhao, A parametric finite element method for
  solid-state dewetting problems with anisotropic surface energies, J. Comput.
  Phys. 330 (2017) 380--400.

\bibitem{baojcm2022}
W.~Bao, Q.~Zhao, An energy-stable parametric finite element method for
  simulating solid-state dewetting problems in three dimensions, J. Comput.
  Math. 41~(4) (2023) 771--796.

\bibitem{Barrett07Ani}
J.~W. Barrett, H.~Garcke, R.~N{\"u}rnberg, Numerical approximation of
  anisotropic geometric evolution equations in the plane, SIMA J. Numer. Anal.
  28~(2) (2007) 292--330.

\bibitem{Barrett08Ani}
J.~W. Barrett, H.~Garcke, R.~N{\"u}rnberg, A variational formulation of
  anisotropic geometric evolution equations in higher dimensions, Numer. Math.
  109~(1) (2008) 1--44.

\bibitem{Hausser07}
F.~Hausser, A.~Voigt, A discrete scheme for parametric anisotropic surface
  diffusion, J. Sci. Comput 30~(2) (2007) 223--235.

\bibitem{li2021energy}
Y.~Li, W.~Bao, An energy-stable parametric finite element method for
  anisotropic surface diffusion, J. Comput. Phys. 446 (2021) 110658.

\bibitem{Zhao19b}
Q.~Zhao, W.~Jiang, W.~Bao, A parametric finite element method for solid-state
  dewetting problems in three dimensions, SIAM J. Sci. Comput. 42~(1) (2020)
  B327--B352.

\bibitem{Barrett20}
J.~W. Barrett, H.~Garcke, R.~N{\"u}rnberg, Parametric finite element
  approximations of curvature driven interface evolutions, Handb. Numer. Anal.
  (Andrea Bonito and Ricardo H. Nochetto, eds.) 21 (2020) 275--423.

\bibitem{bao2023symmetrized}
W.~Bao, W.~Jiang, Y.~Li, A symmetrized parametric finite element method for
  anisotropic surface diffusion of closed curves, SIAM J. Numer. Anal. 61~(2)
  (2023) 617--641.

\bibitem{bao2023symmetrized1}
W.~Bao, Y.~Li, A symmetrized parametric finite element method for anisotropic
  surface diffusion in 3{D}, SIAM J. Sci. Comput. 45~(4) (2023) A1438--A1461.

\bibitem{bao2023unified}
W.~Bao, Y.~Li, A unified structure-preserving parametric finite element method
  for anisotropic surface diffusion, arXiv preprint arXiv:2401.00207 (2023).

\bibitem{li2023symmetrized}
M.~Li, Y.~Li, L.~Pei, A symmetrized parametric finite element method for
  simulating solid-state dewetting problems, Appl. Math. Model 121 (2023)
  731--750.

\bibitem{ZHANG2025113605}
Y.~Zhang, Y.~Li, W.~Ying, A stabilized parametric finite element method for
  surface diffusion with an arbitrary surface energy, J. Comput. Phys. 523
  (2025) 113605.

\bibitem{Sutton95}
A.~P. Sutton, R.~W. Balluffi, Interfaces in crystalline materials, Clarendon
  Press, 1995.

\bibitem{Cahn74}
J.~Cahn, D.~Hoffman, A vector thermodynamics for anisotropic surfaces: I.
  curved and faceted surfaces, Acta Metall. 22~(10) (1974) 1205--1214.

\end{thebibliography}
\end{document}